\documentclass[nonblindrev]{informs3-YCA}

\OneAndAHalfSpacedXI %

\usepackage{subcaption}

\usepackage{amsmath}
\usepackage{amssymb}
\usepackage{natbib}
 \bibpunct[, ]{(}{)}{,}{a}{}{,}%
 \def\bibfont{\small}%
 \def\bibsep{\smallskipamount}%

 \usepackage{algorithm}
 \usepackage[noend]{algorithmic}
 
 \usepackage{ctable}
 \usepackage{color}
 \usepackage{graphicx}
 \usepackage{xspace}
 
 \usepackage{lmodern}
 
 \usepackage{natbib}
 \usepackage{multibib}
 \newcites{app}{References for Online Appendix}
 \newcites{ecom}{References for E-Companion}

\def\leaves{\mathbf{leaves}}
\def\splits{\mathbf{splits}}
\def\leftleaves{\mathbf{left}}
\def\rightleaves{\mathbf{right}}

\def\leftchild{\mathbf{leftchild}}
\def\rightchild{\mathbf{rightchild}}
\def\root{\mathbf{root}}

\def\LP{\mathbf{LP}}
\def\RP{\mathbf{RP}}

\def\CDC{\text{CDC}}

\def\ProductMIONaive{\texttt{PD}\xspace}
\def\SplitMIOGreedy{\texttt{SG}\xspace}
\def\SplitMIONaive{\texttt{SD}\xspace}

\def\fb{\mathbf{f}}

\def\xb{\mathbf{x}}
\def\yb{\mathbf{y}}

\def\zerob{\mathbf{0}}

\def\alphab{\boldsymbol \alpha}
\def\betab{\boldsymbol \beta}

\def\thetab{\boldsymbol \theta}
\def\lambdab{\boldsymbol \lambda}

\def\Ab{\mathbf{A}}
\def\Bb{\mathbf{B}}
\def\Cb{\mathbf{C}}

\def\Ib{\mathbf{I}}
\def\Pb{\mathbf{P}}

\def\Dcal{\mathcal{D}}
\def\Ecal{\mathcal{E}}
\def\Fcal{\mathcal{F}}
\def\Ical{\mathcal{I}}

\def\Mcal{\mathcal{M}}
\def\Ncal{\mathcal{N}}

\def\Scal{\mathcal{S}}
\def\Tcal{\mathcal{T}}
\def\Ucal{\mathcal{U}}
\def\Wcal{\mathcal{W}}
\def\Zcal{\mathcal{Z}}
\def\Ccal{\mathcal{C}}
\def\Ibb{\mathbb{I}}
\def\Rbb{\mathbb{R}}

\def\Benders{\text{Benders}}
\def\BendersTwoOnly{\text{Benders-2-only}}
\def\Direct{\text{Direct}}
\def\DNC{\text{D\&C}}
\def\GT{\texttt{GT}}

\def\AODF{\text{AODF}}

\newcommand{\YCR}[1]{{\color{black} #1}}
\newcommand{\YCRR}[1]{{\color{black} #1}}
\newcommand{\YCRF}[1]{{\color{black} #1}} %

\def\SplitMIO{\textsc{SplitMIO}\xspace}
\def\ProductMIO{\textsc{ProductMIO}\xspace}
\def\LocalSearchTen{\textsc{LS10}\xspace}
\def\LocalSearch{\textsc{LS}\xspace}
\def\ROA{\textsc{ROA}\xspace}

\def\LS{\mathbf{LS}}
\def\RS{\mathbf{RS}}

\def\Halmos{$\square$}

\usepackage{natbib}
 \bibpunct[, ]{(}{)}{,}{a}{}{,}%
 \def\bibfont{\small}%
 \def\bibsep{\smallskipamount}%

\TheoremsNumberedThrough     %
\ECRepeatTheorems

\EquationsNumberedThrough    %

\begin{document}

\RUNAUTHOR{Akchen and Mi\v{s}i\'{c}}

\RUNTITLE{Assortment Optimization under the Decision Forest Model}

\TITLE{\Large Assortment Optimization under the Decision Forest Model}

\ARTICLEAUTHORS{%
	 \AUTHOR{Yi-Chun Akchen}
	\AFF{School of Management, University College London, London E14 5AB, United Kingdom.\\\EMAIL{\tt yi-chun.akchen@ucl.ac.uk}}
	\AUTHOR{Velibor V. Mi\v{s}i\'{c}}
	\AFF{UCLA Anderson School of Management, University of California, Los Angeles, California 90095, United States, \EMAIL{\tt velibor.misic@anderson.ucla.edu}} %
}
\ABSTRACT{{\it Problem definition}: 
We study the problem of finding the optimal assortment that maximizes expected revenue under the decision forest model, a recently proposed nonparametric choice model that is capable of representing any discrete choice model and in particular, can be used to represent non-rational customer behavior. 
This problem is of practical importance because it allows a firm to tailor its product offerings to profitably exploit deviations from rational customer behavior, but at the same time is challenging due to the extremely general nature of the decision forest model. 
{\it Methodology/Results}: 
We approach this problem from a mixed-integer optimization perspective and present two different formulations. We further propose a methodology for solving the two formulations at a large-scale based on Benders decomposition, and show that the Benders subproblem can be solved efficiently by primal-dual greedy algorithms when the master solution is fractional for one of the formulations, and in closed form when the master solution is binary for both formulations. Using synthetically generated instances, we demonstrate the practical tractability of our formulations and our Benders decomposition approach, and their edge over heuristic approaches. 
{\it Managerial implications:} In a case study based on a real-world transaction data, we demonstrate that our proposed approach can factor the behavioral anomalies observed in consumer choice into assortment decision and create higher revenue.
}%


\maketitle

\section{Introduction}
\label{sec:intro}

Assortment optimization is a basic operational problem faced by many firms. In its simplest form, the problem can be posed as follows. A firm has a set of products that it can offer, and a set of customers who have preferences over those products; what is the set of products the firm should offer so as to maximize the revenue that results when the customers choose from these products? 

While assortment optimization and the related problem of product line design have been studied extensively under a wide range of choice models, the majority of research in this area focuses on rational choice models, specifically those that follow the random utility maximization (RUM) assumption. A significant body of research in the behavioral sciences shows that customers behave in ways that deviate significantly from predictions made by RUM models. In addition, there is a substantial number of empirical examples of firms that make assortment decisions in ways that directly exploit customer irrationality. For example, the paper of \cite{kivetz2004extending} provides an example of an assortment of document preparation systems from Xerox that are structured around the decoy effect, and an example of an assortment of servers from IBM that are structured around the compromise effect. 

A recent paper by \cite{chen2019decision} proposed a new choice model called the \emph{decision forest model} for capturing customer irrationalities. This model involves representing the customer population as a probability distribution over binary trees, with each tree representing the decision process of one customer type. In a key result of the paper, the authors showed that this model is universal: \emph{every discrete choice model is representable as a decision forest model}. While the paper of \cite{chen2019decision} alludes to the downstream assortment optimization problem, it is entirely focused on model representation and prediction: it does not provide any answer to how one can select an optimal assortment with respect to a decision forest model. 

In the present paper, we present a methodology for assortment optimization under the decision forest model, based on mixed-integer optimization. Our approach allows the firm to obtain assortments that are either provably optimal or within a desired optimality gap for a given decision forest model. \YCRR{Moreover, in general, business rules (e.g., capacity constraints) can be readily incorporated by adding corresponding linear constraints to the MIO formulation.} Most importantly, given the universality property of the decision forest model, our optimization approach allows a firm to optimally tailor its assortment to any kind of predictable irrationality in the firm's customer population. 

We make the following specific contributions:
\label{page:contributions_in_intro}

\begin{enumerate}
\item We formally define the assortment optimization problem under the decision forest model and analyze its computational hardness. Specifically, \YCRR{when the tree depth is bounded by a constant $k$, we prove that it is NP-hard to approximate the problem within a factor that grows exponentially in $k$}. \YCRR{This result complements the findings of \citet{chen2019decision}, who show that tree depth governs the representational power of the model: with sufficiently deep trees, the decision forest can represent any discrete choice model. However, they also note that deeper trees increase the risk of overfitting and recommend selecting tree depth via cross-validation. From an operations perspective, our hardness result further underscores the importance of shallow trees: while deeper trees enhance representation power, they also render the downstream assortment optimization problem inapproximable, highlighting the need to balance model expressiveness with tractability.} Given the computational challenge of the assortment optimization problem under the decision forest model, we then adopt an integer programming approach and present two formulations: \SplitMIO and \ProductMIO. 
\item We propose a Benders decomposition approach for solving the two formulations at scale. \YCRR{We apply the method to both the linearly relaxed and the integer programs so that the cuts generated from the relaxation provide strong upper bounds and warm-start the integer procedure}. For \SplitMIO, we show that Benders cuts for the linear relaxation can be obtained via a greedy algorithm that simultaneously solves the primal and dual subproblems. We also give a simple example demonstrating that this greedy construction fails for the primal subproblem of \ProductMIO. For both \SplitMIO\ and \ProductMIO, we further derive closed-form Benders cuts for the integer subproblems. \YCRR{Overall, our decomposition generalizes the scheme developed for the ranking-based model \citep{bertsimas2019exact} by introducing a new dual-variable update procedure that exploits the tree topology and is tailored to the subproblem structures of our formulations.}
\item We present three numerical experiments using both synthetic and real-world data. The first, based on synthetic instances, evaluates the formulation strength of \SplitMIO and \ProductMIO and demonstrates the scalability of the Benders decomposition approach for \SplitMIO on problems with up to 3,000 products, 500 trees, and 512 leaves per tree. The second uses the IRI Dataset \citep{bronnenberg2008database}, a large-scale grocery transaction dataset, to show how the decision forest model can capture a behavioral anomaly (\emph{choice overload}) and recommend assortments that differ from those implied by rational choice models. \YCRR{The third is a semi-synthetic, end-to-end (data-to-decision) experiment: we generate data from a ranking-based model derived from the Sushi Dataset \citep{kamishima2003nantonac}, augmented with a choice overload effect \citep{iyengar2000choice}. We show that the decision forest model delivers robust and consistent performance even as the overload effect intensifies, whereas the multinomial logit (MNL) and ranking-based models deteriorate noticeably, even under mild choice overload. Taken together, the second and third experiments highlight that the decision forest model can effectively incorporate behavioral anomalies observed in consumer choice into the assortment decision process.}
\end{enumerate}

We organize the paper as follows. Section~\ref{sec:literature_review} reviews related literature. Section~\ref{sec:model} defines the assortment problem, characterizes its inapproximability, and presents the two MIO formulations. Section~\ref{sec:benders} proposes a Benders decomposition approach for these formulations. Sections~\ref{sec:synthetic} and \ref{sec:IRI} present numerical results involving both synthetic and real-world data. All proofs are relegated to the appendix.

\section{Literature review}
\label{sec:literature_review}

Assortment optimization has been extensively studied in the operations management community; we refer readers to \cite{gallego2019assortment} for a recent review of the literature. The literature on assortment optimization has focused on developing approaches for finding the optimal assortment under many different rational choice models, such as the MNL model \citep{talluri2004revenue,sumida2020revenue}, the latent class MNL model \citep{bront2009column,mendez2014branch}, the Markov chain choice model \citep{feldman2017revenue,desir2020constrained} and the ranking-based model \citep{aouad2015assortment,aouad2018approximability,feldman2019assortment}. 

In addition to the assortment optimization literature, our paper is also related to the literature on product line design found in the marketing community. While assortment optimization is more often focused on the tactical decision of selecting which existing products to offer, where the products are ones that have been sold in the past and the choice model comes from transaction data involving those products, the product line design problem involves selecting which new products to offer, where the products are candidate products (i.e., they have not been offered before) and the choice model comes from conjoint survey data, where customers are asked to rate or choose between hypothetical products. Research in this area has considered different approaches to solve the problem under the ranking-based/first-choice model \citep{mcbride1988integer,belloni2008optimizing,bertsimas2019exact} and the MNL model \citep{chen2000mathematical,schon2010optimal}. %

Our paper is related to \cite{bertsimas2019exact}, which presents integer optimization formulations of the product line design problem when the choice model is a ranking-based model. As we will see later, our formulation \SplitMIO can be viewed a generalization of the formulation \cite{bertsimas2019exact}, to the decision forest model. In addition, the paper of \cite{bertsimas2019exact} develops a specialized Benders decomposition approach for its formulation, which uses the fact that one can solve the subproblem associated with each customer type by applying a greedy algorithm. We will show in Section~\ref{sec:benders} that this same property generalizes to the \SplitMIO formulation, leading to a tailored Benders decomposition algorithm for solving \SplitMIO at scale. 

Beyond these specific connections, the majority of the literature on assortment optimization and product line design considers rational choice models, whereas our paper contributes a methodology for non-rational assortment optimization. Fewer papers have focused on choice modeling for non-rational customer behavior; besides the decision forest model, other models include the generalized attraction model (GAM; 
\citealt{gallego2015general}), the generalized stochastic preference model \citep{berbeglia2018generalized} and the generalized Luce model \citep{echenique2019general}. An even smaller set of papers has considered assortment optimization under non-rational choice models, which we now review. The paper of \cite{flores2017assortment} considers assortment optimization under the two-stage Luce model, and develops a polynomial time algorithm for solving the unconstrained assortment optimization problem. The paper of \cite{rooderkerk2011incorporating} considers a context-dependent utility model where the utility of a product can depend on other products that are offered and that can capture compromise, attraction and similarity effects; the paper empirically demonstrates how incorporating context effects leads to a predicted increase of 5.4\% in expected profit.

Relative to these papers, our paper differs in that it considers the decision forest model. As noted earlier, the decision forest model can represent any type of choice behavior, and as such, an assortment optimization methodology based on such a model is attractive in terms of allowing a firm to take the next step from a high-fidelity model to a decision. In addition, our methodology is built on mixed-integer optimization. This is advantageous because it allows a firm to leverage continuing improvements in integer optimization solvers such as Gurobi \citep{gurobi} and CPLEX \citep{cplex}, as well as continuing improvements in computer hardware. At the same time, integer optimization allows firms to accommodate business requirements using linear constraints, enhancing the practical applicability of the approach. Lastly, integer optimization also allows one to take advantage of well-studied large-scale solution methods. One such method that we focus on in this paper is Benders decomposition, which has seen an impressive resurgence in recent years for delivering state-of-the-art performance on large-scale problems such as hub location \citep{contreras2011benders}, facility location \citep{fischetti2017redesigning} and set covering \citep{cordeau2019benders}; see also \cite{rahmaniani2017benders} for a review of the recent literature. %

In addition to the assortment optimization and product line design literatures, our formulations also \YCRR{relate to} others that have been proposed in the broader optimization literature. The formulation \SplitMIO that we will present later can be viewed as a special case of the mixed-integer optimization formulation of \cite{misic2019optimization} for optimizing the predicted value of a tree ensemble model, such as a random forest or a boosted tree model. The other formulation, \ProductMIO, also has a connection to the integer optimization literature on formulating disjunctive constraints through independent branching schemes \citep{vielma2010mixed,vielma2011modeling,huchette2019combinatorial}; we also discuss this connection in more detail in Section~\ref{subsec:model_ProductMIO}.

\section{Optimization model}
\label{sec:model}

In this section, we define the decision forest assortment optimization problem (Section~\ref{subsec:model_problemdefinition}), analyze its computational complexity (Section~\ref{subsec:inapproximability}), and present two mixed-integer formulations: \SplitMIO (Section~\ref{subsec:model_SplitMIO}) and \ProductMIO (Section~\ref{subsec:model_ProductMIO}). We then compare these formulations and highlight their respective advantages across different regimes (Section~\ref{subsec:SplitMIO_vs_ProductMIO}).

\subsection{Problem definition}
\label{subsec:model_problemdefinition}

We first briefly review the decision forest model of \cite{chen2019decision}, and then formally state the assortment optimization problem. We assume that there are $n$ products, indexed from 1 to $n$, and let $\Ncal = \{1,\dots,n\}$ denote the set of all products. An assortment $S$ corresponds to a subset of $\Ncal$. When offered $S$, a customer may choose to purchase one of the products in $S$, or to not purchase anything from $S$ at all; we use the index $0$ to denote the latter possibility, which we will also refer to as the no-purchase option or the outside option. 

The basic building block of the decision forest model is a purchase decision tree. Each tree represents the purchasing behavior of one type of customer when presented with an assortment $S$. Specifically, for an assortment $S$, the customer behaves as follows: the customer starts at the root of the tree. The customer checks whether the product corresponding to the root node is contained in $S$; if it is, she proceeds to the left child, and if not, she proceeds to the right child. She then checks again with the product at the new node, and the process repeats, until the customer reaches a leaf; the option that is at the leaf represents the choice of that customer. Figure~\ref{figure:decision_tree_example} (left) shows an example of a purchase decision tree being used to map an assortment to a purchase decision.

\begin{figure}
	\begin{center}
		\begin{subfigure}[t]{0.3\textwidth}
			\centering
			\includegraphics[width=1\textwidth]{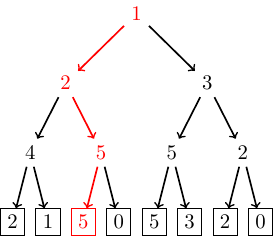}
		\end{subfigure}
		\qquad\qquad
		\begin{subfigure}[t]{0.3\textwidth}
			\centering
			\includegraphics{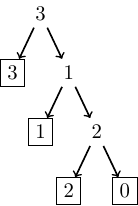}
		\end{subfigure}
	\end{center}
	\caption{(left) Example of a purchase decision tree for $n = 5$ products. Leaf nodes are enclosed in squares, while split nodes are not enclosed. The number on each node corresponds either to $v(t,s)$ for splits, or $c(t,\ell)$ for leaves. \YCRR{The path highlighted in red shows how a customer represented by this tree makes a purchase decision when facing the assortment $S = \{ 1, 3, 4, 5 \}$.} For this assortment, the customer's decision is to purchase product 5; (right) Example of a purchase decision tree that represents a preference ranking $\{ 3 \succ 1 \succ 2 \succ 0  \}$. \label{figure:decision_tree_example}}
\end{figure}

Formally, a purchase decision tree is a directed binary tree, with each leaf node corresponding to an option in $\Ncal \cup \{0\}$, and each non-leaf (or \emph{split}) node corresponding to a product in $\Ncal$. We use $\splits(t)$ to denote the set of split nodes of tree $t$, and $\leaves(t)$ to denote the set of leaf nodes. We use $c(t,\ell)$ to denote the purchase decision of leaf $\ell$ of tree $t$, i.e., the option chosen by tree $t$ if the assortment is mapped to leaf $\ell$. We use $v(t,s)$ to denote the product that is checked at split node $s$ in tree $t$. We make the following assumption about our purchase decision trees.
\begin{assumption}
Let $t$ be a purchase decision tree. For any two split nodes $s$ and $s'$ of $t$ such that $s'$ is a descendant of $s$, $v(t,s) \neq v(t,s')$. \label{assumption:Requirement_3}
\end{assumption}
This assumption states that once a product appears on a split $s$, it cannot appear on any subsequent split $s'$ that is reached by proceeding to the left or right child of $s$; in other words, each product in $\Ncal$ appears at most once along the path from the root node to a leaf node, for every leaf node. As discussed in \cite{chen2019decision}, this assumption is not restrictive, as any tree for which this assumption is violated has a set of splits and leaves that are redundant and unreachable, and the tree can be modified to obtain an equivalent tree that satisfies the assumption.

The decision forest model assumes that the customer population is represented by a collection of trees or a \emph{forest} $F$. Each tree $t \in F$ corresponds to a different customer type. We use $\lambda_t$ to denote the probability associated with customer type/tree $t$, and $\lambdab = (\lambda_t)_{t \in F}$ to denote the probability distribution over the forest $F$. For each tree $t$, we use $\eta(t,S)$ to denote the choice that a customer type following tree $t$ will make when given the assortment $S$. For a given assortment $S \subseteq \Ncal$ and a given choice $j \in S \cup \{0\}$, we use $\Pb^{(F,\lambdab)}(j \mid S)$ to denote the choice probability, i.e., the probability of a random customer choosing $j$ when offered the assortment $S$. It is defined as
\begin{equation}
\Pb^{(F, \lambdab)}(j \mid S) = \sum_{t \in F} \lambda_t \cdot \Ibb\{  \eta(t,S)  = j  \}. 
\end{equation}

We now define the assortment optimization problem. We use $\rho_i$ to denote the marginal revenue of product $i$; for convenience, we use $\rho_0 = 0$ to denote the revenue of the no-purchase option. The assortment optimization problem that we wish to solve is 
\begin{equation}
\underset{S \subseteq \Ncal}{\text{maximize}} \ \sum_{i \in S} \YCR{\rho_i} \cdot \Pb^{(F, \lambdab)}(i \mid S). \label{prob:assortment_opt_abstract}
\end{equation}
This is a challenging problem due to the generality of the choice model $\Pb^{(F, \lambdab)}(\cdot \mid \cdot)$. Notably, it subsumes the assortment optimization problem under the ranking-based model, which is known to be NP-hard \citep{aouad2018approximability,feldman2019assortment}. This conclusion arises from the observation that, in the special case where each tree $t \in F$ branches exclusively to the right -- meaning the left child of every split node in tree $t$ is always a leaf node (as illustrated in the right panel of Figure~\ref{figure:decision_tree_example}) -- the forest $F$ simplifies into a collection of preference rankings. This observation directly leads to the following corollary.
\begin{corollary}
	The assortment optimization problem~\eqref{prob:assortment_opt_abstract} is NP-Hard. \label{corollary:AODF_NPHard}
\end{corollary}

\subsection{Inapproximability Results}
\label{subsec:inapproximability}

We strengthen the complexity results for the assortment optimization problem~\eqref{prob:assortment_opt_abstract} by providing an inapproximability gap based on the depth of the decision trees. First, we define the depth of a tree $t$ as the maximal distance from the root of tree $t$ to any leaf node $\ell \in \leaves(t)$. For example, both trees in Figure~\ref{figure:decision_tree_example} are of depth three. Second, we let $\AODF(k)$ denote problem~\eqref{prob:assortment_opt_abstract} where each tree $t \in F$ can have depth at most $k$. With these definitions, we have the following result.
\begin{theorem}
	\label{theorem:inapproximable_exponential_gap}
	If P $\neq$ NP, then for every constant $k \geq 7$ and $\epsilon >0$, there is no polynomial-time $\left( 2^{k - 2 \lceil  \sqrt{k}  \rceil } / (2ek) - \epsilon \right)$-approximation algorithm for \AODF(k+1), even when each tree has equal probability weight in the model, i.e., $\lambda_t = 1/ |F|$. 
\end{theorem}
We prove Theorem~\ref{theorem:inapproximable_exponential_gap} by constructing an approximation-preserving reduction between $\AODF(k+1)$ and the Max 1-in-E$k$SAT problem \citep{guruswami2005complexity}, for which the same inapproximability result holds. The Max 1-in-E$k$SAT problem is a variant of the classic SAT problem \citep{garey2002computers}, where each clause is satisfied if and only if exactly one literal in the clause is true, and each clause consists of exactly $k$ literals. In our reduction, we construct a decision tree with depth $k+1$ and $O(k^2)$ leaf nodes for each clause in the Max 1-in-E$k$SAT problem.
	
Theorem~\ref{theorem:inapproximable_exponential_gap} implies that the inapproximability gap of the assortment problem~\eqref{prob:assortment_opt_abstract} grows \emph{at least} exponentially with respect to the tree depth, since the ratio in Theorem~\ref{theorem:inapproximable_exponential_gap} follows $2^{k - 2 \lceil  \sqrt{k}  \rceil } / (2ek) = 2^{  k - 2 \lceil \sqrt{k} \rceil - \log k - 1 - \log_2 e  } = 2^{k - o(k)}$. The following proposition further shows that this exponential dependence is asymptotically tight in $k$ (i.e., when $k$ is large), as there exists an approximation algorithm achieving approximation ratio $2^k$.
\begin{proposition}
	\label{prop:randomized_approx_algo}
	\AODF(k) can be approximated within factor $2^k$.
\end{proposition}

Before we proceed to solve the challenging assortment optimization problem~\eqref{prob:assortment_opt_abstract}, we pause to make two remarks on Theorem~\ref{theorem:inapproximable_exponential_gap}. First, recall that any ranking of length $k$ can be represented by a decision tree that is constrained to branch to the right and of depth $k$. \cite{aouad2018approximability} shows that the assortment optimization under the ranking-based model is NP-hard to \YCRR{approximate} within $O(k^{1-\epsilon})$ factor provided that each ranking in the model has a length of at most $k$ (see Section 3.2 of \cite{aouad2018approximability}). Theorem~\ref{theorem:inapproximable_exponential_gap} shows that the assortment problem becomes much harder when optimizing over decision trees instead of rankings. Specifically, the inapproximability factor grows from a linear rate $O(k)$ to an exponential rate $O(2^k)$ as a function of $k$.

Second, together with \cite{chen2019decision}, Theorem~\ref{theorem:inapproximable_exponential_gap} highlights a fundamental tradeoff between the representation power and tractability in the decision forest model. In particular, \cite{chen2019decision} shows that as the decision trees grow deeper, the model becomes more expressive in capturing choice behavior; at its maximum depth, $n$, the decision forest model can represent any discrete choice model (see Theorems 1 and 2 in \cite{chen2019decision}). At the same time, our Theorem~\ref{theorem:inapproximable_exponential_gap} establishes that deeper trees also lead to increased computational complexity for assortment optimization, with inapproximability growing exponentially in tree depth. \YCRR{This tradeoff highlights the need to manage tree depth so that the decision forest model remains expressive enough to characterize consumer choice yet tractable for downstream operational tasks.}

Given the inapproximability of problem~\eqref{prob:assortment_opt_abstract}, we consider a mixed-integer optimization (MIO) approach to solve the problem and present two formulations in the next two subsections.

\subsection{Formulation 1: \SplitMIO}
\label{subsec:model_SplitMIO}

We now present our first formulation of the assortment optimization problem~\eqref{prob:assortment_opt_abstract} as a mixed-integer optimization (MIO) problem. To formulate the problem, we introduce some additional notation. For notational convenience we let $r_{t,\ell} = \rho_{c(t,\ell)}$ be the revenue of the purchase option of leaf $\ell$ of tree $t$. We let $\leftleaves(s)$ denote the set of leaf nodes that are to the left of split $s$ (i.e., can only be reached by taking the left branch at split $s$), and similarly, we let $\rightleaves(s)$ denote the leaf nodes that are to the right of $s$. We introduce two sets of decision variables. For each $i \in \Ncal$, we let $x_i$ be a binary decision variable that is 1 if product $i$ is included in the assortment, and 0 otherwise. For each tree $t \in F$ and leaf $\ell \in \leaves(t)$, we let $y_{t,\ell}$ be a binary decision variable that is 1 if the assortment encoded by $\xb$ is mapped to leaf $\ell$ of tree $t$, and 0 otherwise.

With these definitions, our first formulation, \SplitMIO, is given below.
\begin{subequations}
\begin{alignat}{2}
\SplitMIO: \quad & \underset{\xb, \yb}{\text{maximize}} &\quad & \sum_{t \in F} \lambda_t \cdot \left[  \sum_{\ell \in \leaves(t)} r_{t,\ell} y_{t,\ell} \right] \\
& \text{subject to} & & \sum_{\ell \in \leaves(t)} y_{t,\ell} = 1, \quad \forall \ t \in F, \label{prob:SplitMIO_y_unitsum} \\
& & & \sum_{\ell \in \leftleaves(s)} y_{t,\ell} \leq x_{v(t,s)}, \quad \forall \ t \in F, \ s \in \splits(t), \label{prob:SplitMIO_left} \\
& & & \sum_{\ell \in \rightleaves(s)} y_{t,\ell} \leq 1 - x_{v(t,s)}, \quad \forall \ t \in F, \ s \in \splits(t), \label{prob:SplitMIO_right} \\
& & & x_i \in \{0,1\}, \quad \forall \ i \in \Ncal, \label{prob:SplitMIO_x_binary}\\
& & & y_{t,\ell} \geq 0, \quad \forall \ t \in F, \ \ell \in \leaves(t). \label{prob:SplitMIO_y_nonnegative}
\end{alignat}
\label{prob:SplitMIO}%
\end{subequations}
In order of appearance, the constraints in this formulation have the following meaning. Constraint~\eqref{prob:SplitMIO_y_unitsum} requires that for each customer type $t$, the assortment encoded by $\xb$ is mapped to exactly one leaf. Constraint~\eqref{prob:SplitMIO_left} imposes that for a split $s$ in tree $t$, if product $v(t,s)$ is not in the assortment, then the assortment cannot be mapped to any of the leaves that are to the left of split $s$ in tree $t$. Constraint~\eqref{prob:SplitMIO_right} is the symmetric case of constraint~\eqref{prob:SplitMIO_left} for the right-hand subtree of split $s$ of tree $t$. The last two constraints require that $\xb$ is binary and $\yb$ is nonnegative. Note that it is not necessary to require $\yb$ to be binary. \YCRR{Whenever $\xb$ is binary, there exists an integer-valued solution for $\yb$.} Finally, the objective function corresponds to the expected per-customer revenue of the assortment.

The \SplitMIO formulation is related to two existing MIO formulations in the literature. First, it can be viewed as a specialized case of the MIO formulation in \cite{misic2019optimization}. In that paper, the author develops a formulation for tree ensemble optimization, i.e., the problem of setting the independent variables in a tree ensemble model (e.g., a random forest or a gradient boosted tree model) to maximize the value predicted by that ensemble. Later, in Section~\ref{sec:benders}, we introduce a specialized algorithm that leverages a primal-dual greedy approach to efficiently solve the linear relaxation of \SplitMIO. This method, which enhances the tractability of the \SplitMIO formulation, represents our contribution to the literature and was not considered in \cite{misic2019optimization}. 
Second, the \SplitMIO formulation also relates to the MIO formulation for the product line design problem under the ranking-based model \citep{bertsimas2019exact}. Recall that the ranking-based model can be regarded as a special case of the decision forest model (illustrated in the right panel of Figure~\ref{figure:decision_tree_example}). In this special case, the formulation~\eqref{prob:SplitMIO} can be reduced to the MIO formulation for product line design under ranking-based models presented in \cite{bertsimas2019exact}.

\subsection{Formulation 2: \ProductMIO}
\label{subsec:model_ProductMIO}

The second formulation of problem~\eqref{prob:assortment_opt_abstract} that we will present is motivated by the behavior of \SplitMIO when a product participates in two or more splits. In particular, observe that a product $i$ may participate in two different splits $s_1$ and $s_2$ in the same decision tree $t$. In this case, constraint~\eqref{prob:SplitMIO_left} in the \SplitMIO will result in two constraints:
\begin{align}
& \sum_{\ell \in \leftleaves(s_1)} y_{t,\ell} \leq x_i, \label{eq:split_constr_s1}\\
& \sum_{\ell \in \leftleaves(s_2)} y_{t,\ell} \leq x_i. \label{eq:split_constr_s2}
\end{align}
In the above two constraints, observe that $\leftleaves(s_1)$ and $\leftleaves(s_2)$ are disjoint (this is a direct consequence of Assumption~\ref{assumption:Requirement_3}). Given this and constraint~\eqref{prob:SplitMIO_y_unitsum} that requires the $y_{t,\ell}$ variables to sum to 1, we can come up with a constraint that strengthens constraints~\eqref{eq:split_constr_s1} and \eqref{eq:split_constr_s2} by combining them:
\begin{equation}
\sum_{\ell \in \leftleaves(s_1)} y_{t,\ell} + \sum_{\ell \in \leftleaves(s_2)} y_{t,\ell} \leq x_i. 
\end{equation}
In general, one can aggregate all the $y_{t,\ell}$ variables that are to the left of all splits involving a product $i$ to produce a single left split constraint for product $i$. The same can also be done for the right split constraints. Generalizing this principle leads to the following alternate formulation, which we refer to as \ProductMIO:
\begin{subequations}
\begin{alignat}{2}
\ProductMIO: \quad & \underset{\xb, \yb}{\text{maximize}} &\quad & \sum_{t \in F} \lambda_t \cdot \left[  \sum_{\ell \in \leaves(t)} r_{t,\ell} y_{t,\ell} \right] \\
& \text{subject to} & & \sum_{\ell \in \leaves(t)} y_{t,\ell} = 1, \quad \forall \ t \in F, \label{prob:ProductMIO_unisum} \\
& & & \sum_{ \substack{s \in \splits(t):\\ v(t,s) = i}} \ \sum_{\ell \in \leftleaves(s)} y_{t,\ell} \leq x_i, \quad \forall \ t \in F, \ i \in \Ncal, \label{prob:ProductMIO_leftsplitconstraint}\\
& & & \sum_{ \substack{s \in \splits(t):\\ v(t,s) = i}} \  \sum_{\ell \in \rightleaves(s)} y_{t,\ell} \leq 1 - x_i, \quad \forall \ t \in F, \ i \in \Ncal, \label{prob:ProductMIO_rightsplitconstraint} \\
& & & x_i \in \{0,1\}, \quad \forall \ i \in \Ncal, \\
& & & y_{t,\ell} \geq 0, \quad \forall \ t \in F, \ \ell \in \leaves(t).
\end{alignat}
\label{prob:ProductMIO}%
\end{subequations}
Note that while both formulations have the same number of variables, \YCRR{formulation~\ProductMIO has a smaller or equal number of constraints compared with \SplitMIO}. In particular, \SplitMIO has one left and one right split constraint for each split in each tree, whereas \ProductMIO has one left and one right split constraint for each product. When the trees involve a large number of splits, this can lead to a sizable reduction in the number of constraints. Additionally, when a product does not appear in any splits of a tree, we can also safely omit constraints~\eqref{prob:ProductMIO_leftsplitconstraint} and \eqref{prob:ProductMIO_rightsplitconstraint} for that product. As we noted earlier, formulation~\ProductMIO is at least as strong as formulation~\SplitMIO. Let $\Fcal_{\ProductMIO}$ be the feasible region of the linear optimization (LO) relaxation of \ProductMIO. The following proposition formalizes the claim regarding the strength of the two formulations.
\begin{proposition}
	For any decision forest model $(F, \lambdab)$, $\Fcal_{\ProductMIO} \subseteq \Fcal_{\SplitMIO}$. \label{proposition:ProductMIO_stronger_than_SplitMIO}
\end{proposition}
The proof follows straightforwardly using the logic given above; we thus omit the proof. Finally, while \ProductMIO appears as a tightened reformulation of \SplitMIO, its structure connects to another type of formulation in the literature. In particular, a stream of papers in the mixed-integer optimization community \citep{vielma2010mixed,vielma2011modeling,huchette2019combinatorial} has considered a general approach for deriving small and strong formulations of disjunctive constraints using independent branching schemes. In Section~\ref{subsec:appendix-CDC}, we briefly review the most general such approach from \cite{huchette2019combinatorial} and showcase its connection to \ProductMIO.

\subsection{Comparative Insights on the \SplitMIO and \ProductMIO Formulations}
\label{subsec:SplitMIO_vs_ProductMIO}

\YCRR{For computational and practical purposes, as we illustrate in Section~\ref{subsec:synthetic_T_integrality_gap}, \ProductMIO is generally preferable to \SplitMIO when the assortment problem fits into computer memory and both formulations can be solved directly using commercial solvers such as Gurobi \citep{gurobi}. In our numerical experiments, this corresponds to the regime where the number of products is moderate (up to $n=500$). In this setting, \ProductMIO clearly outperforms \SplitMIO when the trees are highly structured and products are frequently repeated across splits, and shows a smaller but still noticeable edge when the trees are less structured. By contrast, when the problem size increases further (up to $n=3000$), \SplitMIO becomes preferable due to the availability of a tailored Benders decomposition method developed in Section~\ref{sec:benders}. In particular, for \SplitMIO, Benders cuts can be generated efficiently for both integral and fractional outer decisions $\xb$, using closed-form expressions in the integral case and our proposed primal-dual greedy algorithm in the fractional case. For \ProductMIO, however, such efficient cut generation is only available when $\xb$ is integral. This gives \SplitMIO a distinct advantage in large-scale settings during the relaxation phase of Benders decomposition. We demonstrate this advantage in Section~\ref{subsec:greedy_experiment}.}

\section{Solution methodology based on Benders decomposition}
\label{sec:benders}

While the formulations in Section~\ref{sec:model} bring the assortment optimization problem under the decision forest choice model closer to being solvable in practice, the effectiveness of these formulations can be limited in large-scale problems. In particular, consider the case where there is a large number of trees in the decision forest model and each tree consists of a large number of splits and leaves. In this setting, both formulations -- \SplitMIO and \ProductMIO -- will have a large number of $y_{t,\ell}$ variables and a large number of constraints to link those variables with the $x_i$ variables, and may require significant computation time.

At the same time, \SplitMIO and \ProductMIO share a common problem structure. In particular, they have two sets of variables: the $\xb$ variables, which determine the products that are to be included, and the $(\yb_t)_{t \in F}$ variables, which model the choice of each customer type. In addition, for any two trees $t, t'$ such that $t \neq t'$, the $\yb_t$ variables and $\yb_{t'}$ variables do not appear together in any constraints. Thus, one can view each of our two formulations as \emph{a two-stage stochastic program}, where each tree $t$ corresponds to a scenario; the variable $\xb$ corresponds to the first-stage decision; and the variable $\yb_t$ corresponds to the second-stage decision under scenario $t$, which is appropriately constrained by the first-stage decision $\xb$. 

Thus, one can apply Benders decomposition to solve the problem. At a high level, Benders decomposition involves using linear optimization duality to represent the optimal value of the second-stage problem for each tree $t$ as a piecewise-linear concave function of $\xb$, and to eliminate the $(\yb_t)_{t \in F}$ variables. One can then re-write the optimization problem in epigraph form, resulting in an optimization problem in terms of the $\xb$ variable and an auxiliary epigraph variable $\theta_t$ for each tree $t$, and a large family of constraints linking $\xb$ and $\theta_t$ for each tree $t$. Although the constraints for each tree $t$ are too many to be enumerated, one can solve the problem through constraint generation.

The main message of this section is that the second-stage problem can be solved in a computationally efficient manner. In the remaining sections, we carefully analyze the second-stage problem for both formulations. For \SplitMIO, we show that the second-stage problem can be solved by a greedy algorithm when $\xb$ is fractional (Section~\ref{subsec:benders_SplitMIO_relaxation}). For \ProductMIO, we show that the same greedy approach does not solve the second-stage problem in the fractional case (Section~\ref{subsec:benders_ProductMIO_relaxation}). For both formulations, when $\xb$ is binary, we characterize the primal and dual solutions in closed form (Section~\ref{subsec:bender_integral_solutions}). Lastly, in Section~\ref{subsec:benders_overall_approach}, we briefly describe our overall algorithmic approach to solving the assortment optimization problem, which involves solving the Benders reformulation of the relaxed problem, followed by the Benders reformulation of the integer problem.

\subsection{Benders reformulation of the \SplitMIO relaxation} 
\label{subsec:benders_SplitMIO_relaxation}

The Benders reformulation of the LO relaxation of \SplitMIO can be written as
\begin{subequations}
	\label{prob:SplitMIO_benders_master_abstract}
	\begin{alignat}{2}
	& \underset{\xb,\thetab}{\text{maximize}} & \quad & \sum_{t \in F} \lambda_t \theta_t \\
	& \text{subject to} & & \theta_t \leq G_t(\xb), \quad \forall \ t \in F,\\
	& & & \xb \in [0,1]^n,
	\end{alignat}
\end{subequations}
where the function $G_t(\xb)$ is the optimal value of the following subproblem for tree $t$:
\begin{subequations}
\label{prob:SplitMIO_sub_primal}
\begin{alignat}{2}
G_t(\xb) = \quad & \underset{\yb_t}{\text{maximize}} & \quad & \sum_{\ell \in \leaves(t)} r_{t,\ell} \cdot y_{t, \ell} \\
& \text{subject to} & & \sum_{\ell \in \leaves(t)} y_{t, \ell} = 1, \label{prob:SplitMIO_sub_primal_unitsum} \\
& & & \sum_{\ell \in \leftleaves(s)} y_{t, \ell} \leq x_{v(t,s)}, \quad \forall \ s \in \splits(t), \label{prob:SplitMIO_sub_primal_left} \\
& & & \sum_{\ell \in \rightleaves(s)} y_{t, \ell} \leq 1 - x_{v(t,s)}, \quad \forall \ s \in \splits(t), \label{prob:SplitMIO_sub_primal_right} \\
& & & y_{t, \ell} \geq 0, \quad \forall \ \ell \in \leaves(t). \label{prob:SplitMIO_sub_primal_nonnegative}
\end{alignat}%
\end{subequations}

\YCRF{We present a greedy algorithm for solving problem~\eqref{prob:SplitMIO_sub_primal}, with implementation details given in Section~\ref{subsec-appendix:greedy-implementation} of the appendix. Intuitively, the greedy algorithm progresses through the leaves from highest to lowest revenue, and sets the $y_{t,\ell}$ variable of each leaf $\ell$ to the highest value it can be set to without violating the left and right split constraints~\eqref{prob:SplitMIO_sub_primal_left} and \eqref{prob:SplitMIO_sub_primal_right} and without violating the constraint $\sum_{\ell \in \leaves(t)} y_{t,\ell} \leq 1$. At each iteration, the algorithm additionally keeps track of which constraint became tight through the event set $\Ecal$. An $A_s$ event indicates that the left split constraint~\eqref{prob:SplitMIO_sub_primal_left} for split $s$ became tight; a $B_s$ event indicates that the right split constraint~\eqref{prob:SplitMIO_sub_primal_right} for split $s$ became tight; and a $C$ event indicates that the constraint $\sum_{\ell \in \leaves(t)} y_{t,\ell} \leq 1$ became tight. When a $C$ event is not triggered, we look for the split which has the least remaining capacity. In the case that the $\arg \min$ is not unique and there are two or more splits that are tied, we break ties by choosing the split $s$ with the lowest depth. We also define a function $f$ to keep track of which leaf $\ell$ was being checked when an $A_s$ / $B_s$ / $C$ event occurred. We note that both $\Ecal$ and $f$ are not needed to find the primal solution, but they are essential to determining the dual solution in the dual procedure (Algorithm~\ref{algorithm:SplitMIO_dual_greedy}, which we will define shortly). We provide the implementation details of the greedy algorithm in Section~\ref{subsec-appendix:greedy-implementation}.}

The greedy algorithm returns a feasible solution that is an extreme point of the polyhedron defined in problem~\eqref{prob:SplitMIO_sub_primal}. We state the result formally as the following theorem.
\begin{theorem}
	Fix $t \in F$. Let $\yb_t$ be a solution to problem~\eqref{prob:SplitMIO_sub_primal} produced by the greedy algroithm (Section~\ref{subsec-appendix:greedy-implementation}). Then: (a) $\yb_t$ is a feasible solution to problem~\eqref{prob:SplitMIO_sub_primal}; and (b) $\yb_t$ is an extreme point of the feasible region of problem~\eqref{prob:SplitMIO_sub_primal}. 
	\label{theorem:SplitMIO_benders_primal_greedy_BFS}
\end{theorem}
Theorem~\ref{theorem:SplitMIO_benders_primal_greedy_BFS} implies that problem~\eqref{prob:SplitMIO_sub_primal} is feasible, and since the problem is bounded, it has a finite optimal value. By strong duality, the optimal value of problem~\eqref{prob:SplitMIO_sub_primal} is equal to the optimal value of its dual:
\begin{subequations}
\label{prob:SplitMIO_sub_dual}
\begin{alignat}{2}
& \underset{\alphab_t, \betab_t, \gamma_t}{\text{minimize}} & \quad & \sum_{s \in \splits(t)} x_{v(t,s)} \cdot \alpha_{t,s} + \sum_{s \in \splits(t)} (1 - x_{v(t,s)}) \beta_{t,s} + \gamma_t \\
& \text{subject to} & & \sum_{s: \ell \in \leftleaves(s)} \alpha_{t,s} + \sum_{s: \ell \in \rightleaves(s)} \beta_{t,s} + \gamma_t \geq r_{t,\ell}, \quad \forall \ \ell \in \leaves(t), \label{prob:SplitMIO_sub_dual_r} \\
& & & \alpha_{t,s} \geq 0, \quad \forall \ s \in \splits(t), \label{prob:SplitMIO_sub_dual_alphageqzero} \\
& & & \beta_{t,s} \geq 0, \quad \forall \ s \in \splits(t). \label{prob:SplitMIO_sub_dual_betageqzero} 
\end{alignat}
\end{subequations}
Letting $\Dcal_{t,\SplitMIO}$ denote the set of feasible solutions to the dual subproblem~\eqref{prob:SplitMIO_sub_dual}, we can formulate the master problem~\eqref{prob:SplitMIO_benders_master_abstract} as
\begin{subequations}
\label{prob:SplitMIO_benders_master_Dcal}
\begin{alignat}{2}
& \underset{\xb, \thetab}{\text{maximize}} & & \sum_{t \in F} \lambda_t \theta_t \\
& \text{subject to} & \quad & \theta_t \leq \sum_{s \in \splits(t)} x_{v(t,s)} \cdot \alpha_{t,s} + \sum_{s \in \splits(t)} (1 - x_{v(t,s)}) \beta_{t,s} + \gamma_t, \nonumber \\
& & & \quad  \forall \ (\alphab_t, \betab_t, \gamma_t) \in \Dcal_{t,\SplitMIO}, \label{prob:SplitMIO_benders_master_Dcal_mainconstraint} \\
& & & \xb \in [0,1]^n. 
\end{alignat}
\end{subequations}

The value of this formulation, relative to the original formulation, is that we have replaced the $(\yb_t)_{t \in F}$ variables and the constraints that link them to the $\xb$ variables, with a large family of constraints in terms of $\xb$. Although this new formulation is still challenging, the advantage of this formulation is that it is suited to constraint generation.

The constraint generation approach to solving problem~\eqref{prob:SplitMIO_benders_master_Dcal} involves starting the problem with no constraints and then, for each $t \in F$, checking whether constraint~\eqref{prob:SplitMIO_benders_master_Dcal_mainconstraint} is violated. If constraint~\eqref{prob:SplitMIO_benders_master_Dcal_mainconstraint} is not violated for any $t \in F$, then we conclude that the current solution $\xb$ is optimal. Otherwise, for any $t \in F$ such that constraint~\eqref{prob:SplitMIO_benders_master_Dcal_mainconstraint} is violated, we add the constraint corresponding to the $(\alphab_t, \betab_t, \gamma_t)$ solution at which the violation occurred, and solve the problem again to obtain a new $\xb$. The procedure then repeats at the new $\xb$ solution until no more violated constraints have been found.

The crucial step to solving this problem is being able to solve the dual subproblem~\eqref{prob:SplitMIO_sub_dual}; that is, for a fixed $t \in F$, either asserting that the current solution $\xb$ satisfies constraint~\eqref{prob:SplitMIO_benders_master_Dcal_mainconstraint} for all $(\alphab_t, \betab_t, \gamma_t) \in \Dcal_{t,\SplitMIO}$ or identifying a $(\alphab_t, \betab_t, \gamma_t)$ at which constraint~\eqref{prob:SplitMIO_benders_master_Dcal_mainconstraint} is violated. This amounts to solving the dual subproblem~\eqref{prob:SplitMIO_sub_dual} and comparing its objective value to $\theta_t$.

Fortunately, it turns out that we can solve the dual subproblem~\eqref{prob:SplitMIO_sub_dual} using a specialized algorithm, in the same way that we can solve the primal subproblem~\eqref{prob:SplitMIO_sub_primal} using the greedy algorithm. The following algorithm, Algorithm~\ref{algorithm:SplitMIO_dual_greedy}, uses auxiliary information obtained during the execution of the greedy algorithm. We use $d(s)$ to denote the depth of an arbitrary split, where the root split corresponds to a depth of 1, and $d_{\max} = \max_{s \in \splits(t)} d(s)$ is the depth of the deepest split in the tree. In addition, we use $\splits(t,d) = \{ s \in \splits(t) \mid d(s) = d\}$ to denote the set of all splits at a particular depth $d$. Algorithm~\ref{algorithm:SplitMIO_dual_greedy} proceeds sequentially from the root to the deepest splits, updating the dual variables $\alpha_{t,s}$ and $\beta_{t,s}$ for each split $s$ based on the tree topology and the auxiliary information $\Ecal$ and $f$ obtained from the greedy algorithm. When $A_s \in \Ecal$, constraint~\eqref{prob:SplitMIO_sub_primal_left} in the primal problem is tight, whereas constraint~\eqref{prob:SplitMIO_sub_primal_right} is not. This implies that the dual variable $\beta_{t,s}$, associated with the latter constraint, must remain zero. The other dual variable, $\alpha_{t,s}$, is then updated according to constraint~\eqref{prob:SplitMIO_sub_dual_r} in the dual problem. Specifically, recall that the auxiliary function $f$ in the greedy algorithm tracks the leaf node $\ell = f(A_s)$ with nonzero $y_{t,\ell}$ when $A_s$ occurs. Since the corresponding constraint for $\ell$ in constraint~\eqref{prob:SplitMIO_sub_dual_r} is tight, it is used to calibrate $\alpha_{t,s}$. A symmetric argument applies when $B_s \in \Ecal$ instead. We provide a worked example of the execution of the greedy algorithm and Algorithm\ref{algorithm:SplitMIO_dual_greedy} in Appendix~\ref{appendix:SplitMIO_benders_example}. 

\begin{algorithm}
\SingleSpacedXI
\small
\begin{algorithmic}[1]
\STATE Initialize $\alpha_{t,s} \gets 0, \beta_{t,s} \gets 0$ for all $s \in \splits(t)$, $\gamma_t \gets 0$
\STATE Set $\gamma \gets r_{f(C)}$
\FOR{ $d = 1,\dots, d_{\max}$ }
	\FOR{ $s \in \splits(t,d)$ }
		\IF{ $A_s \in \Ecal$}
			\STATE Set $\alpha_{t,s} \gets r_{t,f(A_s)} - \gamma_t - \sum_{\substack{s' \in \LS(f(A_s)):  A_{s'} \in \Ecal,  d(s') < d}} \alpha_{t,s'} - \sum_{\substack{s' \in \RS(f(A_s)):  B_{s'} \in \Ecal,  d(s') < d}} \beta_{t,s'}$
		\ENDIF
		\IF{ $B_s \in \Ecal$}
			\STATE Set $\beta_{t,s} \gets r_{t,f(B_s)} - \gamma_t - \sum_{\substack{s' \in \LS(f(A_s)):  A_{s'} \in \Ecal, d(s') < d}} \alpha_{t,s'} - \sum_{\substack{s' \in \RS(f(A_s)):  B_{s'} \in \Ecal,  d(s') < d}} \beta_{t,s'}$
		\ENDIF
	\ENDFOR
\ENDFOR
\end{algorithmic}

\caption{Dual greedy algorithm for \SplitMIO. \label{algorithm:SplitMIO_dual_greedy}}
\end{algorithm}

We can show that the dual solution produced by Algorithm~\ref{algorithm:SplitMIO_dual_greedy} is a feasible extreme point solution of the dual subproblem~\eqref{prob:SplitMIO_sub_dual}.
\begin{theorem}
	Fix $t \in F$. Let $(\alphab_t, \betab_t, \gamma_t)$ be a solution to problem~\eqref{prob:SplitMIO_sub_dual} produced by Algorithm~\ref{algorithm:SplitMIO_dual_greedy}. Then: (a) $(\alphab_t, \betab_t, \gamma_t)$ is a feasible solution to problem~\eqref{prob:SplitMIO_sub_dual}; and (b) $(\alphab_t, \betab_t, \gamma_t)$ is an extreme point of the feasible region of problem~\eqref{prob:SplitMIO_sub_dual}.
	\label{theorem:SplitMIO_benders_dual_greedy_BFS}
\end{theorem}

Lastly, and most importantly, we show that the solutions produced by the greedy algorithm and Algorithm \ref{algorithm:SplitMIO_dual_greedy} are optimal for their respective problems. Thus, Algorithm~\ref{algorithm:SplitMIO_dual_greedy} is a valid procedure for identifying values of $(\alphab_t, \betab_t, \gamma_t)$ at which constraint~\eqref{prob:SplitMIO_benders_master_Dcal_mainconstraint} is violated.

\begin{theorem}
	Fix $t \in F$. Let $\yb_t$ be a solution to problem~\eqref{prob:SplitMIO_sub_primal} produced by the greedy algorithm (Section~\ref{subsec-appendix:greedy-implementation}) and $(\alphab_t, \betab_t, \gamma_t)$ be a solution to problem~\eqref{prob:SplitMIO_sub_dual} produced by Algorithm~\ref{algorithm:SplitMIO_dual_greedy}. Then: (a) $\yb_t$ is an optimal solution to problem~\eqref{prob:SplitMIO_sub_primal}; and (b) $(\alphab_t, \betab_t, \gamma_t)$ is an optimal solution to problem~\eqref{prob:SplitMIO_sub_dual}. 
	\label{theorem:SplitMIO_benders_primal_dual_greedy_optimal}
\end{theorem}
The proof of this result proceeds by verifying that the two solutions satisfy complementary slackness. \YCRF{We note that the optimality of the greedy algorithm for solving the primal subproblem~\eqref{prob:SplitMIO_sub_primal} can also be interpreted through the lens of laminar capacity polytopes. In particular, for fixed $\xb$, the feasible region of problem~\eqref{prob:SplitMIO_sub_primal} is defined by capacity constraints over leaf subsets induced by rooted subtrees of $t$. Since these subtree leaf sets form a laminar family, the subproblem admits a natural laminar-capacity structure. Linear optimization over such structures often admits greedy solutions (see, e.g., \cite{schrijver2003combinatorial,korte2008combinatorial}). This observation provides additional structural intuition for why a greedy principle applies to the \SplitMIO relaxation.}

Before continuing, we pause to make two remarks. First, Algorithm~\ref{algorithm:SplitMIO_dual_greedy} can be viewed as the generalization of the algorithm arising in the Benders decomposition approach to the ranking-based assortment optimization problem \citep{bertsimas2019exact}. The results of that paper show that the primal subproblem of the MIO formulation in \cite{bertsimas2019exact} can be solved via a greedy algorithm and the dual subproblem can be solved via an algorithm that uses information from the primal algorithm (analogous to Algorithm~\ref{algorithm:SplitMIO_dual_greedy}). This generalization is not straightforward. The main challenge in this generalization is designing the sequence of updates in the greedy algorithm according to the tree topology. For example, in solving the primal subproblem~\eqref{prob:SplitMIO_sub_primal}, one considers all left/right splits and the $y$ values of their left/right leaves when constructing the lowest upper bound of $y_\ell$ for each leaf node $\ell$. Also, as shown in Algorithm~\ref{algorithm:SplitMIO_dual_greedy}, the dual variables $\alpha_{t,s}$ and $\beta_{t,s}$ have to be updated according to the tree topology and the events $A_{s'}$ and $B_{s'}$ of the split $s'$ with smaller depth. In contrast, in the ranking-based assortment optimization problem, one only needs to calculate the ``capacities'' (the $q_s$ values in Algorithm~\ref{algorithm:SplitMIO_primal_greedy} provided in Section~\ref{subsec-appendix:greedy-implementation}) by subtracting the $y$ values of the preceding products in the ranking, which does not inherit any topological structure. For these reasons, the primal and dual Benders subproblems for the decision forest assortment problem are much more challenging than those of the ranking-based assortment problem.

Second, while the dual problem~\eqref{prob:SplitMIO_sub_dual} can be solved directly using commercial software such as Gurobi \citep{gurobi}, the proposed greedy algorithm is generally more efficient. In Appendix~\ref{subsec:greedy_experiment}, we numerically demonstrate its efficiency, showing that when solving the relaxed \SplitMIO problem, the greedy algorithm can reduce runtime by up to 90\% compared to the approach where each constraint is added via solving the dual subproblem using Gurobi.

\subsection{Benders reformulation of the \ProductMIO relaxation} 
\label{subsec:benders_ProductMIO_relaxation}
We also consider a Benders reformulation of the linear relaxation of \ProductMIO. The Benders master problem is given by formulation~\eqref{prob:SplitMIO_benders_master_abstract} where the function $G_t(\xb)$ is defined as the optimal value of the \ProductMIO subproblem for tree $t$. To aid in the definition of the subproblem, let $P(t)$ denote the set of products that appear in the splits of tree $t$: $$P(t) = \{ i \in \Ncal \mid i = v(t,s)\ \text{for some}\ s \in \splits(t) \}.$$ With a slight abuse of notation, let $\leftleaves(i)$ denote the set of leaves for which product $i$ must be included in the assortment for those leaves to be reached, and similarly, let $\rightleaves(i)$ denote the set of leaves for which product $i$ must be excluded from the assortment for those leaves to be reached; formally, $$\leftleaves(i)  = \bigcup_{ \substack{s \in \splits(t): v(t,s) = i}} \leftleaves(s),  \qquad
\rightleaves(i)  = \bigcup_{\substack{s \in \splits(t): v(t,s) = i}} \rightleaves(s).$$
With these definitions, we can write down the \ProductMIO subproblem as follows:
\begin{subequations}
\label{prob:ProductMIO_sub_primal}
\begin{alignat}{2}
G_t(\xb) = \quad & \underset{\yb_t}{\text{maximize}} & \quad & \sum_{\ell \in \leaves(t)} r_{t,\ell} \cdot y_{t, \ell} \\
& \text{subject to} & & \sum_{\ell \in \leaves(t)} y_{t,\ell} = 1, \label{prob:ProductMIO_sub_primal_unitsum} \\
& & & \sum_{ \ell \in \leftleaves(i) } y_{t,\ell} \leq x_i, \quad \forall \ i \in P(t),  \label{prob:ProductMIO_sub_primal_left} \\
& & & \sum_{ \ell \in \rightleaves(i)} y_{t,\ell} \leq 1 - x_i, \quad \forall \ i \in P(t), \label{prob:ProductMIO_sub_primal_right} \\
& & & y_{t,\ell} \geq 0, \quad \forall \ \ell \in \leaves(t). \label{prob:ProductMIO_sub_primal_nonnegative}
\end{alignat}%
\end{subequations}
In the same way as \SplitMIO, one can consider solving problem~\eqref{prob:ProductMIO_sub_primal} using a greedy approach, where one iterates through the leaves from highest to lowest revenue, and sets each leaf's $y_{t,\ell}$ variable to the highest possible value without violating any of the constraints. %
Unlike \SplitMIO, it unfortunately turns out that this greedy approach is not always optimal. We formalize this claim as a proposition in Section~\ref{subsec:Bender_subproblem_productMIO_not_greedy_solvable} of the appendix, supported by a counterexample.

\subsection{Bender Cuts for Integer Master Solutions}
\label{subsec:bender_integral_solutions}

We further propose closed form expressions for the structure of the optimal primal and dual Benders subproblem solutions for integer solutions $\xb$.

\subsubsection{\SplitMIO}
\label{}

Our results in Section~\ref{subsec:benders_SplitMIO_relaxation} for obtaining primal and dual solutions for the subproblem of \SplitMIO apply for any $\xb \in [0,1]^n$; in particular, they apply for fractional choices of $\xb$, thus allowing us to solve the Benders reformulation of the relaxation of \SplitMIO. 

In the special case that $\xb$ is a candidate integer solution of \SplitMIO, we can find optimal solutions to the primal and dual subproblems of \SplitMIO in closed form.
\begin{theorem}
	Fix $t \in F$, and let $\xb \in \{0,1\}^n$. Define the primal subproblem solution $\yb_t$ as 
	\begin{equation*}
	y_{t,\ell} = \left\{ \begin{array}{ll} 1 & \text{if} \ \ell = \ell^*,\\
	0 & \text{if} \ \ell \neq \ell^*, \end{array} \right.
	\end{equation*}
	where $\ell^*$ denotes the leaf that the assortment encoded by $\xb$ is mapped to. Define the dual subproblem solution $(\alphab_t, \betab_t, \gamma_t)$ as
	\begin{align*}
	\alpha_{t,s} & = \left\{ \begin{array}{ll} \max\{ 0, \max_{\ell \in \leftleaves(s)} r_{t,\ell} - r_{t,\ell^*} \} & \text{if}\ s \in \RS(\ell^*), \\
	0 & \text{otherwise}, \end{array} \right. \\
	\beta_{t,s} & = \left\{ \begin{array}{ll} \max\{0, \max_{\ell \in \rightleaves(s)} r_{t,\ell} - r_{t,\ell^*}\} & \text{if} \ s \in \LS(\ell^*), \\
	0 & \text{otherwise}, \end{array} \right. \\
	\gamma_t & = r_{t,\ell^*}.
	\end{align*}
	Then: (a) $\yb_t$ is a feasible solution to problem~\eqref{prob:SplitMIO_sub_primal}; (b) $(\alphab_t, \betab_t, \gamma_t)$ is a feasible solution to problem~\eqref{prob:SplitMIO_sub_dual}; and (c) $\yb_t$ and $(\alphab_t, \betab_t, \gamma_t)$ are optimal for problems~\eqref{prob:SplitMIO_sub_primal} and \eqref{prob:SplitMIO_sub_dual}, respectively.
	\label{theorem:SplitMIO_benders_primal_dual_binary_closedform_optimal}
\end{theorem}

The significance of Theorem~\ref{theorem:SplitMIO_benders_primal_dual_binary_closedform_optimal} is that it provides a simpler means to checking for violated constraints when $\xb$ is binary than applying Algorithms~\ref{algorithm:SplitMIO_primal_greedy} and \ref{algorithm:SplitMIO_dual_greedy}. In particular, for the integer version of \SplitMIO, a similar derivation as in Section~\ref{subsec:benders_SplitMIO_relaxation} leads us to the following Benders reformulation of the integer problem for the \SplitMIO formulation:
\begin{subequations}
\label{prob:SplitMIO_benders_integer_master_Dcal}
\begin{alignat}{2}
& \underset{\xb, \thetab}{\text{maximize}} & & \sum_{t \in F} \lambda_t \theta_t \\
& \text{subject to} & \quad & \theta_t \leq \sum_{s \in \splits(t)} x_{v(t,s)} \cdot \alpha_{t,s} + \sum_{s \in \splits(t)} (1 - x_{v(t,s)}) \beta_{t,s} + \gamma_t, \nonumber \\
& & & \quad  \forall \ (\alphab_t, \betab_t, \gamma_t) \in \Dcal_{t,\SplitMIO}, \label{prob:SplitMIO_benders_integer_master_Dcal_mainconstraint} \\
& & & \xb \in \{0,1\}^n. 
\end{alignat}
\end{subequations}
To check whether constraint~\eqref{prob:SplitMIO_benders_integer_master_Dcal_mainconstraint} is violated for a particular $\xb$ and a tree $t$, we can simply use Theorem~\ref{theorem:SplitMIO_benders_primal_dual_binary_closedform_optimal} to determine the optimal value of the subproblem, and compare it against $\theta_t$; if the constraint corresponding to the dual solution of Theorem~\ref{theorem:SplitMIO_benders_primal_dual_binary_closedform_optimal} is violated, we add that constraint to the problem. In our implementation of Benders decomposition, we embed the constraint generation process for the integer problem~\eqref{prob:SplitMIO_benders_integer_master_Dcal} within the branch-and-bound tree, using a technique referred to as \emph{lazy constraint generation}; we discuss this more in Section~\ref{subsec:benders_overall_approach}.

\subsubsection{\ProductMIO}

We also consider \ProductMIO. We begin by writing down the dual of the subproblem, for which we need to define several additional sets. We let $\LP(\ell)$ denote the set of ``left products'' of leaf $\ell$ (those products that must be included in the assortment for leaf $\ell$ to be reached), and let $\RP(\ell)$ denote the set of ``right products'' of leaf $\ell$ (those products that must be excluded from the assortment for leaf $\ell$ to be reached). Note that $\ell \in \leftleaves(i)$ if and only if $i \in \LP(\ell)$, and similarly $\ell \in \rightleaves(i)$ if and only if $i \in \RP(\ell)$. 

With these definitions, the dual of the primal subproblem~\eqref{prob:ProductMIO_sub_primal} is 
\begin{subequations}
	\label{prob:ProductMIO_sub_dual}
	\begin{alignat}{2}
	& \underset{\alphab_t, \betab_t, \gamma_t}{\text{minimize}} & \quad & \sum_{i \in P(t)} \alpha_{t,i} x_i + \sum_{i \in P(t)} \beta_{t,i} (1 - x_i) + \gamma_t \\
	& \text{subject to} & & \sum_{i \in \LP(\ell)} \alpha_{t,i} + \sum_{i \in \RP(\ell)} \beta_{t,i} + \gamma_t \geq r_{\ell}, \quad \forall \ \ell \in \leaves(t), \label{prob:ProductMIO_sub_dual_r} \\
	& & & \alpha_{t,i} \geq 0, \quad \forall\ i \in P(t), \label{prob:ProductMIO_sub_dual_alphageqzero} \\
	& & & \beta_{t,i} \geq 0, \quad \forall\ i \in P(t). \label{prob:ProductMIO_sub_dual_betageqzero}
	\end{alignat}%
\end{subequations}

In the case that $\xb$ is integer, we can obtain optimal solutions to the primal subproblem~\eqref{prob:ProductMIO_sub_primal} and its dual \eqref{prob:ProductMIO_sub_dual} in closed form. 

\begin{theorem}
	Fix $t \in F$ and let $\xb \in \{0,1\}^n$. Let $\yb_t$ be defined as in Theorem~\ref{theorem:SplitMIO_benders_primal_dual_binary_closedform_optimal} 
	and let $(\alphab_t, \betab_t, \gamma_t)$ be defined as 
	\begin{align}
	\alpha_{t,i} & = \left\{ \begin{array}{ll} \max\{0, \max_{\ell \in \leftleaves(i)} r_{t,\ell} - r_{t,\ell^*} \}, & \text{if}\ i \in \RP(\ell^*), \\
	0 & \text{otherwise}, \end{array} \right. \\
	\beta_{t,i} & = \left\{ \begin{array}{ll} \max\{0, \max_{\ell \in \rightleaves(i)} r_{t,\ell} - r_{t,\ell^*} \}, & \text{if}\ i \in \LP(\ell^*), \\
	0 & \text{otherwise}, \end{array} \right. \\
	\gamma_t & = r_{t,\ell^*}.
	\end{align}
	Then: (a) $\yb_t$ is a feasible solution for problem~\eqref{prob:ProductMIO_sub_primal}; (b) $(\alphab_t, \betab_t, \gamma_t)$ is a feasible solution for problem~\eqref{prob:ProductMIO_sub_dual}; and (c) $\yb_t$ and $(\alphab_t, \betab_t, \gamma_t)$ are optimal for problems~\eqref{prob:ProductMIO_sub_primal} and \eqref{prob:ProductMIO_sub_dual}, respectively.
	\label{theorem:ProductMIO_benders_primal_dual_binary_closedform_optimal}
\end{theorem}

\subsection{Overall Benders algorithm}
\label{subsec:benders_overall_approach}

We conclude Section~\ref{sec:benders} by summarizing how the results are used. In our overall algorithmic approach below, we first focus on \SplitMIO, as the subproblem can be solved for the formulation when $\xb$ is either fractional or binary. 
\begin{enumerate}
\item \emph{Relaxation phase}. We first solve the relaxed problem~\eqref{prob:SplitMIO_benders_master_Dcal} using ordinary constraint generation. Given a solution $\xb \in [0,1]^n$, we generate Benders cuts by running the primal-dual procedure (Algorithm~\ref{algorithm:SplitMIO_primal_greedy} followed by Algorithm~\ref{algorithm:SplitMIO_dual_greedy}). 
\item \emph{Integer phase}. In the integer phase, we add all of the Benders cuts generated in the relaxation phase to the integer version of problem~\eqref{prob:SplitMIO_benders_master_Dcal}. We then solve the problem as an integer optimization problem, where we generate Benders cuts for integer solutions using the closed form expressions in Theorem~\ref{theorem:SplitMIO_benders_primal_dual_binary_closedform_optimal}. We add these cuts using \emph{lazy constraint generation}. That is, we solve the master problem using a single branch-and-bound tree, and we check whether the main constraint~\eqref{prob:SplitMIO_benders_master_Dcal_mainconstraint} of the Benders formulation is violated at every integer solution generated in the branch-and-bound tree.
\end{enumerate}

For \ProductMIO, as the subproblem can only be solved when $\xb$ is binary, its overall Benders algorithm directly starts with the \emph{Integer Phase}.

We conclude this section by highlighting the advantages of the \SplitMIO formulation in the large-scale regime. At first glance, following Section~\ref{sec:model}, \ProductMIO may appear preferable to \SplitMIO since it is a tighter formulation (Proposition~\ref{proposition:ProductMIO_stronger_than_SplitMIO}) and involves no more constraints. However, in large-scale settings, the proposed primal-dual greedy algorithm enables significantly faster computation of the linear relaxation for \SplitMIO compared to \ProductMIO. As a result, as outlined above, \SplitMIO can benefit from a more efficient warm start during the relaxation phase. 

To support the discussion surrounding the two-phase structure and the advantages afforded by \SplitMIO through the greedy algorithm, we provide two pieces of numerical evidence. \YCRR{First, Appendix~\ref{subsec:value_of_relaxation_phase} shows that incorporating the relaxation phase can effectively reduce the optimality gap of the Benders approach, especially for large instances.} Second, Appendix~\ref{subsec:greedy_experiment} demonstrates that, with the greedy algorithm, solving the relaxed \SplitMIO problem can be up to 95\% faster than solving the relaxed \ProductMIO while providing comparable upper bounds.

\section{Numerical Experiments with Synthetic Data}
\label{sec:synthetic}

In this section, we examine the formulation strength of \SplitMIO and \ProductMIO (Section~\ref{subsec:synthetic_T_integrality_gap}) and demonstrate the scalability of the Benders decomposition approach (Section~\ref{subsec:synthetic_T_B}). We focus on synthetically generated instances so that we can scale the problem size. In Appendix~\ref{appendix:additional_numerical_results}, we report our implementation details and include additional numerical results.

\subsection{Background}

\label{subsec:synthetic_background}

To test our method, we generate three different families of synthetic decision forest instances, which differ in the topology of the trees and the products that appear in the splits:
\begin{enumerate}
\item \textbf{T1 instances}. A T1 forest consists of balanced trees of depth $d$ (i.e., trees where all leaves are at depth $d+1$). For each tree, we sample $d$ products $i_1,\dots, i_d$ uniformly without replacement from $\Ncal$. Then, for every depth $d' \in \{1,\dots, d\}$, we set the split product $v(t,s)$ as $v(t,s) = i_{d'}$ for every split $s$ that is at depth $d'$. 
\item \textbf{T2 instances}. A T2 forest consists of balanced trees of depth $d$. For each tree, we set the split products at each split iteratively, starting at the root, in the following manner: (a) Initialize $d' = 1$; (b) For all splits $s$ at depth $d'$, set $v(s,t) = i_s$ where $i_s$ is drawn uniformly at random from the set $\Ncal \setminus \cup_{s' \in A(s)} \{ v(t, s') \}$, where $A(s)$ is the set of ancestor splits to split $s$ (i.e., all splits appearing on the path from the root node to split $s$); (c) Increment $d' \gets d' + 1$; and (d) If $d' > d$, stop; otherwise, return to Step (b). 

\item \textbf{T3 instances}. A T3 forest consists of unbalanced trees with $L$ leaves. Each tree is generated according to the following iterative procedure: (a) Initialize $t$ to a tree consisting of a single leaf; (b) Select a leaf $\ell$ uniformly at random from $\leaves(t)$, and replace it with a split $s$ and two child leaves $\ell_1, \ell_2$. For split $s$, set $v(s,t) = i_s$ where $i_s$ is drawn uniformly at random from $\Ncal \setminus \cup_{s' \in A(s)} \{ v(t, s') \}$; (c) If $|\leaves(t)| = L$, terminate; otherwise, return to Step (b).
\end{enumerate}

For all three types of forests, we generate the purchase decision $c(t,\ell)$ for each leaf $\ell$ in each tree $t$ in the following way: for each leaf $\ell$, we uniformly at random choose a product $i \in \cup_{s \in \LS(\ell)} \{ v(t,s) \} \cup \{ 0 \}$. In words, the purchase decision is chosen to be consistent with the products that are known to be in the assortment if leaf $\ell$ is reached. Figure~\ref{figure:example_T1_T2_T3} shows an example of each type of tree (T1, T2, and T3). Given a forest of any of the three types above, we generate the customer type probability vector $\lambdab = (\lambda_t)_{t \in F}$ by drawing it uniformly from the $(|F|-1)-$dimensional unit simplex. 

\begin{figure}
\begin{center}
\begin{subfigure}[t]{0.3\textwidth}
\centering
\includegraphics[height=3.5cm]{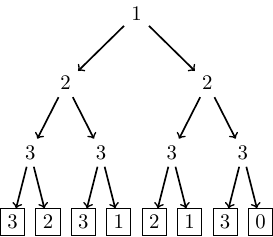}
\caption{T1 tree ($d = 3$).}
\end{subfigure}
\quad
\begin{subfigure}[t]{0.3\textwidth}
\centering
\includegraphics[height=3.5cm]{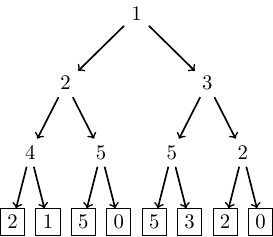}
\caption{T2 tree ($d = 3$).}
\end{subfigure}
\quad
\begin{subfigure}[t]{0.3\textwidth}
\centering
\includegraphics[height=4.5cm]{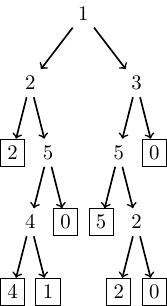}
\caption{T3 tree ($L = 8$).}
\end{subfigure}
\end{center}

\caption{Examples of T1, T2 and T3 trees. \label{figure:example_T1_T2_T3}}
\end{figure}

In our experiments, we fix the number of products $n = 100$ and vary the number of trees $|F| \in \{50, 100, 200, 500\}$, and the number of leaves $|\leaves(t)| \in \{8, 16, 32, 64\}$. (Note that the chosen values for $|\leaves(t)|$ correspond to depths of $\{3,4,5,6\}$ for the T1 and T2 instances.) For each combination of $n$, $|F|$ and $|\leaves(t)|$ and each type of instance (T1, T2 and T3) we randomly generate 20 problem instances, where a problem instance consists of a decision forest model and the product marginal revenues $\rho_1,\dots, \rho_n$. For each instance, the decision forest model is generated according to the process described above and the product revenues are sampled uniformly with replacement from the set $\{1,\dots, 100\}$.

\subsection{Experiment \#1: Formulation Strength}
\label{subsec:synthetic_T_integrality_gap}

Our first experiment is to simply understand how the two formulations --  \SplitMIO and \ProductMIO -- compare in terms of formulation strength. Recall from Proposition \ref{proposition:ProductMIO_stronger_than_SplitMIO} that \ProductMIO is at least as strong as \SplitMIO. For a given instance and a given formulation $\Mcal$ (one of \SplitMIO and \ProductMIO), we define the integrality gap $G^\text{int}_{\Mcal} \equiv 100\% \times \left({Z_{\Mcal} - Z^*}\right) / {Z^*}$, where $Z^*$ is the optimal objective value of the integer problem and $Z_\Mcal$ is the optimal objective of the LO relaxation. We consider the T1, T2 and T3 instances with $n = 100$, $|F| \in \{ 50, 100, 200, 500 \}$ and $|\leaves(t)| = 8$. We restrict our focus to instances with $n = 100$ and $|\leaves(t)| = 8$, as the optimal value $Z^*$ of the integer problem could be computed within one hour for these instances. 

\begin{table}[]
	\centering
	\footnotesize
	\begin{tabular}{rrrlrrrlrrr} \toprule
		\multicolumn{3}{c}{Type T1}                                       &  & \multicolumn{3}{c}{Type T2}                                       &  & \multicolumn{3}{c}{Type T3}                                       \\ \cmidrule{1-3} \cmidrule{5-7} \cmidrule{9-11}
		$|F|$ & $G^\text{int}_{\SplitMIO}$ & $G^\text{int}_{\ProductMIO}$ &  & $|F|$ & $G^\text{int}_{\SplitMIO}$ & $G^\text{int}_{\ProductMIO}$ &  & $|F|$ & $G^\text{int}_{\SplitMIO}$ & $G^\text{int}_{\ProductMIO}$ \\ \cmidrule{1-3} \cmidrule{5-7} \cmidrule{9-11} 
		50    & 0.9                        & 0.0                          &  & 50    & 0.2                        & 0.2                          &  & 50    & 0.2                        & 0.2                          \\
		100   & 2.5                        & 0.1                          &  & 100   & 1.0                        & 1.0                          &  & 100   & 0.5                        & 0.5                          \\
		200   & 5.6                        & 0.2                          &  & 200   & 5.4                        & 5.3                          &  & 200   & 4.1                        & 3.9                          \\
		500   & 15.8                       & 3.3                          &  & 500   & 16.7                       & 16.4                         &  & 500   & 14.2                       & 14.0     \\                   \bottomrule
	\end{tabular}
\caption{Average integrality gap of \SplitMIO and \ProductMIO for T1, T2 and T3 instances. \label{table:synthetic_T_integrality_gap}}
\end{table}

Table~\ref{table:synthetic_T_integrality_gap} displays the average integrality gap of each of the two formulations for each combination of $n$ and $|F|$ and each instance type. From this table, we can see that the integrality gap of both \SplitMIO and \ProductMIO is in general about 0 to 17\%. Note that the difference between \ProductMIO and \SplitMIO is most pronounced for the T1 instances, as the decision forests in these instances exhibit the highest degree of repetition of products within the splits of a tree. In contrast, the difference is much smaller for the T2 and T3 instances, where there is less repetition of products within the splits of the tree (as the trees are not forced to have the same product appear on all of the splits at a particular depth). These results align with the motivation for \ProductMIO. Note also that when a tree does not have any repeated products in its splits, the corresponding constraints of \ProductMIO and \SplitMIO are the same, and \ProductMIO does not confer an advantage over \SplitMIO. Thus, as one constructs trees with less and less repetition in the splits, the improvement in the integrality gap of \ProductMIO over \SplitMIO should become smaller, with the two formulations having the same gap in the most extreme case where no product is repeated.

We also test the tractability of \SplitMIO and \ProductMIO when they are solved as integer programs (i.e., not as \YCRR{a} relaxation). We present the results in Appendix~\ref{subsec:synthetic_T_optimality_gap_T1_T2}. 

\YCRR{
We note that the experiments in this section (Experiment \#1) and in Appendix~\ref{subsec:synthetic_T_optimality_gap_T1_T2} focus on instances where the full formulations of \SplitMIO and \ProductMIO fit in memory and can be solved directly by Gurobi. In this setting, both Table~\ref{table:synthetic_T_integrality_gap} and Table~\ref{table:synthetic_T_optimality_gap} in Appendix~\ref{subsec:synthetic_T_optimality_gap_T1_T2} show that \ProductMIO is generally preferable to \SplitMIO. In particular, when trees are highly structured with extreme product repetition in split nodes, as in the T1 instances, \ProductMIO exhibits a much smaller integrality gap and, in some cases, is solved to optimality by Gurobi up to five times faster than \SplitMIO. When the trees are less repetitive, as in the T2 and T3 instances, the two formulations perform comparably in terms of both integrality gap and runtime. In the next experiment, we further scale up the problem size to enter a regime where directly solving the full formulations of \SplitMIO and \ProductMIO with Gurobi is no longer feasible.
}

\subsection{Experiment \#2: Benders Decomposition for Large-Scale Problems}
\label{subsec:synthetic_T_B}

In this experiment, we report on the performance of our Benders decomposition approach for solving large scale instances of \SplitMIO, \YCRR{where solving the full extended formulation may no longer be viable. These problem instances are significantly larger than those considered in the previous experiment.} We focus on the \SplitMIO formulation, as we are able to efficiently generate Benders cuts for both fractional and integral values of $\xb$. \YCRR{Two further experiments in this large-scale regime are presented in the Appendix. Appendix~\ref{subsec:value_of_relaxation_phase} highlights the value of the Benders cuts generated from fractional values of $\xb$ within the full Benders decomposition framework. Appendix~\ref{subsec:greedy_experiment} compares \SplitMIO and \ProductMIO, demonstrating that, with the proposed primal–dual greedy algorithm in Section~\ref{subsec:benders_SplitMIO_relaxation}, the linear relaxation of \SplitMIO can be solved up to 95\% faster than that of \ProductMIO.}

First, \YCRR{we scale up the problem sizes considered in Experiment \#1.} Specifically, we generate a collection of T3 instances with $n \in \{200, 500, 1000, 2000, 3000\}$, $|F| = 500$ trees and $|\leaves(t)| = 512$ leaves. 
As before, the marginal revenue $\rho_i$ of each product $i$ is chosen uniformly at random from $\{1,\dots,100\}$. For each value of $n$, we generate 5 instances. For each instance, we solve the \SplitMIO problem subject to the constraint $\sum_{i=1}^n x_i = b$, where $b$ is set as $b = \mu n$ and we vary $\mu \in \{0.02, 0.04, 0.06, 0.08, 0.10, 0.12\}$. 

We compare three different methods: the two-phase Benders method described in Section~\ref{subsec:benders_overall_approach}, using the \SplitMIO cut results (Section~\ref{subsec:benders_SplitMIO_relaxation} and Section~\ref{subsec:bender_integral_solutions}); the divide-and-conquer (D\&C) heuristic; and the direct solution approach, where we attempt to directly solve the full \SplitMIO formulation using Gurobi \citep{gurobi}. The D\&C heuristic is a type of local search heuristic proposed in the product line design literature (see \citealt{green1993conjoint}; see also \citealt{belloni2008optimizing}). In this heuristic, one iterates through the $b$ products currently in the assortment, and replaces a single product with the product outside of the assortment that leads to the highest improvement in the expected revenue; this process repeats until the assortment can no longer be improved. We choose the initial assortment uniformly at random from the collection of assortments of size $b$. For each instance, we repeat the D\&C heuristic 10 times, and retain the best solution. We do not impose a time limit on the D\&C heuristic. For the Benders approach, we impose a time limit of one hour on the LO phase, and impose a time limit of one hour on the integer phase. For the direct solution approach, we impose a time limit of two hours, in order to be comparable to the Benders approach.

Table~\ref{table:synthetic_B_T6} reports the performance of the three methods -- the Benders approach, the D\&C heuristic and direct solution of \SplitMIO -- across all combinations of $n$ and $\mu$. We consider several metrics. The metric $G_{\Mcal}$ is defined as $G_\Mcal = (Z' - Z_\Mcal)/Z' \times 100\%$, i.e., it is the optimality gap of the solution obtained by method $\Mcal$ -- either the Benders approach, the direct approach or the D\&C heuristic -- relative to the objective value of the best solution obtained out of the three methods, which is indicated by $Z'$. Lower values of $G_\Mcal$ indicate that the approach tends to deliver solutions that are close to being the best out of the three methods, and a value of 0\% implies that the solution obtained by the approach is the best (or tied for the best) out of the three methods. The metric $T_\Mcal$ indicates the solution time required for approach $\Mcal$ in seconds. Finally, for the Benders and direct approaches, we compute the optimality gap $O_{\Mcal}$, which is defined as $O_{\Mcal} = (Z_{\text{UB}, \Mcal} - Z_{\text{LB}, \Mcal})/Z_{\text{UB}, \Mcal} \times 100\%$, where $Z_{\text{UB},\Mcal}$ and $Z_{\text{LB}, \Mcal}$ are the best upper and lower bounds, respectively, obtained from either the Benders or direct approach after the computation time limit is exhausted. The value reported of each metric is the average over the five replications corresponding to the particular $(n, \mu)$ combination.

\begin{table}
\centering
\footnotesize
\begin{tabular}{lllrrrrrrrr} \toprule
$n$ & $\mu$ & $b$ & $G_{\Benders}$ & $G_{\DNC}$ & $G_{\Direct}$ & $T_{\Benders}$ & $T_{\DNC}$ & $T_{\Direct}$ & $O_{\Benders}$ & $O_{\Direct}$ \\ \midrule
200 & 0.02 &   4 & 0.00 & 0.00 & 0.00 & 14.54 & 2.86 & 2060.81 & 0.00 & 0.00 \\ 
  200 & 0.04 &   8 & 0.00 & 0.04 & 0.61 & 187.96 & 4.46 & 7015.65 & 0.00 & 7.72 \\ 
  200 & 0.06 &  12 & 0.11 & 0.00 & 0.77 & 3615.82 & 7.34 & 7200.25 & 8.23 & 16.71 \\ 
  200 & 0.08 &  16 & 0.70 & 0.00 & 2.86 & 3612.43 & 9.97 & 7200.47 & 17.38 & 23.39 \\ 
  200 & 0.10 &  20 & 0.47 & 0.01 & 6.94 & 3613.19 & 15.07 & 7200.16 & 22.13 & 30.38 \\ 
  200 & 0.12 &  24 & 0.67 & 0.12 & 7.68 & 3614.96 & 18.93 & 7200.25 & 27.27 & 34.03 \\ \midrule
  500 & 0.02 &  10 & 0.00 & 0.19 & 0.00 & 16.71 & 11.24 & 184.29 & 0.00 & 0.00 \\ 
  500 & 0.04 &  20 & 0.00 & 0.20 & 0.08 & 1640.78 & 28.93 & 7200.21 & 0.48 & 3.06 \\ 
  500 & 0.06 &  30 & 0.02 & 0.68 & 3.57 & 3630.19 & 65.75 & 7200.26 & 8.10 & 11.89 \\ 
  500 & 0.08 &  40 & 0.00 & 0.48 & 1.05 & 3623.51 & 115.36 & 7200.57 & 14.26 & 14.82 \\ 
  500 & 0.10 &  50 & 0.00 & 1.49 & 6.88 & 3625.89 & 177.26 & 7200.18 & 18.26 & 23.74 \\ 
  500 & 0.12 &  60 & 0.69 & 0.95 & 4.11 & 3631.14 & 223.61 & 7200.38 & 22.94 & 25.30 \\ \midrule
  1000 & 0.02 &  20 & 0.00 & 0.29 & 0.00 & 18.83 & 47.62 & 156.49 & 0.00 & 0.00 \\ 
  1000 & 0.04 &  40 & 0.00 & 1.65 & 2.76 & 2823.19 & 183.33 & 7200.18 & 1.05 & 4.10 \\ 
  1000 & 0.06 &  60 & 0.00 & 2.48 & 6.61 & 3671.12 & 397.35 & 7200.18 & 6.16 & 12.53 \\ 
  1000 & 0.08 &  80 & 0.00 & 2.92 & 7.86 & 3662.76 & 687.02 & 7200.44 & 10.07 & 17.39 \\ 
  1000 & 0.10 & 100 & 0.00 & 2.31 & 8.65 & 3670.95 & 1052.76 & 7200.17 & 13.60 & 20.99 \\ 
  1000 & 0.12 & 120 & 0.00 & 2.13 & 5.92 & 3696.24 & 1552.76 & 7200.22 & 15.77 & 20.73 \\ \midrule
  2000 & 0.02 &  40 & 0.00 & 1.36 & 0.00 & 14.55 & 328.25 & 70.02 & 0.00 & 0.00 \\ 
  2000 & 0.04 &  80 & 0.00 & 3.49 & 0.00 & 1774.85 & 1259.21 & 1057.28 & 0.21 & 0.00 \\ 
  2000 & 0.06 & 120 & 0.00 & 4.26 & 21.38 & 3774.67 & 2578.07 & 7200.16 & 2.42 & 23.29 \\ 
  2000 & 0.08 & 160 & 0.00 & 4.63 & 100.00 & 3806.42 & 4431.39 & 7200.28 & 5.21 & 100.00 \\ 
  2000 & 0.10 & 200 & 0.00 & 5.23 & 100.00 & 3925.28 & 6435.67 & 7200.17 & 7.16 & 100.00 \\ 
  2000 & 0.12 & 240 & 0.74 & 0.94 & 100.00 & 4319.62 & 10949.83 & 7200.16 & 11.91 & 100.00 \\ \midrule
  3000 & 0.02 &  60 & 0.00 & 1.97 & 0.00 & 16.16 & 923.12 & 32.40 & 0.00 & 0.00 \\ 
  3000 & 0.04 & 120 & 0.00 & 4.03 & 0.00 & 1883.80 & 3365.35 & 1541.40 & 0.08 & 0.00 \\ 
  3000 & 0.06 & 180 & 0.00 & 5.26 & 0.04 & 4144.41 & 7620.89 & 7200.16 & 1.81 & 1.80 \\ 
  3000 & 0.08 & 240 & 0.00 & 4.55 & 99.98 & 4554.74 & 13624.07 & 7209.41 & 4.23 & 99.99 \\ 
  3000 & 0.10 & 300 & 0.00 & 4.04 & 99.98 & 5013.31 & 31137.18 & 7200.38 & 6.17 & 99.98 \\ 
  3000 & 0.12 & 360 & 0.32 & 1.12 & 99.98 & 6999.60 & 43173.66 & 7216.75 & 9.83 & 99.99 \\ \bottomrule
\end{tabular}
\caption{Comparison of the Benders decomposition approach, the D\&C heuristic and direct solution of \SplitMIO in terms of solution quality, computation time and optimality gap. 
\label{table:synthetic_B_T6} }
\end{table}

Comparing the performance of the Benders approach with the D\&C heuristic, we can see that in general, the Benders approach is able to find better solutions than the D\&C heuristic. In particular, for larger instances, $G_{\Benders}$ is lower than $G_{\DNC}$ (for example, with $n = 2000$, $\mu = 0.08$, the Benders solution is in general the best one, and the D\&C solution has an expected revenue that is 4.6\% worse). In addition, from a computation time standpoint, the Benders approach compares quite favorably to the D\&C heuristic. While the D\&C heuristic is faster for small problems with low $n$ and/or low $\mu$, it can require a significant amount of time for $n = 2000$ or $n = 3000$. In addition to this comparison against the D\&C heuristic, in Appendix~\ref{subsec:synthetic_T_vs_H}, we also provide a comparison of the MIO solutions for the smaller T1, T2 and T3 instances used in the previous two sections against three other heuristic solutions; in those instances, we similarly find that solutions obtained from our MIO formulations can be significantly better than heuristic solutions. 

Comparing the performance of the Benders approach with the direct solution approach, our results indicate two types of behavior. The first type of behavior corresponds to ``easy'' instances. These are instances with $\mu \in \{0.02, 0.04\}$ for which it is sometimes possible to directly solve \SplitMIO to optimality within the two hour time limit. For example, with $n = 2000$ and $\mu = 0.04$, all five instances are solved to optimality by the direct approach. For those instances, the Benders approach is either able to prove optimality (for example, for $n = 200$ and $\mu = 0.04$, $O_{\Benders} = 0\%$) or terminate with a low optimality gap (for example, for $n = 3000$ and $\mu = 0.04$, $O_{\Benders} = 0.08\%$); among all instances with $\mu \in \{0.02, 0.04\}$, the average optimality gap is no more than about 1.05\%. More importantly, the solution obtained by the Benders approach is at least as good as the solution obtained after two hours of direct solution of \SplitMIO, which can be seen from the fact that $G_{\Benders}$ is 0.0\% in all of these $(n,\mu)$ combinations.

The second type of behavior corresponds to ``hard'' instances, which are the instances with $\mu \in \{0.06, 0.08, 0.10, 0.12\}$. For these instances, Gurobi generally struggles to solve the LO relaxation of \SplitMIO within the two hour time limit. When this happens, the integer solution returned by Gurobi is obtained from applying heuristics before solving the root node of the branch-and-bound tree, which is often quite suboptimal. (This often results in integer solutions with an objective value of 0.0, leading to values of $O_{\Direct}$ and $G_{\Direct}$ that are close to 100\%.) In contrast, the Benders approach delivers significantly better solutions; among all of these instances, $G_{\Benders}$ is within 1\%, indicating that on average the Benders approach is within 1\% of the best solution of each instance. In addition, $O_{\Benders}$ is generally lower than $O_{\Direct}$, indicating that the Benders approach in general makes more progress towards proving optimality than the direct approach. 

Overall, these results suggest that our Benders approach can deliver high quality solutions to large-scale instances of the assortment optimization problem in a reasonable computational timeframe that are at least as good, and often significantly better, than those obtained by the D\&C heuristic or those obtained by directly solving the problem using Gurobi.

\section{Numerical Experiments with Real Transaction Data}
\label{sec:IRI}

We examine the performance of the optimal assortment generated by the decision forest model on problem instances calibrated with a real-world transaction dataset. Different from Section~\ref{sec:synthetic}, the problem instances in this section are of a smaller scale, which help us glean operational insights.

\subsection{Background}
\label{subsec:IRI_background}
We consider the IRI Academic Dataset \citep{bronnenberg2008database}, which is comprised of real-world transaction records of store sales for thirty product categories from forty-seven U.S. markets. The same dataset has been used by \cite{jagabathula2019limit} to empirically demonstrate the loss of rationality in consumer choice and by \cite{chen2019decision} to evaluate the predictive performance of the decision forest model. 

We follow the literature to pre-process the raw transaction data. We first aggregate items with the same vendor code as a product, a common pre-processing technique in the marketing science community \citep{bronnenberg2004market,nijs2007retail}. Following the setup in the literature \citep{jagabathula2019limit,chen2019decision} and focusing on the data from the first two weeks of the calendar year 2007 due to the data volume, we select the top nine purchased items as the products and combine the remaining items as the outside/no-purchase option. We further transform the transactions into assortment-choice pairs as follows. We first call $\Tcal$ the set of transactions. For each transaction $\tau \in \Tcal$, we have the following information: the week of the purchase $\tau_\text{week}$, the store where the transaction happened $\tau_\text{store}$, the sold product $\tau_{\text{prod}}$, and the selling price $\tau_{\text{price}}$. Let $\Wcal$ and $\Zcal$ be the nonrepeated collection of $\{ \tau_\text{week} \}_{\tau \in \Tcal}$ and $\{ \tau_\text{store} \}_{\tau \in \Tcal}$, respectively. For each week $w \in \Wcal$ and store $s \in \Zcal$, we define $S_{w,s} =  \bigcup_{\tau \in \Tcal}  \{  \tau_\text{prod} \mid \tau_\text{week} = w , \tau_\text{store} = s   \}$ as the set of products that was purchased at least once at store $s$ during week $w$. The set of transactions $\Tcal$ is thus transformed into the set of assortment-choice pairs as $\{ \left( S_{\tau_\text{week},\tau_\text{store}} , \tau_\text{prod} \right) \}_{\tau \in \Tcal}$, which will be used for choice model estimation. We further define the marginal revenue $\rho_i$ as the average of the historical prices $ \left(   \tau_{\text{price}} \right)_{\tau \in \Tcal: \tau_{\text{prod}} = i}$ for $i \in \Ncal$.

\subsection{Results: Improvement in Expected Revenue}
\label{subsec:IRI_improvement}

We estimate both the ranking-based model and the decision forest model from the assortment-choice pairs $\{ \left( S_{t_\text{week},t_\text{store}} , t_\text{prod} \right) \}_{t \in \Tcal}$ in each product category by maximum likelihood estimation. For details on estimating the ranking-based model, see \cite{van2014market,van2017expectation}, and for the decision forest model, refer to \cite{chen2019decision}. In particular, we follow \cite{chen2019decision} to warm start the solver by setting the initial solution as the estimated ranking-based model.  For simplicity, we set the tree depth limit as three (i.e., leaves can have depth at most four). Typically, the tree depth is determined by cross validation. \cite{chen2019decision} have reported that trees of depth three lead to good predictive performance, while deeper trees tend to overfit.

We further denote $S^{\text{RM}}$ and $S^{\text{DF}}$ as the optimal assortments under the estimated ranking-based model and the decision forest model, respectively. Note that we obtain $S^{\text{RM}}$ using an integer programming approach \citep{bertsimas2019exact,feldman2019assortment}. We use $Z^*_{{\Ccal}}$ to denote the maximal expected revenue under a given choice model $\Ccal$ and $Z_{\Ccal}(S)$ to denote the expected revenue of assortment $S$. Obviously, $Z^*_{\text{RM}} = Z_{\text{RM}}(S^{\text{RM}})$ and $Z^*_{\text{DF}} = Z_{\text{DF}}(S^{\text{DF}})$. To discuss the relative performance of assortments $S^{\text{RM}}$ and $S^{\text{DF}}$, we define the following two metrics: $
RI_{\text{DF}} = 100 \% \times  \frac{Z_{\text{DF}}(S^{\text{DF}}) - Z_{\text{DF}}(S^{\text{RM}})}{Z_{\text{DF}}(S^{\text{RM}})}$ and $RI_{\text{RM}} = 100 \% \times  \frac{Z_{\text{RM}}(S^{\text{RM}}) - Z_{\text{RM}}(S^{\text{DF}})}{Z_{\text{RM}}(S^{\text{DF}})}$. 
The two metrics measure the relative improvement of the optimal assortment $S^{\text{DF}}$ (or $S^{\text{RM}}$) from $S^{\text{RM}}$ (or $S^{\text{DF}}$) under the decision forest model (or the ranking-based model). Note that both metrics are non-negative by their definitions. We also define the Hamming distance $\Delta_{\text{H}} =  \sum_{i=1}^{n} |  \mathbb{I} \left[  i \in  S^{\text{DF}}  \right] - \mathbb{I} \left[  i \in  S^{\text{RM}}  \right]  |$ to measure how different the two assortments are. 

Table~\ref{table:IRI} summarizes the comparison of $S^{\text{DF}}$ and $S^{\text{RM}}$ in terms of expected revenue under the measures $RI_{\text{DF}} $ and $RI_{\text{RM}} $. When the ground truth is the estimated decision forest model, the assortment $S^{\text{DF}}$ can outperform $S^{\text{RM}}$ up to $32\%$ and with $7\%$ improvement on average in the expected revenue.  Meanwhile, when the ground truth is the estimated ranking-based model, the assortment $S^{\text{DF}}$ only performs $3\%$ worse than the optimal assortment $S^{\text{RM}} $ on average. Recall that the class of the ranking-based models is equivalent to the class of choice models of random utility maximization (RUM) principle, or, the class of rational choice models \citep{jagabathula2019limit}. Table~\ref{table:IRI} suggests that the optimal assortment generated by the decision forest model can be quite beneficial when customer choice deviates from the rational choice theory, and does not lose much even if customers are strictly rational. We also remark that on average, the two assortments $S^{\text{DF}}$ and $S^{\text{RM}}$ only differ from each other with Hamming distance $1.9$, while they have similar sizes: the average sizes of $S^{\text{DF}}$ and  $S^{\text{RM}}$ are 4.9 and 5.3, respectively. 

Note that, in this real-data setting, the problem instances involve significantly fewer products compared to the synthetic-data setting described in Section~\ref{sec:synthetic}. As a result, the optimal assortments $S^{\text{DF}}$ and $S^{\text{RM}}$ are obtained almost instantaneously (in less than one second) using the integer programming approach.

{
	\begin{table}
		\centering
		\footnotesize
		\begin{tabular}{lrrrllrrr}
			\toprule
			Category             & $RI_\text{DF}$ & $RI_\text{RM}$ & $\Delta_{\text{H}}$ &  & Category                  & $RI_\text{DF}$ & $RI_\text{RM}$ & $\Delta_{\text{H}}$ \\ \cmidrule{1-4} \cmidrule{6-9}
			Beer                 & 32.8           & 5.9            & 3                   &  & Mayonnaise                & 0.0            & 0.0            & 0                   \\
			Blades               & 0.0            & 0.0            & 0                   &  & Milk                      & 5.2            & 0.7            & 1                   \\
			Carbonated Beverages & 4.8            & 3.8            & 4                   &  & Mustard / Ketchup         & 5.6            & 1.1            & 1                   \\
			Cigarettes           & 2.5            & 8.3            & 1                   &  & Paper Towels              & 1.0            & 2.4            & 2                   \\
			Coffee               & 26.3           & 11.0           & 4                   &  & Peanut Butter             & 10.9           & 5.6            & 3                   \\
			Cold Cereal          & 6.1            & 0.3            & 2                   &  & Photo                     & 0.0            & 0.0            & 0                   \\
			Deodorant            & 3.7            & 2.2            & 3                   &  & Salty Snacks              & 12.0           & 3.7            & 2                   \\
			Diapers              & 0.0            & 0.0            & 0                   &  & Shampoo                   & 19.0           & 1.4            & 4                   \\
			Facial Tissue        & 0.2            & 0.6            & 1                   &  & Soup                      & 13.0           & 5.0            & 3                   \\
			Frozen Dinners       & 2.7            & 5.2            & 2                   &  & Spaghetti / Italian Sauce & 4.9            & 3.3            & 2                   \\
			Frozen Pizza         & 10.9           & 4.0            & 4                   &  & Sugar Substitutes         & 10.7           & 6.0            & 1                   \\
			Household Cleaners   & 16.0           & 6.6            & 4                   &  & Toilet Tissue             & 0.0            & 0.0            & 0                   \\
			Hotdogs              & 18.7           & 1.4            & 4                   &  & Toothbrush                & 0.0            & 0.0            & 0                   \\
			Laundry Detergent    & 0.1            & 1.9            & 1                   &  & Toothpaste                & 2.7            & 0.9            & 2                   \\
			Margarine / Butter   & 5.3            & 12.9           & 1                   &  & Yogurt                    & 3.8            & 2.1            & 1                   \\ \midrule 
			&                &                &                     &  & Average over all categories                 & 7.3            & 3.2            & 1.9    \\  \bottomrule         
		\end{tabular}
		\caption{The relative performance of the assortments generated by the decision forest model and the ranking-based model in each product category in the IRI Academic Dataset}
		\label{table:IRI}
	\end{table}
}

\subsection{Case Study: Beer Category}
\label{subsec:beer}

To investigate why we observed high relative improvement with the assortment $S^{\text{DF}}$ over $S^{\text{RM}}$ in some categories, we look into the Beer category, where the relative improvement is 32.8\%.  

We first observe that the consumer choice within the Beer category exhibits a notable phenomenon known as \emph{choice overload}, a well-documented behavioral anomaly in the field of marketing science \citep{iyengar2000choice,chernev2015choice,long2023choice}. This phenomenon describes the situation in which an abundance of available products can actually deter customers from making a purchase. Choice overload \YCRR{contradicts} the principles of rational choice theory. According to this theory, when a store increases the assortment size, a rational consumer should theoretically be more inclined to make a purchase from the assortment. This is because a larger assortment would increase the chance that the customer finds one product that aligns with her preference. In the context of choice modeling, a rational choice model satisfies the following:
\begin{align}
	\label{eq:no_purchase_regularity}
	P( 0 \mid S) \geq P(0 \mid S') \quad \text{ whenever } \quad S \subseteq S'.
\end{align}
Note that \YCRR{Condition~\eqref{eq:no_purchase_regularity}} directly stems from the \emph{regularity property}, which the ranking-based model and other RUM choice models, such as the mixed-MNL model, always adhere to \citep{rieskamp2006extending,jagabathula2019limit}. However, real-world scenarios often deviate from \YCRR{Condition~\eqref{eq:no_purchase_regularity}}. When provided with more options, customers may become fatigued from the search process or experience a decrease in their confidence in decision-making, leading them to opt not to make a purchase. In the consumer psychology literature, \cite{iyengar2000choice} conducted a field experiment demonstrating that individuals are more likely to purchase gourmet jams or chocolates when presented with a smaller assortment of size six rather than a more extensive assortment of size 24 or 30.

\begin{figure}
	\begin{center}
		\begin{subfigure}[t]{0.47\textwidth}
			\centering
			\includegraphics[width=1\textwidth]{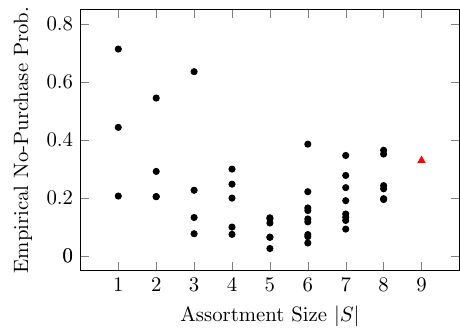}
		\end{subfigure}
		\qquad
		\begin{subfigure}[t]{0.45\textwidth}
			\centering
			\includegraphics{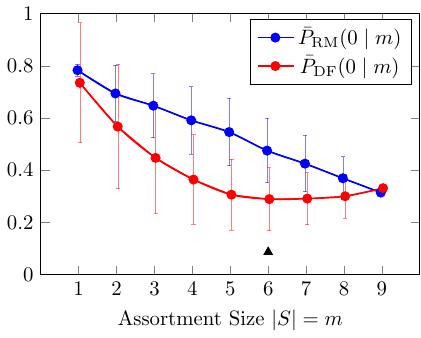}
		\end{subfigure}
	\end{center}
	\caption{(left) The choice overload effect in the Beer category. Each scatter point represents the empirical no-purchase probability $\bar{P}( 0 \mid S)$ for an historical assortment $S \in \mathcal{S}_\text{hist}$; (right) The average no-purchase probability of the estimated ranking-based and decision forest models (with the error bars as the standard deviation). The black triangle represents that $P_\text{DF}(0 \mid S^\text{DF}) = 8.6\%$
	} \label{figure:beer_choice_overload_two_figures}
\end{figure}

Figure~\ref{figure:beer_choice_overload_two_figures} (left) visually depicts the prevalence of choice overload within the Beer category of the IRI dataset. Each data point in the figure corresponds to a pair ($|S|$, $\bar{P}(0\mid S)$), where $|S|$ represents the assortment size of a historical assortment $S \in \Scal_\text{hist} \equiv \{ S_{w,s} \mid w \in \Wcal, s \in \Scal \}$, and $\bar{P}(0\mid S)$ represents the empirical probability of the no-purchase option being chosen when the assortment $S$ was offered. For a better illustration, we have excluded assortments with very few transactions (less than or equal to 10) from Figure~\ref{figure:beer_choice_overload_two_figures} (left). In this figure, the rightmost data point, denoted by a red triangle, corresponds to the no-purchase probability when the assortment contains all available products ($|S| = 9$). It is important to note that in a scenario where \YCRR{Condition~\eqref{eq:no_purchase_regularity}} holds, indicating the absence of choice overload, each data point should exhibit a higher $y$-value than the red triangle. However, this condition is met for only 8 out of the 41 data points in the figure, suggesting a violation of \YCRR{Condition~\eqref{eq:no_purchase_regularity}} within the Beer category.

The presence of choice overload in the data has a significant impact on the estimation outcomes of both the ranking-based model and the decision forest model, leading to divergent assortment decisions. To begin, we define $\bar{P}_\Ccal (0 \mid m)$ as the average no-purchase probability when presented with an assortment of size $m$; formally, %
$\bar{P}_\Ccal (0 \mid m) = { \sum_{S \subseteq \Ncal \,: \, |S| = m} P_{\Ccal} ( 0 \mid S)   } \big\slash { | \{  S \subseteq \Ncal \, :\, | S | = m  \}    |  }$, where $P_\Ccal( 0 \mid S)$ represents the no-purchase probability given assortment $S$ under a choice model $\Ccal$. We consider two choice models, the decision forest model (DF) and the ranking-based model (RM), both estimated from data. Figure~\ref{figure:beer_choice_overload_two_figures} (right) presents the average no-purchase probability curves, $\bar{P}_\Ccal (0 \mid m)$, for these two models, with the error bars representing the standard deviation. Recall that the ranking-based model adheres to the regularity property and strictly follows \YCRR{Condition~\eqref{eq:no_purchase_regularity}}. Consequently, the resulting no-purchase probability curve (depicted in blue) does not exhibit the choice overload effect; it consistently decreases as we expand the assortment size. In contrast, the decision forest model is not bound by the regularity property and has the capacity to learn the choice overload effect from the data. The resulting no-purchase probability curve (shown in red) indicates that customers are indeed more inclined to make a purchase as the assortment initially expands. However, after reaching an assortment size of $|S| = 6$, the choice overload effect becomes evident. Customers become less motivated, and we observe an increase in the no-purchase probability when $|S| \geq 7$.

The extent to which the decision forest model and the ranking-based model capture the phenomenon of choice overload from the data significantly influences their downstream assortment decisions. It is important to note that within the dataset, the marginal revenues $\rho_i$ of the products exhibit low variation, ranging between $8.14$ and $10.83$ USD (keeping in mind that beer is often sold in packs of six). Since prices are quite uniform and well above zero, the optimal strategy is \emph{not} about making customers choose a specific product but rather about making them buy \emph{any} product from the offered assortment, i.e., reducing the no-purchase probability.

For the ranking-based model, the optimal assortment $S^\text{RM}$ includes all products, resulting in $|S^\text{RM}| = 9$, since this is the only way it can minimize the no-purchase probability. In contrast, the decision forest model, having captured the choice overload effect from the data, takes a different approach. It examines assortments of size six and identifies one that decreases the no-purchase probability from ${P}_{\text{DF}}( 0 \mid S^\text{RM}) = 33.1\%$ to ${P}_{\text{DF}}( 0 \mid S^\text{DF}) = 8.6\%$ (see the black triangle in Figure~\ref{figure:beer_choice_overload_two_figures} (right)). This strategic choice results in a remarkable 32\% improvement in expected revenue, measured according to the decision forest model. This highlights the advantage of the assortment planning approach proposed in this paper: by leveraging the decision forest model's ability to capture consumer choice patterns, businesses can tailor their product offerings to effectively capitalize on consumers' departures from strictly rational purchasing behavior.

\YCRR{
\subsection{Additional Numerical Experiments based on the Sushi Dataset}
\label{subsec:sushi_dataset_intro_mainbody}

The IRI dataset provides valuable real-world transaction data, but it also comes with limitations. First, it does not specify a ground-truth choice model, so the evaluation of assortment performance must rely on comparisons across models, as in Table~\ref{table:IRI}. Second, even after standard preprocessing steps from the literature, the historical assortments and no-purchase options are not always fully consistent across all assortment–choice pairs. To address these limitations, we complement the IRI-based analysis with a semi-synthetic experiment using the Sushi dataset \citep{kamishima2003nantonac}, presented in Section~\ref{appendix:additional_numerical_results_real_data}.
}

\YCRR{In this experiment, we construct a ground-truth choice model over ten sushi types ($n=10$) that combines (i) the ranking-based preferences in the Sushi dataset and (ii) a synthetic choice overload effect, where customers increasingly opt for no purchase as assortments expand. We use this ground truth to generate synthetic assortment–choice data, from which we then estimate the MNL, ranking-based, and decision forest models. Let $S^{\text{MNL}}$, $S^{\text{RB}}$, and $S^{\text{DF}}$ denote the optimal assortments recommended by the respective estimated models. We then evaluate the expected revenues of these assortments under the synthetic ground truth.}

\YCRR{The results show that the decision forest model ($S^{\text{DF}}$) consistently delivers strong revenue performance, even when the choice overload effect is intensified. By contrast, the revenues of $S^{\text{MNL}}$ and $S^{\text{RB}}$ degrade significantly under stronger overload. When the overload effect is mild, the ranking-based model performs best---as expected, since the ground truth is nearly ranking-based in this case. However, as overload increases, the decision forest model significantly outperforms both benchmarks.}

\YCRR{This semi-synthetic experiment highlights the end-to-end (data-to-decision) performance of the decision forest model in the presence of a specific behavioral anomaly. At the same time, it isolates only one anomaly---choice overload---whereas real-world settings often exhibit multiple behavioral patterns simultaneously, such as the decoy effect \citep{huber1982adding} and compromise effect \citep{simonson1989choice}. In such cases, the IRI dataset study in Section~\ref{subsec:IRI_improvement} provides a broader perspective on how the decision forest model compares with rational-choice models such as the ranking-based model. Taken together, the IRI and Sushi experiments complement each other: the former grounds our analysis in real data with rich behavioral heterogeneity, while the latter offers a controlled environment with a known ground truth. Together, they provide a comprehensive assessment of the decision forest model.}

\section{Conclusions}
\label{sec:conclusions}

We developed a mixed-integer optimization methodology for solving the assortment optimization problem under the decision forest model, which accounts for both rational and non-rational customer behaviors. We first established the inapproximability of the problem and then presented two mixed-integer programs. To enhance scalability, we proposed a large-scale solution approach based on Benders decomposition. Through synthetic data experiments, we demonstrated the efficiency of our methods, and with real-world data, we showed how consumer behavioral anomalies can be learned and leveraged within our framework.

\section*{Acknowledgments}
We sincerely thank the department editor Huseyin Topaloglu, the associate editor and two anonymous reviewers for their thoughtful comments that have helped to significantly improve the paper. The authors gratefully acknowledge Information Resources Inc. (IRI) and the authors of \cite{bronnenberg2008database} for the IRI Academic Data Set that was used in Section~\ref{sec:IRI}. %
This work has benefited from the generous research support provided by the UCL School of Management and the UCLA Anderson School of Management.

\makeatletter
\newcommand*\mysize{%
  \@setfontsize\mysize{10.0}{9.0}%
}
\makeatother

\renewcommand{\bibfont}{\mysize}
{\setlength{\bibsep}{3pt}
\bibliographystyle{plainnat}
\bibliography{aodf_literature.bib}
}

\newpage

\begin{center}
  {\large Electronic Companion:\\ Assortment Optimization under the Decision Forest Model}\\
  (Yi-Chun Akchen and Velibor V. Mi\v{s}i\'{c})
\end{center}

\vspace{0.5em}

This electronic companion covers Sections~\ref{subsec-appendix:greedy-implementation}--\ref{appendix:additional_numerical_results_real_data}, which contain additional numerical analyses and implementation details. Theoretical results, including omitted proofs (Section~\ref{sec:proofs}) and further analysis (Section~\ref{sec:one_tree_analysis}), are not included in this electronic companion due to page limits. These materials are instead available at \url{https://arxiv.org/abs/2103.14067}.

\vspace{0.5em}

\begin{APPENDICES}

\section{Solving the \SplitMIO Relaxation via a Greedy Algorithm}
\label{subsec-appendix:greedy-implementation}

As noted in Section~\ref{subsec:benders_SplitMIO_relaxation}, each primal subproblem~\eqref{prob:SplitMIO_sub_primal} of the \SplitMIO relaxation can be solved using a greedy algorithm. While the algorithm was briefly outlined in Section~\ref{subsec:benders_SplitMIO_relaxation}, we provide additional implementation details here, particularly regarding the events $A_s$, $B_s$, and $C$, through the following pseudocode (Algorthm~\ref{algorithm:SplitMIO_primal_greedy}).

First,  we define an ordering $\tau$ of the leaves in nondecreasing revenue. In particular, we require a bijection $\tau: \{1,\dots, |\leaves(t)|\} \to \leaves(t)$ such that $r_{t,\tau(1)} \geq r_{t,\tau(2)} \geq \dots \geq r_{t,\tau(|\leaves(t)| \YCR{ )} }$, i.e., an ordering of leaves in nondecreasing revenue. In addition, in the definition of Algorithm~\ref{algorithm:SplitMIO_primal_greedy}, we use $\LS(\ell)$ and $\RS(\ell)$ to denote the sets of left and right splits of $\ell$, respectively, which are defined as
\begin{align*}
\LS(\ell)  = \{ s \in \splits(t) \mid \ell \in \leftleaves(s) \} \quad \text{and} \quad \RS(\ell)  = \{ s \in \splits(t) \mid \ell \in \rightleaves(s) \}.
\end{align*}
In words, $\LS(\ell)$ is the set of splits for which we must proceed to the left in order to be able to reach $\ell$, and $\RS(\ell)$ is the set of splits for which we must proceed to the right to reach $\ell$. A split $s \in \LS(\ell)$ if and only if $\ell \in \leftleaves(s)$, and similarly, $s \in \RS(\ell)$ if and only if $\ell \in \rightleaves(s)$.

The algorithm processes the leaves in decreasing order of revenue and sets the variable $y_{t,\ell}$ for each leaf $\ell$ to the largest value that does not violate the left and right split constraints~\eqref{prob:SplitMIO_sub_primal_left} and \eqref{prob:SplitMIO_sub_primal_right}, nor the constraint $\sum_{\ell \in \leaves(t)} y_{t,\ell} \leq 1$. At each iteration, the algorithm also records which constraint becomes tight by maintaining the event set $\Ecal$. An $A_s$ event indicates that the left split constraint~\eqref{prob:SplitMIO_sub_primal_left} for split $s$ has become tight; a $B_s$ event indicates that the right split constraint~\eqref{prob:SplitMIO_sub_primal_right} for split $s$ has become tight; and a $C$ event indicates that the constraint $\sum_{\ell \in \leaves(t)} y_{t,\ell} \leq 1$ has become tight. If a $C$ event is not triggered, Algorithm~\ref{algorithm:SplitMIO_primal_greedy} identifies the split with the smallest remaining capacity (line~15). If the $\arg\min$ is not unique and multiple splits are tied, ties are broken by selecting the split $s$ with the smallest depth $d(s)$ (i.e., the split closest to the root node of the tree).

The function $f$ records which leaf $\ell$ is being processed when an $A_s$, $B_s$, or $C$ event occurs. Although neither $\Ecal$ nor $f$ is required to compute the primal solution, they are essential for determining the dual solution in the dual procedure (Algorithm~\ref{algorithm:SplitMIO_dual_greedy}).

\begin{algorithm}
\SingleSpacedXI
\small
\begin{algorithmic}[1]
\REQUIRE Bijection $\tau: \{1,\dots, | \leaves(t)| \} \to \leaves(t)$ such that $r_{t, \tau(1)} \geq r_{t,\tau(2)} \geq \dots \geq r_{t, \tau( |\leaves(t)|)}$
\STATE Initialize $y_{t,\ell} \gets 0$ for each $\ell \in \leaves(t)$. 
\FOR{ $i = 1, \dots, |\leaves(t)|$}
	\STATE Set $q_C \gets 1 - \sum_{j=1}^{i-1} y_{t, \tau(j)}$.
	\FOR{$s \in \LS(\tau(i))$}
		\STATE Set $q_s \gets x_{v(t,s)} - \sum_{1 \leq j \leq i-1: \tau(j) \in \leftleaves(s) } y_{t, \tau(j)} $
	\ENDFOR
	\FOR{$s \in \RS(\tau(i))$}
		\STATE Set $q_s \gets 1 - x_{v(t,s)} -  \sum_{1 \leq j \leq i-1 :  \tau(j) \in \rightleaves(s)} y_{t, \tau(j)} $
	\ENDFOR
	\STATE Set $q_{A,B} \gets \min_{s \in \LS(\tau(i)) \cup \RS(\tau(i)) } q_s$
	\STATE Set $q^* \gets \min\{ q_C, q_{A,B} \}$
	\STATE Set $y_{t,\tau(i)} \gets q^*$
	\IF{ $q^* = q_{C}$ }
		\STATE Set $\Ecal \gets \Ecal \cup \{ C \}$.
		\STATE Set $f(C) = \tau(i)$.
	\ELSE
		\STATE Set $s^* \gets \arg \min_{s \in \LS(\tau(i)) \cup \RS(\tau(i))} q_s$
		\IF{ $s^* \in \LS(\tau(i))$ }
			\STATE Set $e = A_{s^*}$
		\ELSE
			\STATE Set $e = B_{s^*}$
		\ENDIF
		\IF{$e \notin \Ecal$}
			\STATE Set $\Ecal \gets \Ecal \cup \{e\}$.
			\STATE Set $f(e) = \tau(i)$.
		\ENDIF
	\ENDIF
\ENDFOR
\end{algorithmic}

\caption{Primal greedy algorithm for \SplitMIO. 
\label{algorithm:SplitMIO_primal_greedy}
}
\end{algorithm}

\section{Example of Benders algorithms for \SplitMIO}
\label{appendix:SplitMIO_benders_example}

In this section, we provide a small example of the primal-dual procedure (Algorithms~\ref{algorithm:SplitMIO_primal_greedy} and \ref{algorithm:SplitMIO_dual_greedy}) for solving the \SplitMIO subproblem. Suppose that $n = 6$, and that $\xb = (x_1, \dots, x_6) = (0.62, 0.45, 0.32, 0.86, 0.05, 0.35)$. Suppose that $\boldsymbol{\rho} = (\rho_1, \dots, \rho_6) = (97, 72, 89, 50, 100, 68)$. Suppose that the purchase decision tree $t$ has the form given in Figure~\ref{figure:SplitMIO_primal_dual_tree_example_products}; in addition, suppose that the splits and leaves are indexed as in Figure~\ref{figure:SplitMIO_primal_dual_tree_example_enumeration}. For example, in the latter case, 8 corresponds to the split node that is furthest to the bottom and to the left, while 30 corresponds to the second leaf from the right. 

\begin{figure}[h!]
	\centering
	\begin{subfigure}[t]{0.44\textwidth}
		\includegraphics[width=1\textwidth]{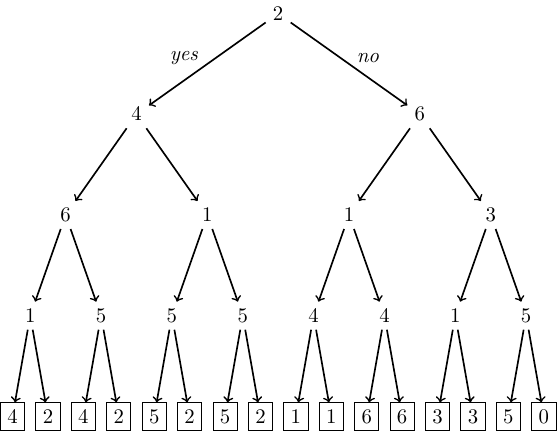}
		\caption{Purchase decision tree. \label{figure:SplitMIO_primal_dual_tree_example_products} }
	\end{subfigure}
	\hspace{0.05\textwidth}
	\begin{subfigure}[t]{0.48\textwidth}
		\includegraphics[width=1\textwidth]{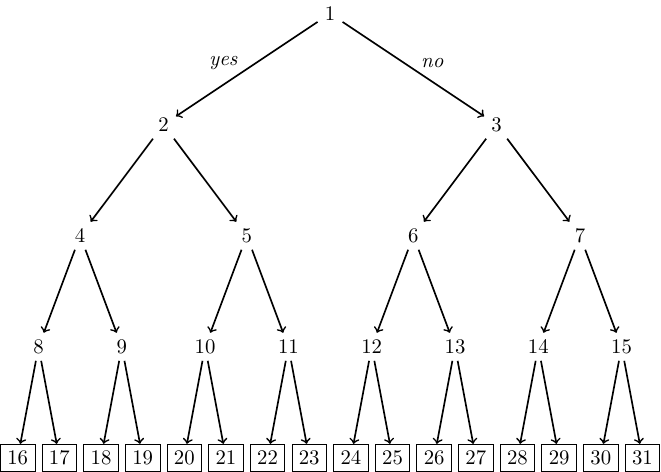}
		\caption{Indexing of nodes in tree. \label{figure:SplitMIO_primal_dual_tree_example_enumeration}}
	\end{subfigure}
	
	\caption{Tree used in example of \SplitMIO primal-dual algorithms. \label{figure:SplitMIO_primal_dual_tree_example}}
	
\end{figure}

We first run Algorithm~\ref{algorithm:SplitMIO_primal_greedy} on the problem, which carries out the steps shown below in Table~\ref{table:SplitMIO_primal_greedy_example}. For this execution of the procedure, we assume that the following ordering of leaves (encoded by $\tau)$ is used:
\begin{equation*}
	20, 22, 30, 24, 25, 28, 29, 17, 19, 21, 23, 26, 27, 16, 18, 31.
\end{equation*}

\begin{table}
	\centering
	\makebox[0pt][c]{\parbox{1.00\textwidth}{%
			\begin{minipage}[b]{0.495\hsize}
				\centering
				\footnotesize
				\begin{tabular}{lll} \toprule
					Iteration & Values of $q_{C}$ and $q_{A,B}$ & Steps \\ \midrule
					$\ell = 20$ & $q_C = 1.0$, $q_{A,B} = 0.05$ & Set $y_{20} \gets 0.05$ \\
					& & $A_{10}$ event \\[0.5em]
					$\ell = 22$ & $q_C = 0.95$, $q_{A,B} = 0.05$ & Set $y_{22} \gets 0.05$ \\
					& & $A_{11}$ event \\[0.5em]
					$\ell = 30$ & $q_C = 0.90$, $q_{A,B} = 0.05$ & Set $y_{30} \gets 0.05$ \\
					& & $A_{15}$ event \\[0.5em]
					$\ell = 24$ & $q_C = 0.85$, $q_{A,B} = 0.35$ & Set $y_{24} \gets 0.35$ \\
					& & $A_{3}$ event \\[0.5em]
					$\ell = 25$ & $q_C = 0.5$, $q_{A,B} = 0.0$ & Set $y_{25} \gets 0.0$ \\[0.5em]
					$\ell = 28$ & $q_C = 0.5$, $q_{A,B} = 0.15$ & Set $y_{28} \gets 0.15$ \\
					& & $B_{1}$ event \\[0.5em]
					$\ell = 29$ & $q_C = 0.35$, $q_{A,B} = 0.0$ & Set $y_{29} \gets 0.0$ \\[0.5em]
					$\ell = 17$ & $q_C = 0.35$, $q_{A,B} = 0.35$ & Set $y_{17} \gets 0.35$ \\
					& & $C$ event \\
					& & \textbf{break} \\ \bottomrule
				\end{tabular}
				\caption{Steps of primal procedure (Algorithm~\ref{algorithm:SplitMIO_primal_greedy}). \label{table:SplitMIO_primal_greedy_example} }
			\end{minipage}
			\begin{minipage}[b]{0.495\hsize}
				\centering
				\footnotesize
				\begin{tabular}{ll} \toprule
					Phase & Calculation \\ \midrule
					Initialization & $\alpha_s \gets 0$, $\beta_s \gets 0$ for all $s$ \\
					Set $\gamma$ & $\gamma \gets r_{17} = 72$ \\
					Loop: $d = 1$ & $\beta_{1} \gets r_{28} - \gamma = 89 - 72 = 17 $ \\
					Loop: $d = 2$ & $\alpha_{3} \gets r_{24} - \gamma - \beta_{1} = 97 - 72 - 17 = 8$ \\
					Loop: $d = 4$ & $\alpha_{10} \gets r_{20} - \gamma = 100 - 72 = 28 $ \\ 
					& $\alpha_{11} \gets r_{22} - \gamma = 100 - 72 = 28 $ \\ 
					& $\alpha_{15} \gets r_{30} - \gamma - \beta_{1} = 100 - 72 - 11 = 11$ \\  \bottomrule
				\end{tabular}
				\caption{Steps of dual procedure (Algorithm~\ref{algorithm:SplitMIO_dual_greedy}). \label{table:SplitMIO_dual_greedy_example} }
			\end{minipage}%
	}}
\end{table}

After running the procedure, the primal solution $\yb = (y_{16},\ldots,y_{31})$ satisfies that $y_{17} = y_{24} = 0.35$, $y_{28} = 0.15$, $y_{20} = y_{22} = y_{30} = 0.05$ and all other components are zero. Here we drop the index $t$ to simplify the notation. The event set is $\Ecal = \{ A_{10}, A_{11}, A_{15}, A_{3}, B_{1}, C \}$, and the function $f: \Ecal \to \leaves$ is defined as $f(A_{10}) = 20$, $f(A_{11}) = 22$, $f(A_{15}) = 30$, $f(A_{3}) = 24$, $f(B_{1}) = 28$, and $f(C) = 17$. 

We now run Algorithm~\ref{algorithm:SplitMIO_dual_greedy}, which carries out the steps shown in Table~\ref{table:SplitMIO_dual_greedy_example}. %
We obtain a dual solution $(\alphab, \betab, \gamma)$, where $\gamma = 72$, $\alpha_{10} = \alpha_{11} = 28$, $\alpha_{15} = 11$, $\beta_1 = 17$, and all other components of $\alphab = (\alpha_1,\ldots,\alpha_{15})$ and $\betab = (\beta_1,\ldots,\beta_{15})$ are zero.

The feasibility of the dual solution is visualized in Figure~\ref{figure:SplitMIO_benders_dual_viz}. The colored bars correspond to the different dual variables; a colored bar appears multiple times when the variable participates in multiple dual constraints. The height of the black lines for each leaf indicates the value of $r_{\ell}$, while the total height of the colored bars at a leaf corresponds to the value $\gamma + \sum_{s \in \LS(\ell)} \alpha_s + \sum_{s \in \RS(\ell)} \beta_s$ (the left hand side of the dual constraint~\eqref{prob:SplitMIO_sub_dual_r}). For each leaf, the total height of the colored bars exceeds the black line, which indicates that all dual constraints are satisfied.

The objective value of the primal solution is $\sum_{\ell \in \leaves } r_\ell \cdot y_\ell = r_{20} \times y_{20} + r_{22} \times y_{22} + r_{30} \times y_{30} + r_{24} \times y_{24} + r_{28} \times y_{28} + r_{17} \times y_{17}  = 87.5$. The objective value of the dual solution is $ \gamma + (1 - x_{2}) \times \beta_{1} + x_{6} \times \alpha_{3} + x_{5} \times \alpha_{10} + x_{5} \times \alpha_{11} + x_{5} \times \alpha_{15} = 72.0 + 0.55 \times 17 + 0.35 \times 8 + 0.05 \times 28 + 0.05 \times 28 + 0.05 \times 11 = 87.5$.

\begin{figure}
	\centering
	\includegraphics[width=0.6\textwidth]{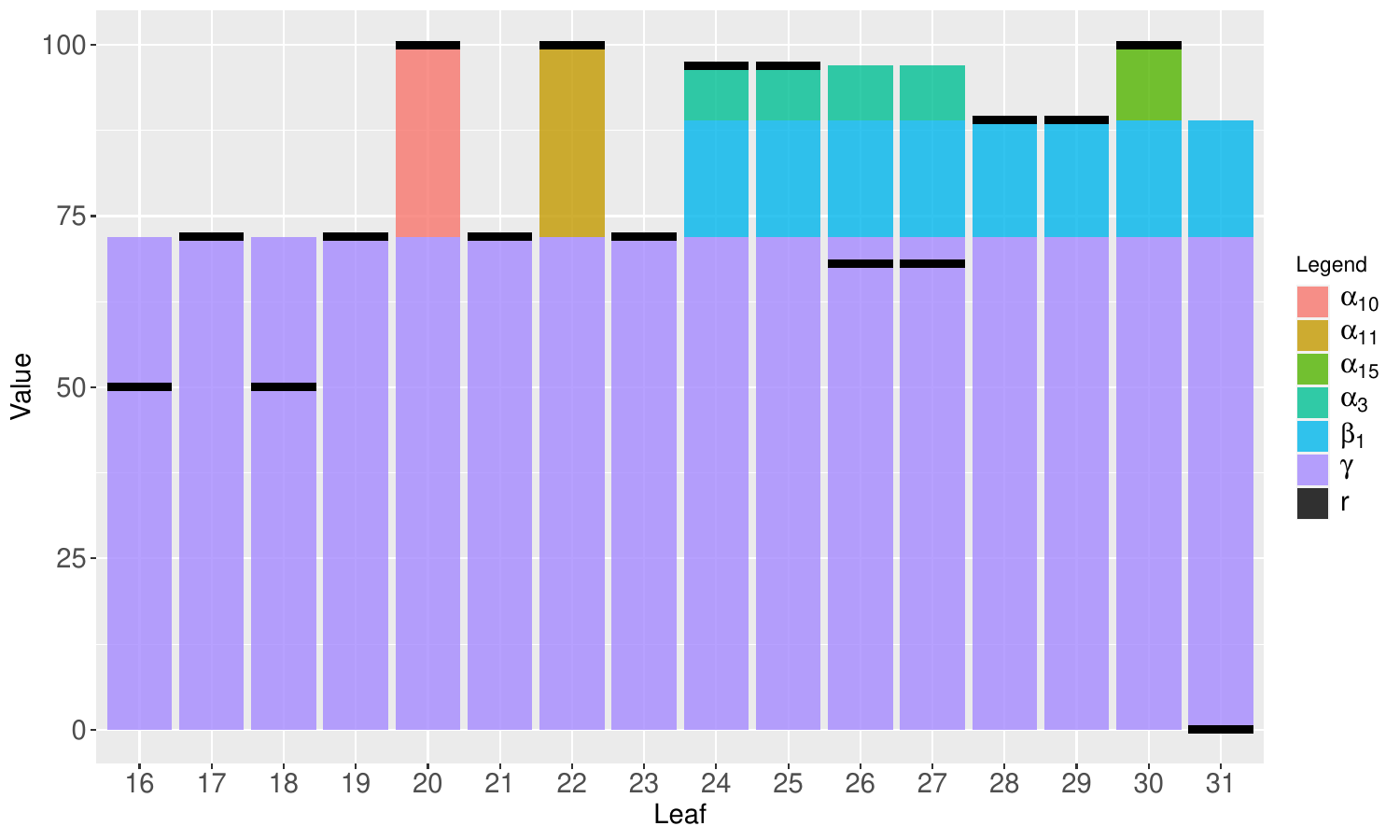}
	\caption{Visualization of feasibility of dual solution. \label{figure:SplitMIO_benders_dual_viz}}
\end{figure}

\section{Benders subproblem of the \ProductMIO relaxation}
\label{subsec:Bender_subproblem_productMIO_not_greedy_solvable}

We continue the discussion in Section~\ref{subsec:benders_ProductMIO_relaxation} and show that Problem~\eqref{prob:ProductMIO_sub_primal} cannot be solved with a greedy approach. We formalize this claim as the following proposition.

\begin{proposition}
	There exists an $\xb \in [0,1]^n$, a tree $t$ and revenues $\rho_1, \dots, \rho_n$ for which the greedy solution to problem~\eqref{prob:ProductMIO_sub_primal} is not optimal. 
	\label{proposition:ProductMIO_benders_primal_greedy_not_optimal}
\end{proposition}

We prove this proposition by providing a counterexample. Consider an instance where $\Ncal = \{1,2,3\}$ and $\xb = (0.5,0.5,0.5)$. Assume the revenues of the products are $\rho_1 = 20$, $\rho_2 = 19$ and $\rho_3 = 18$. Consider the tree shown in Figure~\ref{figure:ProductMIO_greedy_counterexample}.
\begin{figure}[t]
	\begin{center}
		\includegraphics[width=0.25\textwidth]{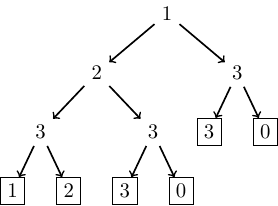}
	\end{center}
	\caption{Structure of tree for which the \ProductMIO primal subproblem is not solvable via a greedy algorithm. \label{figure:ProductMIO_greedy_counterexample}}
\end{figure}
We label the leaves as 1, 2, 3, 4, 5 and 6 from left to right. When we apply the greedy algorithm to solve this LO problem, we can see that there are multiple orderings of the leaves in decreasing revenue: (i) 1,2,3,5,4,6; (ii) 1,2,5,3,4,6; (iii) 1,2,3,5,6,4; and (iv) 1,2,5,3,6,4. For any of these orderings, the greedy solution will turn out to be  $(y_1,y_2,y_3,y_4,y_5,y_6) = (0.5,0,0,0,0,0.5)$, resulting in an objective value of $10$. However, the actual optimal solution $G_t(\xb)$ turns out to be $(y^*_1,y^*_2,y^*_3,y^*_4,y^*_5,y^*_6) = (0,0.5,0,0,0.5,0)$, for which the objective value is $18.5$. 

This counterexample shows that in general, the \ProductMIO primal subproblem cannot be solved to optimality via the same type of greedy algorithm as for \SplitMIO.

\section{Additional Numerical Results Based on Synthetic Data}
\label{appendix:additional_numerical_results}

Our experiments were implemented in the Julia programming language, version 1.8.4 \citepecom{bezanson2017julia} and executed on Amazon Elastic Compute Cloud (EC2) using a single instance of type \texttt{m6a.48xlarge} (AMD EPYC 7R13 processor with 2.95GHz clock speed, 192 virtual CPUs and 768 GB memory). 
All LO and MIO formulations were solved using Gurobi version 10.0.1 and modeled using the \texttt{JuMP} package \citepecom{lubin2015computing}, with a maximum of four cores per formulation.

\subsection{Further Discussion on the T1, T2, and T3 instances}
\label{subsec:synthetic_T1T2T3}

\YCRR{In Section~\ref{subsec:synthetic_T1T2T3}, we introduced the T1, T2, and T3 instances, which represent decision trees with distinct topologies. In this section, we provide additional discussion and empirical evidence related to these instances.

First, the trees in the T1 instances are highly symmetrical and structured. Thus, they are unlikely to arise in decision forest models estimated from real-world data. By contrast, trees in both the T2 and T3 instances trees are commonly observed in estimated decision forest models. Their difference lies in whether the trees are balanced; however, as shown in Table~\ref{table:synthetic_T_integrality_gap}, \SplitMIO and \ProductMIO perform comparably on these instances, with \ProductMIO exhibiting modest advantages in terms of integrality gap, optimality gap, and runtime.

In what follows, we use the IRI Dataset to illustrate the types of tree structures that actually appear in estimated decision forest models. We first define two indices for a given set of estimated trees $F$ in the decision forest model. The first, $I_{\text{FU}}$ (Fraction of Unbalanced Trees), is defined as
\begin{align*}
	I_{\text{FU}} = \sum_{t \in F} \mathbb{I}\left[ \text{tree $t$ is unbalanced} \right] / |F|,
\end{align*}
which measures the fraction of trees in $F$ that are unbalanced. The second, $I_{\text{LU}}$ (Level of Unbalancedness), quantifies how unbalanced the leaf depths are, averaged across trees. 
For a leaf $\ell$ in tree $t$, let $\text{depth}(\ell)$ denote its depth. Recall from Section~\ref{subsec:benders_SplitMIO_relaxation} that the root split node is defined to be at depth $1$. Consequently, $\text{depth}(\ell) \geq 2$ for all $\ell \in \leaves(t)$. 
With a slight abuse of notation, define $d^t_{\max} \equiv \max \{ \text{depth}(\ell) \mid \ell \in \leaves(t)   \}$ and $d^t_{\min} \equiv \min \{ \text{depth}(\ell) \mid \ell \in \leaves(t)   \}$ as the maximum and minimal leaf depths of tree $t$, respectively. The level of unbalancedness of tree $t$ is then
\begin{align*}
	\text{LU}(t) = \frac{ d^t_{\max}  - d^t_{\min}    }{  d^t_{\max}  - 2 },
\end{align*}
which takes value in $[0,1]$. By construction, $\text{LU}(t) = 0$ if and only if tree $t$ is perfectly balanced, while $\text{LU}(t) = 1$ for maximally unbalanced trees (for example, a tree whose split nodes form a single chain; note that this class of trees includes rankings). We then define the overall index for a collection of trees as
\begin{align*}
	I_{\text{LU}} =  \frac{1}{|F|}  \sum_{t \in F} \text{LU}(t) =  \frac{1}{|F|}  \sum_{t \in F}  \frac{ d^t_{\max} - d^t_{\min}  }{ d^t_{\max} - 2 }.
\end{align*}

Table~\ref{table:IRI_unbalanced_trees} reports the two indices for decision forest models estimated from each product category of the IRI Dataset. Both metrics show that the estimated trees are predominantly unbalanced. In particular, $I_{\text{FU}}$ indicates that up to 91\% of trees are unbalanced in some categories, and the overall average is 83\% across all categories. These results suggest that the trees we observe in practice resemble the T3 instances much more closely than the T2 instances. Importantly, we did not observe T1-type trees in these estimated decision forest models, supporting our point that the T1 instances in Section~\ref{subsec:synthetic_background} were constructed more for theoretical illustration rather than practical relevance.}

\begin{table}
	\centering
	\SingleSpacedXI
	\footnotesize
	\begin{tabular}{lrrllrr}
		\toprule
		Category             & $I_{\text{FU}}$ & $I_{\text{LU}}$&   & Category                  & $I_{\text{FU}}$ & $I_{\text{LU}}$  \\ \cmidrule{1-3} \cmidrule{5-7}
		Beer                 & 0.46           & 0.41                             &  & Mayonnaise                & 0.82            & 0.78                           \\
		Blades               & 0.82            & 0.64                              &  & Milk                      & 0.87            & 0.77                             \\
		Carbonated Beverages & 0.86            & 0.82                            &  & Mustard / Ketchup         & 0.91           & 0.84                              \\
		Cigarettes           & 0.88            & 0.82                           &  & Paper Towels              & 0.86            & 0.85                          \\
		Coffee               & 0.86           & 0.74                        &  & Peanut Butter             & 0.90           & 0.87                            \\
		Cold Cereal          & 0.88           & 0.86                       &  & Photo                     & 0.68            & 0.65                \\
		Deodorant            & 0.86            & 0.74                         &  & Salty Snacks              & 0.79           & 0.64                       \\
		Diapers              & 0.88            & 0.71                       &  & Shampoo                   & 0.57           & 0.40              \\
		Facial Tissue        & 0.82            & 0.80                       &  & Soup                      & 0.86           & 0.80                       \\
		Frozen Dinners       & 0.87            & 0.75                    &  & Spaghetti / Italian Sauce & 0.85            & 0.51                        \\
		Frozen Pizza         & 0.75           & 0.65                    &  & Sugar Substitutes         & 0.89           & 0.80                      \\
		Household Cleaners   & 0.88           & 0.87                     &  & Toilet Tissue             & 0.90            & 0.88                       \\
		Hotdogs              & 0.60           & 0.55                      &  & Toothbrush                & 0.9            & 0.90                      \\
		Laundry Detergent    & 0.84            & 0.79                    &  & Toothpaste                & 0.91            & 0.86                       \\
		Margarine / Butter   & 0.87            & 0.81                     &  & Yogurt                    & 0.89            & 0.75                      \\ \midrule 
		&                &                      &  & Average over all categories                 & 0.83           & 0.74            \\  \bottomrule         
	\end{tabular}
	\caption{The fraction of unbalanced trees and level of unbalancedness of the estimated decision forest model for each product category of the IRI Academic Dataset}
	\label{table:IRI_unbalanced_trees}
\end{table}

\subsection{Experiment: Tractability}
\label{subsec:synthetic_T_optimality_gap_T1_T2}

In this experiment, we seek to understand the tractability of \SplitMIO and \ProductMIO when they are solved as integer problems (i.e., not as relaxations). For a given instance and a given formulation $\Mcal$, we solve the integer version of formulation $\Mcal$. Due to the large size of some of the problem instances, we impose a computation time limit of 1 hour for each formulation. We record $T_{\Mcal}$, the computation time required for formulation $\Mcal$, and we record $O_{\Mcal}$ which is the final optimality gap, and is defined as $
O_{\Mcal} = 100\% \times \left({Z_{\text{UB},\Mcal} - Z_{\text{LB}, \Mcal} }\right)/{Z_{\text{UB}, \Mcal}}$, where $Z_{\text{UB},\Mcal}$ and $Z_{\text{LB}, \Mcal}$ are the best upper and lower bounds, respectively, obtained at the termination of formulation $\Mcal$ for the instance. We test all of the T1, T2 and T3 instances with $n = 100$, $|F| \in \{ 50, 100, 200, 500 \}$ and $|\leaves(t)| \in \{ 8, 16, 32, 64 \}$. Table~\ref{table:synthetic_T_optimality_gap} displays the average computation time and average optimality gap of each formulation for each combination of $n$, $|F|$ and $|\leaves(t)|$. For ease of notation in the table, we abbreviate $\SplitMIO$ as $\textsc{S-MIO}$ and $\ProductMIO$ as $\textsc{P-MIO}$, and write $N_L \equiv |\leaves(t)|$. 

From this table, we can see that for the smaller instances, \SplitMIO requires more time to solve than \ProductMIO. For larger instances \YCRR{in this set of experiments}, where the computation time limit is exhausted, the average gap obtained by \ProductMIO tends to be lower than that of \SplitMIO. This is congruent with our theoretical findings (specifically Proposition~\ref{proposition:ProductMIO_stronger_than_SplitMIO}). Similarly to our results on formulation strength in Section~\ref{subsec:synthetic_T_integrality_gap}, \ProductMIO's edge over \SplitMIO is most pronounced for the T1 instances, which exhibit the largest degree of product repetition among the splits compared to the T2 and T3 instances. For the T2 and T3 instances, which exhibit a lower degree of product repetition, the two formulations behave similarly, and the edge of \ProductMIO is smaller; this also makes sense, because as we note in Section~\ref{subsec:synthetic_T_integrality_gap}, \ProductMIO and \SplitMIO coincide when a tree is such that no product is repeated.

\begin{table}[ht]
	\centering
	\SingleSpacedXI
	\footnotesize
	\begin{tabular}{lrrrrrrllrrrrrr} \toprule
		Type & $|F|$ & $N_\text{L}$ & $O_{\textsc{S-MIO}}$ & $O_{\textsc{P-MIO}}$ & $T_{\textsc{S-MIO}}$ & $T_{\textsc{P-MIO}}$ & \hspace{0.5cm} & Type & $|F|$ & $N_\text{L}$ & $O_{\textsc{S-MIO}}$ & $O_{\textsc{P-MIO}}$ & $T_{\textsc{S-MIO}}$ & $T_{\textsc{P-MIO}}$\\ \cmidrule{1-7}  \cmidrule{9-15}
		T1   & 50    & 8              & 0.0             & 0.0               & 0.0             & 0.0               &  & T2   & 200   & 8              & 0.0             & 0.0               & 0.6             & 0.6               \\
		T1   & 50    & 16             & 0.0             & 0.0               & 0.1             & 0.0               &  & T2   & 200   & 16             & 0.0             & 0.0               & 747.7           & 673.6             \\
		T1   & 50    & 32             & 0.0             & 0.0               & 0.5             & 0.1               &  & T2   & 200   & 32             & 9.8             & 9.7               & 3600.0          & 3600.0            \\
		T1   & 50    & 64             & 0.0             & 0.0               & 2.4             & 0.5               &  & T2   & 200   & 64             & 17.1            & 16.3              & 3600.1          & 3600.0            \\  \cmidrule{1-7} \cmidrule{9-15}
		T1   & 100   & 8              & 0.0             & 0.0               & 0.0             & 0.0               &  & T2   & 500   & 8              & 0.0             & 0.0               & 171.6           & 167.9             \\
		T1   & 100   & 16             & 0.0             & 0.0               & 0.8             & 0.1               &  & T2   & 500   & 16             & 14.1            & 13.8              & 3600.0          & 3600.0            \\
		T1   & 100   & 32             & 0.0             & 0.0               & 13.5            & 1.6               &  & T2   & 500   & 32             & 23.4            & 23.0              & 3600.1          & 3600.1            \\
		T1   & 100   & 64             & 0.0             & 0.0               & 479.0           & 65.1              &  & T2   & 500   & 64             & 28.7            & 28.1              & 3600.1          & 3600.1            \\  \cmidrule{1-7} \cmidrule{9-15}
		T1   & 200   & 8              & 0.0             & 0.0               & 0.1             & 0.0               &  & T3   & 50    & 8              & 0.0             & 0.0               & 0.0             & 0.0               \\
		T1   & 200   & 16             & 0.0             & 0.0               & 25.0            & 3.0               &  & T3   & 50    & 16             & 0.0             & 0.0               & 0.1             & 0.1               \\
		T1   & 200   & 32             & 0.7             & 0.0               & 2330.2          & 421.5             &  & T3   & 50    & 32             & 0.0             & 0.0               & 0.6             & 0.6               \\
		T1   & 200   & 64             & 9.8             & 5.3               & 3600.1          & 3600.0            &  & T3   & 50    & 64             & 0.0             & 0.0               & 4.1             & 3.7               \\  \cmidrule{1-7} \cmidrule{9-15}
		T1   & 500   & 8              & 0.0             & 0.0               & 4.6             & 1.4               &  & T3   & 100   & 8              & 0.0             & 0.0               & 0.0             & 0.0               \\
		T1   & 500   & 16             & 5.0             & 1.2               & 3600.0          & 2951.7            &  & T3   & 100   & 16             & 0.0             & 0.0               & 0.8             & 0.7               \\
		T1   & 500   & 32             & 15.9            & 11.8              & 3600.1          & 3600.0            &  & T3   & 100   & 32             & 0.0             & 0.0               & 29.0            & 26.4              \\
		T1   & 500   & 64             & 20.9            & 17.4              & 3600.1          & 3600.1            &  & T3   & 100   & 64             & 0.9             & 0.5               & 2600.8          & 2206.5            \\  \cmidrule{1-7}  \cmidrule{9-15}
		T2   & 50    & 8              & 0.0             & 0.0               & 0.0             & 0.0               &  & T3   & 200   & 8              & 0.0             & 0.0               & 0.4             & 0.4               \\
		T2   & 50    & 16             & 0.0             & 0.0               & 0.1             & 0.1               &  & T3   & 200   & 16             & 0.0             & 0.0               & 76.8            & 79.1              \\
		T2   & 50    & 32             & 0.0             & 0.0               & 0.5             & 0.5               &  & T3   & 200   & 32             & 5.3             & 5.1               & 3600.0          & 3600.0            \\
		T2   & 50    & 64             & 0.0             & 0.0               & 24.0            & 24.6              &  & T3   & 200   & 64             & 13.5            & 12.5              & 3600.1          & 3600.0            \\  \cmidrule{1-7} \cmidrule{9-15}
		T2   & 100   & 8              & 0.0             & 0.0               & 0.0             & 0.0               &  & T3   & 500   & 8              & 0.0             & 0.0               & 75.9            & 72.9              \\
		T2   & 100   & 16             & 0.0             & 0.0               & 1.5             & 1.4               &  & T3   & 500   & 16             & 8.8             & 8.8               & 3600.0          & 3600.0            \\
		T2   & 100   & 32             & 0.0             & 0.0               & 245.2           & 234.9             &  & T3   & 500   & 32             & 21.0            & 20.3              & 3600.1          & 3600.0            \\
		T2   & 100   & 64             & 4.8             & 4.4               & 3600.0          & 3600.0            &  & T3   & 500   & 64             & 28.9            & 27.9              & 3600.2          & 3600.1     \\ \bottomrule
	\end{tabular}
	\caption{Comparison of final optimality gaps and computation times for \SplitMIO and \ProductMIO.
		\label{table:synthetic_T_optimality_gap}
	}
\end{table}

\YCRR{We note that Gurobi is able to process the full formulations of both \SplitMIO\ and \ProductMIO\ for all instances in this experiment, though the larger instances are not solved to optimality within the one-hour time limit.}

\subsection{Comparison to Heuristic Approaches}
\label{subsec:synthetic_T_vs_H}

In this experiment, we compare the performance of the two formulations to three different heuristic approaches:
\begin{enumerate}
\item \LocalSearch: A local search heuristic, which starts from the empty assortment, and in each iteration moves to the neighboring assortment which improves the expected revenue the most. The neighborhood of assortments consists of those assortments obtained by adding a new product to the current assortment, or removing one of the existing products from the assortment. The heuristic terminates when there is no assortment in the neighborhood of the current one that provides an improvement. 
\item \LocalSearchTen: This heuristic involves running \LocalSearch from ten randomly chosen starting assortments. Each assortment is chosen uniformly at random from the set of $2^n$ possible assortments. After the ten repetitions, the assortment with the best expected revenue is retained.
\item \ROA: This heuristic involves finding the optimal revenue ordered assortment. More formally, we define $S_k = \{i_1,\dots, i_k\}$, where $i_1,\dots, i_n$ corresponds to an ordering of the products so that $r_{i_1} \geq r_{i_2} \geq \dots \geq r_{i_n}$, and we find $\arg \max_{S \in \{S_1, \dots, S_n\}} R^{(F,\lambdab)}(S)$. 
\end{enumerate}

We compare these heuristics against the best integer solution obtained by each of our two MIO formulations, leading to a total of five methods for each instance. We measure the performance of the solution corresponding to approach $\Mcal$ using the metric $G_{\Mcal}$, which is defined as $G_{\Mcal} = 100\% \times { \left(Z' - Z_{\Mcal}\right)}/Z'$, where $Z'$ is the highest lower bound (i.e., integer solution) obtained from among the two MIO formulations and the three heuristics, and $Z_{\Mcal}$ is the objective value of the solution returned by approach $\Mcal$. 

\begin{table}
	\centering
	\SingleSpacedXI
	\footnotesize
	\begin{tabular}{lrrrrrrrllrrrrrrr} \toprule
		Type & $|F|$ & $N_L$ &  $G_{\textsc{S-MIO}}$ & $G_{\textsc{P-MIO}}$ & $G_{\LocalSearch}$ & $G_{\LocalSearchTen}$ & $G_{\ROA}$ & \hspace{0.5cm}   & Type & $|F|$ & $N_L$ &  $G_{\textsc{S-MIO}}$ & $G_{\textsc{P-MIO}}$ & $G_{\LocalSearch}$ & $G_{\LocalSearchTen}$ & $G_{\ROA}$ \\  \cmidrule{1-8} \cmidrule{ 10-17 }
		T1 & 50  & 8  & 0.0 & 0.0 & 8.1  & 2.1 & 26.8 &  & T2 & 200 & 8  & 0.0 & 0.0 & 3.8  & 0.8 & 23.1 \\
		T1 & 50  & 16 & 0.0 & 0.0 & 9.8  & 4.0 & 31.2 &  & T2 & 200 & 16 & 0.0 & 0.0 & 5.5  & 2.8 & 24.2 \\
		T1 & 50  & 32 & 0.0 & 0.0 & 9.0  & 5.2 & 27.5 &  & T2 & 200 & 32 & 0.2 & 0.2 & 7.0  & 4.5 & 24.7 \\
		T1 & 50  & 64 & 0.0 & 0.0 & 10.0 & 7.2 & 28.6 &  & T2 & 200 & 64 & 1.0 & 0.4 & 8.2  & 5.5 & 23.5 \\ \cmidrule{1-8} \cmidrule{ 10-17 }
		T1 & 100 & 8  & 0.0 & 0.0 & 3.9  & 2.4 & 26.0 &  & T2 & 500 & 8  & 0.0 & 0.0 & 2.0  & 0.5 & 15.6 \\
		T1 & 100 & 16 & 0.0 & 0.0 & 6.6  & 3.7 & 26.2 &  & T2 & 500 & 16 & 0.1 & 0.1 & 3.8  & 1.5 & 16.3 \\
		T1 & 100 & 32 & 0.0 & 0.0 & 8.5  & 6.3 & 25.6 &  & T2 & 500 & 32 & 0.5 & 0.4 & 4.9  & 1.8 & 15.1 \\
		T1 & 100 & 64 & 0.0 & 0.0 & 9.9  & 6.6 & 24.5 &  & T2 & 500 & 64 & 1.1 & 0.9 & 4.5  & 1.9 & 14.3 \\ \cmidrule{1-8} \cmidrule{ 10-17 }
		T1 & 200 & 8  & 0.0 & 0.0 & 3.5  & 1.3 & 19.9 &  & T3 & 50  & 8  & 0.0 & 0.0 & 13.2 & 3.8 & 33.1 \\
		T1 & 200 & 16 & 0.0 & 0.0 & 5.5  & 3.1 & 21.7 &  & T3 & 50  & 16 & 0.0 & 0.0 & 14.4 & 5.2 & 34.9 \\
		T1 & 200 & 32 & 0.0 & 0.0 & 7.8  & 4.6 & 22.3 &  & T3 & 50  & 32 & 0.0 & 0.0 & 12.6 & 5.0 & 33.7 \\
		T1 & 200 & 64 & 0.3 & 0.0 & 7.6  & 6.3 & 19.5 &  & T3 & 50  & 64 & 0.0 & 0.0 & 13.9 & 8.1 & 33.0 \\ \cmidrule{1-8} \cmidrule{ 10-17 }
		T1 & 500 & 8  & 0.0 & 0.0 & 1.5  & 0.3 & 14.6 &  & T3 & 100 & 8  & 0.0 & 0.0 & 8.0  & 1.9 & 30.2 \\
		T1 & 500 & 16 & 0.0 & 0.0 & 3.7  & 1.6 & 15.3 &  & T3 & 100 & 16 & 0.0 & 0.0 & 10.0 & 3.1 & 32.6 \\
		T1 & 500 & 32 & 0.3 & 0.0 & 5.3  & 2.4 & 15.9 &  & T3 & 100 & 32 & 0.0 & 0.0 & 10.4 & 3.8 & 32.0 \\
		T1 & 500 & 64 & 0.6 & 0.1 & 4.9  & 3.5 & 13.8 &  & T3 & 100 & 64 & 0.0 & 0.0 & 10.0 & 4.5 & 31.0 \\ \cmidrule{1-8} \cmidrule{ 10-17 }
		T2 & 50  & 8  & 0.0 & 0.0 & 13.8 & 3.2 & 31.5 &  & T3 & 200 & 8  & 0.0 & 0.0 & 4.0  & 1.1 & 25.8 \\
		T2 & 50  & 16 & 0.0 & 0.0 & 11.6 & 4.9 & 32.2 &  & T3 & 200 & 16 & 0.0 & 0.0 & 6.5  & 2.5 & 26.5 \\
		T2 & 50  & 32 & 0.0 & 0.0 & 10.0 & 5.8 & 31.1 &  & T3 & 200 & 32 & 0.1 & 0.1 & 7.6  & 3.3 & 26.9 \\
		T2 & 50  & 64 & 0.0 & 0.0 & 11.6 & 7.0 & 30.4 &  & T3 & 200 & 64 & 0.3 & 0.1 & 9.2  & 4.1 & 24.1 \\ \cmidrule{1-8} \cmidrule{ 10-17 }
		T2 & 100 & 8  & 0.0 & 0.0 & 5.5  & 1.9 & 28.1 &  & T3 & 500 & 8  & 0.0 & 0.0 & 2.6  & 0.4 & 15.8 \\
		T2 & 100 & 16 & 0.0 & 0.0 & 8.2  & 4.0 & 30.8 &  & T3 & 500 & 16 & 0.0 & 0.0 & 4.3  & 1.1 & 16.5 \\
		T2 & 100 & 32 & 0.0 & 0.0 & 8.9  & 4.8 & 31.3 &  & T3 & 500 & 32 & 0.5 & 0.5 & 5.3  & 1.3 & 16.0 \\
		T2 & 100 & 64 & 0.0 & 0.0 & 11.6 & 6.6 & 27.1 &  & T3 & 500 & 64 & 0.9 & 0.4 & 5.5  & 1.3 & 16.3 \\ \bottomrule
	\end{tabular}
	\caption{Comparison of \SplitMIO and \ProductMIO against the heuristics \LocalSearch, \LocalSearchTen and \ROA. \label{table:synthetic_T_vs_H} }
\end{table}

Table~\ref{table:synthetic_T_vs_H} shows the performance of the five approaches -- \SplitMIO, \ProductMIO, \LocalSearch, \LocalSearchTen and \ROA \ -- for each family of instances. The gaps are averaged over the twenty instances for each combination of instance type, $|F|$ and $|\leaves(t)|$.  Again, for ease of notation in the table, we abbreviate $\SplitMIO$ as $\textsc{S-MIO}$ and $\ProductMIO$ as $\textsc{P-MIO}$, and write $N_L \equiv |\leaves(t)|$. We can see from this table that in general, for the small instances, the solutions obtained by the MIO formulations are either the best or close to the best out of the five approaches, while the solutions produced by the heuristic approaches are quite suboptimal. For cases where the gap is zero for the MIO formulations, the gap of \LocalSearch ranges from 1.5\% to 13.9\%; the \LocalSearchTen heuristic improves on this, due to its use of restarting and randomization, but still does not perform as well as the MIO solutions (gaps ranging from 0.3 to 8.1\%). For the larger instances, where the gap of the MIO solutions is larger, \LocalSearch and \LocalSearchTen still tend to perform worse. Across all of the instances, \ROA achieves much higher gaps than all of the other approaches (ranging from 13.8 to 34.9\%). Overall, these results suggest that \emph{the very general structure of the decision forest model poses significant difficulty to standard heuristic approaches}, and highlight the value of using exact approaches over inexact/heuristic approaches to the assortment optimization problem.

\YCRR{
\subsection{Value of the Relaxation Phase in the \SplitMIO Benders Decomposition}
\label{subsec:value_of_relaxation_phase}

In this subsection, we present results from a computational experiment designed to understand the value of solving the linear relaxation of the \SplitMIO master problem~\eqref{prob:SplitMIO_benders_master_Dcal} before solving its integer version. We compare the performance of (i) our original Benders approach, which includes both the relaxation phase (phase~1) and the integer phase (phase~2) (see Section~\ref{subsec:benders_overall_approach}), and (ii) an alternative Benders approach that consists only of phase~2. In the latter approach, we solve the \SplitMIO master problem~\eqref{prob:SplitMIO_benders_integer_master_Dcal} using lazy constraint generation to separate constraints~\eqref{prob:SplitMIO_benders_integer_master_Dcal_mainconstraint} at integer solutions, without initializing the problem with cuts obtained from solving the LP relaxation.

We denote the original Benders approach (with phases~1 and~2) by the subscript ``\Benders'', and the alternative version that only executes phase~2 by ``\BendersTwoOnly.'' To ensure a fair comparison, we impose a two-hour time limit on the {\BendersTwoOnly} approach, matching the total time allocated to the original Benders approach (one hour for phase~1 and one hour for phase~2). Table~\ref{table:B_norlx2hrs_comparison_O} reports the final optimality gap metric, $O_{\Mcal} = (Z_{\text{UB}, \Mcal} - Z_{\text{LB}, \Mcal})/Z_{\text{UB}, \Mcal} \times 100\%$, where $Z_{\text{UB},\Mcal}$ and $Z_{\text{LB}, \Mcal}$ are the best upper and lower bounds, respectively, obtained from either the original Benders approach ($\Mcal  = \Benders$) or the phase-2-only approach ($\Mcal = \BendersTwoOnly$) after the computation time limit is exhausted. The same instances from Table~\ref{table:synthetic_B_T6} are used for consistency.

Table~\ref{table:B_norlx2hrs_comparison_O} summarizes the results and shows that the relaxation phase substantially improves performance across all instances. The reduction in optimality gap is particularly pronounced for larger problems ($n = 2000$ or $3000$), where the gap is reduced by roughly half---or even more. For example, for the $(n,\rho) = (2000, 0.06)$ instances, the original Benders approach achieves an average gap of about 2.5\%, while the phase-2-only approach yields a gap of roughly 10\%. Similarly, for $(n,\rho) = (3000, 0.12)$, the original Benders approach reaches an optimality gap of around 10\%, compared to approximately 35\% for the phase-2-only approach. This improvement occurs because the cuts from phase~1 provide a tighter upper bound and a more accurate piecewise-linear approximation of the functions ${ G_t(\cdot)}_{t \in F}$, enabling more effective pruning in the branch-and-bound process and resulting in better bounds within the limited computation time. In addition to the gap performance, we also observed that the integer solutions obtained by the original Benders approach are generally better than those from the phase-2-only approach for the larger instances.

}

\begin{table}
	\SingleSpacedXI
	\footnotesize
	\centering
	\YCRR{
	\begin{tabular}{ccrrr} \toprule
		$n$ & $\rho$ & $b$ & $O_{\Benders}$ & $O_{\BendersTwoOnly}$ \\ \midrule
		200 & 0.02 &   4 & 0.00 & 0.00 \\ 
		200 & 0.04 &   8 & 0.00 & 0.00 \\ 
		200 & 0.06 &  12 & 8.23 & 5.28 \\ 
		200 & 0.08 &  16 & 17.38 & 17.46 \\ 
		200 & 0.10 &  20 & 22.13 & 24.16 \\ 
		200 & 0.12 &  24 & 27.27 & 28.82 \\ \midrule
		500 & 0.02 &  10 & 0.00 & 0.00 \\ 
		500 & 0.04 &  20 & 0.48 & 1.66 \\ 
		500 & 0.06 &  30 & 8.10 & 12.15 \\ 
		500 & 0.08 &  40 & 14.26 & 19.09 \\ 
		500 & 0.10 &  50 & 18.26 & 25.84 \\ 
		500 & 0.12 &  60 & 22.94 & 29.13 \\ \midrule
		1000 & 0.02 &  20 & 0.00 & 0.00 \\ 
		1000 & 0.04 &  40 & 1.05 & 6.80 \\ 
		1000 & 0.06 &  60 & 6.16 & 13.98 \\ 
		1000 & 0.08 &  80 & 10.07 & 17.77 \\ 
		1000 & 0.10 & 100 & 13.60 & 20.78 \\ 
		1000 & 0.12 & 120 & 15.77 & 25.98 \\ \midrule
		2000 & 0.02 &  40 & 0.00 & 0.20 \\ 
		2000 & 0.04 &  80 & 0.21 & 6.43 \\ 
		2000 & 0.06 & 120 & 2.42 & 9.70 \\ 
		2000 & 0.08 & 160 & 5.21 & 13.30 \\ 
		2000 & 0.10 & 200 & 7.16 & 15.83 \\ 
		2000 & 0.12 & 240 & 11.91 & 17.01 \\ \midrule
		3000 & 0.02 &  60 & 0.00 & 0.46 \\ 
		3000 & 0.04 & 120 & 0.08 & 6.22 \\ 
		3000 & 0.06 & 180 & 1.81 & 10.44 \\ 
		3000 & 0.08 & 240 & 4.23 & 14.84 \\ 
		3000 & 0.10 & 300 & 6.17 & 14.07 \\ 
		3000 & 0.12 & 360 & 9.83 & 34.78 \\ \bottomrule
	\end{tabular}
	}
	\caption{Comparison of the full Benders approach (phases 1 and 2) with the phase-2-only Benders approach.\label{table:B_norlx2hrs_comparison_O} }
\end{table}

\subsection{Value of the Greedy-Based Benders Algorithm}
\label{subsec:greedy_experiment}

In this subsection, we present results from a computational experiment to showcase the value of the greedy-based Benders approach for solving the linear relaxation of \SplitMIO. We consider the same T3 instances from Section~\ref{subsec:synthetic_background}.

For each instance, we solve the linear relaxation of \SplitMIO using the greedy-based Benders approach, where we solve the master problem using constraint generation, and the primal and dual subproblems are solved using Algorithms~\ref{algorithm:SplitMIO_primal_greedy} and \ref{algorithm:SplitMIO_dual_greedy}, respectively. We denote this solution method by \SplitMIOGreedy. We also solve the linear relaxation of \SplitMIO using constraint generation, where we solve the dual subproblem~\eqref{prob:SplitMIO_sub_dual} directly using Gurobi, without using our greedy algorithms. We denote this solution method by \SplitMIONaive. Lastly, we also solve the linear relaxation of \ProductMIO using constraint generation, where again we solve the dual subproblem~\eqref{prob:ProductMIO_sub_dual} in that formulation directly using Gurobi. We denote this solution method by \ProductMIONaive. We impose a time limit of 2 hours on all three solution methods.

For each instance and for each method $m$ we record the solution time, $T_m$. We additionally compute $R_{m}$ as the reduction in computation time of using \SplitMIOGreedy relative to method $m$; mathematically, it is defined as 
\begin{equation}
	R_{m} = 100\% \times \left(  T_{m} - T_{\SplitMIOGreedy} \right) \slash {T_{m}} .
\end{equation}
We compute this metric for $m$ equal to either \SplitMIONaive or \ProductMIONaive. Lastly, we also compute the relative gap, $H$, of the \SplitMIO relaxation relative to the \ProductMIO relaxation, as 
\begin{equation}
	H = 100\% \times \left({Z_{\SplitMIOGreedy} - Z_{\ProductMIONaive}}\right) \slash {Z_{\ProductMIONaive}} ,
\end{equation}
where $Z_{\SplitMIOGreedy}$ and $Z_{\ProductMIONaive}$ are the objective values from the \SplitMIOGreedy and \ProductMIONaive approaches, respectively. Recall that the bound from \ProductMIO is always at least as tight as that of \SplitMIO, and hence this $H$ will always be a positive percentage.

\begin{table}
	\SingleSpacedXI
	\footnotesize
	\centering
		\begin{tabular}{ccrrrrrr} \toprule
		$n$ & $b$ & $T_{\SplitMIOGreedy}$ & $T_{\SplitMIONaive}$ & $T_{\ProductMIONaive}$ & $R_{\SplitMIONaive}$ & $R_{\ProductMIONaive}$ & $H$ \\ \midrule
		200 & 4 & 11.6 & 134.5 & 107.2 & 91.4 & 89.2 & 0.3 \\ 
		200 & 8 & 11.7 & 119.9 & 107.9 & 90.3 & 89.2 & 0.6 \\ 
		200 & 12 & 12.0 & 121.9 & 110.5 & 90.1 & 89.1 & 0.8 \\ 
		200 & 16 & 12.0 & 125.1 & 110.2 & 90.4 & 89.1 & 0.8 \\ 
		200 & 20 & 12.4 & 120.5 & 106.3 & 89.7 & 88.4 & 0.7 \\ 
		200 & 24 & 11.8 & 115.8 & 104.6 & 89.8 & 88.7 & 0.7 \\ \midrule
		500 & 10 & 9.9 & 110.0 & 111.7 & 91.0 & 91.0 & 0.5 \\ 
		500 & 20 & 20.8 & 182.4 & 189.0 & 88.6 & 88.9 & 0.4 \\ 
		500 & 30 & 23.0 & 189.9 & 193.4 & 87.9 & 88.1 & 0.4 \\ 
		500 & 40 & 23.2 & 186.8 & 190.4 & 87.6 & 87.8 & 0.4 \\ 
		500 & 50 & 24.7 & 192.8 & 194.3 & 87.2 & 87.2 & 0.3 \\ 
		500 & 60 & 25.1 & 184.4 & 189.5 & 86.4 & 86.8 & 0.3 \\ \midrule
		1000 & 20 & 11.4 & 121.3 & 170.7 & 90.6 & 93.4 & 0.1 \\ 
		1000 & 40 & 44.2 & 259.7 & 345.4 & 83.1 & 87.3 & 0.1 \\ 
		1000 & 60 & 55.5 & 285.4 & 412.2 & 80.6 & 86.6 & 0.2 \\ 
		1000 & 80 & 61.9 & 295.0 & 412.0 & 79.0 & 85.0 & 0.2 \\ 
		1000 & 100 & 70.2 & 303.0 & 415.4 & 76.8 & 83.1 & 0.2 \\ 
		1000 & 120 & 78.5 & 304.7 & 408.1 & 74.3 & 80.8 & 0.2 \\ \midrule
		2000 & 40 & 11.0 & 144.3 & 239.5 & 92.3 & 95.4 & 0.0 \\ 
		2000 & 80 & 36.8 & 204.1 & 439.8 & 82.0 & 91.7 & 0.0 \\ 
		2000 & 120 & 122.5 & 399.3 & 746.4 & 69.7 & 83.8 & 0.1 \\ 
		2000 & 160 & 196.2 & 507.6 & 950.8 & 61.6 & 79.5 & 0.0 \\ 
		2000 & 200 & 288.9 & 627.3 & 1096.7 & 53.9 & 73.7 & 0.1 \\ 
		2000 & 240 & 385.0 & 741.5 & 1249.3 & 48.2 & 69.3 & 0.1 \\ \midrule
		3000 & 60 & 11.6 & 163.2 & 361.6 & 92.9 & 96.7 & 0.0 \\ 
		3000 & 120 & 64.1 & 253.4 & 658.5 & 75.7 & 90.7 & 0.0 \\ 
		3000 & 180 & 299.4 & 632.2 & 1391.3 & 52.8 & 78.6 & 0.0 \\ 
		3000 & 240 & 619.7 & 1025.3 & 2079.3 & 39.6 & 70.2 & 0.0 \\ 
		3000 & 300 & 1009.8 & 1489.2 & 2609.9 & 32.2 & 61.3 & 0.0 \\ 
		3000 & 360 & 1524.0 & 2011.5 & 3258.1 & 24.2 & 53.2 & 0.0 \\  \bottomrule
	\end{tabular}
	\caption{Comparison of the solution methods for the \SplitMIO and \ProductMIO relaxations.}
	\label{table:greedy_vs_naive}
\end{table}

In Table~\ref{table:greedy_vs_naive}, we report the runtime $T_m$ for $m \in \{ \SplitMIOGreedy, \SplitMIONaive,  \ProductMIONaive \}$, along with the reduction in computation time of \SplitMIOGreedy relative to \SplitMIONaive and \ProductMIONaive, denoted as $R_{\SplitMIONaive}$ and $R_{\ProductMIONaive}$, respectively. Additionally, we present the relative gap $H$ of the \SplitMIO relaxation relative to the \ProductMIO relaxation. From this table, we obtain two important insights. The first is that \SplitMIOGreedy is always faster than \SplitMIONaive for solving the \SplitMIO relaxation across all of the combinations of $n$ and $b$. Comparing \SplitMIOGreedy and \SplitMIONaive, the reduction in solution time can be quite large in some cases. For example, for $n = 2000$ and $b = 40$, the average solution time for \SplitMIOGreedy is 11.0 seconds, while for \SplitMIONaive it is 144.3 seconds, which corresponds to a reduction of $R_{\SplitMIONaive} = 92.3\%$. Thus, even focusing on \SplitMIO, our proposed greedy-based method can be of significant utility in being able to solve the relaxation of this problem quickly at scale. 

The second insight is that \SplitMIOGreedy solves its relaxation faster than \ProductMIONaive across all values of $n$ and $b$. Interestingly, \ProductMIONaive takes longer to solve than \SplitMIONaive, and the reduction in computation time achieved by \SplitMIOGreedy relative to \ProductMIONaive is even more pronounced. For example, in the same instances where $n = 2000$ and $b = 40$, \SplitMIOGreedy requires on average 11.0 seconds, whereas \ProductMIONaive requires on average 239.5 seconds, which corresponds to a reduction in computation of $R_{\ProductMIONaive} = 95.4\%$. 

With regard to this second insight, a natural question is how the objective values of relaxed \SplitMIO and \ProductMIO compare for these instances: it could be the case that \ProductMIONaive produces an appreciably tighter bound than \SplitMIOGreedy, and hence the increase in computation time associated with \ProductMIONaive could potentially be worth it. However, this turns out to not be the case. In general, the average value of $H$ across all $n$ and $b$ values is no more than 0.8\%. This is consistent with our findings in Section~\ref{subsec:synthetic_T_integrality_gap}, where we show that the reduction in the integrality gap of \ProductMIO is largest for the T1 instances, which exhibit a huge degree of product repetition across splits, and smallest for the T3 instances, where there is a smaller degree of repetition and a higher degree of flexibility in the tree structure. 

Overall, these results highlight two synergistic aspects of the \SplitMIO formulation: (1) for the same problem instance, the greedy Benders approach for solving the linear relaxation of \SplitMIO is in general faster than the traditional ``direct'' Benders approaches for solving the linear relaxations of \SplitMIO and \ProductMIO; and (2) while \SplitMIO is in theory not as tight as \ProductMIO, the loss in bound quality is negligible in large-scale, randomly generated instances that do not exhibit a high degree of structure.

\section{Additional Numerical Results Based on the Sushi Dataset}
\label{appendix:additional_numerical_results_real_data}

In this section, we evaluate the end-to-end performance of the decision forest model in a data-to-decision setting, using a semi-synthetic experiment based on the Sushi dataset \citepecom{kamishima2003nantonac}. Our results demonstrate that when the data exhibits non-rational choice behavior that violates the regularity property, the decision forest model can yield better assortment decisions than rational choice models.

\subsection{Data Processing}
\label{subsec:sushi_preprocess}
Our data preprocessing process closely follows \citeecom{berbeglia2022comparative}. The Sushi dataset consists of the sushi preferences for 5,000 respondents. In the survey, each respondent was shown $n = 10$ type of sushi: \emph{ebi} (shrimp), \emph{anago} (sea eel), \emph{maguro} (tuna), \emph{ika} (squid), \emph{uni} (sean urchin), \emph{ikura} (salmon roe), \emph{tamago} (egg), \emph{toro} (fatty tuna), \emph{tekka maki} (tuna roll), and \emph{kappa maki} (cucumber roll). The respondent then was asked to rank these ten sushi types. Note that in this experiment, the respondents were not asked to rank the no-purchase option among the sushi types. We therefore consider a setting which \citeecom{berbeglia2022comparative} referred as \emph{Top 3}; see Section 4.2 of \citeecom{berbeglia2022comparative}. In this setting, each respondent's ranking is trimmed to have only length three. In other words, when offered an assortment, each respondent would only buy a sushi if one of her top 3 sushi types is in the assortment. %

We use $\sigma_1,\ldots,\sigma_{5000}$ to denote the resulting trimmed rankings. We assume each ranking $\sigma_k$, $k = 1,\ldots,5000$, represent a customer type in the market with probability weight $\lambda_k = 1/5000$. Note that each ranking $\sigma_{k}: \Ncal \cup \{ 0 \} \rightarrow \Ncal \cup \{ 0 \} $ is a bijection that assigns each option to a rank. Particularly, we use $\sigma_k(i)$ to denote the rank of option $i \in \Ncal \cup \{ 0 \}$, and $\sigma_k(i) < \sigma_k(j)$ indicates that $i$ is more preferred to $j$ under ranking $\sigma_k$. We use $\sigma_k[S]$ to denote the most preferred option when assortment $S$ is offered, i.e., $\sigma_k[S] \equiv \arg\min_{i \in S \cup \{ 0 \}} \sigma_k(i)$.

\subsection{Ground-Truth Choice Model with Choice Overload}

\YCRR{We study the performance of the decision forest model in a data-to-decision setting where customer choices may deviate from the regularity property. To this end, we construct a ground-truth choice model using rankings from the Sushi dataset, augmented with a behavioral component that captures choice overload (see Section~\ref{subsec:beer} for an overview). The idea is that as the assortment size increases, customers experience greater mental fatigue when evaluating options, making them more likely to leave without making a purchase.

For each $k \in \{1,\ldots,5000 \}$, let $\tilde{\sigma}_k[S]$ denote the purchase decision of customer type $k$ facing assortment $S$. We assume
\begin{align*}
	\tilde{\sigma}_k[S] = \left\{ \begin{array}{ll} \sigma_k[S], & \text{ with probability } 1- \pi(|S|), \\
		0, & \text{ with probability } \pi(|S|), \end{array} \right.
\end{align*}
where $\pi(|S|)$ is the baseline no-purchase probability. If $\pi(|S|)$ is fixed -- e.g., following \citeecom{berbeglia2022comparative}, where $\pi(|S|) = 0.5$ for all $S$ -- the resulting ground-truth choice model
\begin{align}
	\label{eq:choice_overload_groundtruth}
	P_{\texttt{GT}} (i \mid S) = \sum_{i=1}^{5000}  \frac{1}{5000} \cdot P \left(  \tilde{\sigma}_k[S] = i  \right) 
\end{align}
is a ranking-based model and thus consistent with the regularity property. However, if $\pi(|S|)$ increases with assortment size, the model may violate regularity and enter a non-rational choice regime.

In our experiment, we let the baseline no-purchase probability grow linearly with the assortment size:
\begin{equation}
\label{eq:baseline_no-purchase_prob}
\pi(|S|) = 0.5 + \alpha_{\text{OL}} \cdot \left( |S| - 3  \right)_+,
\end{equation}
where $(x)_+ = \max \{x,0 \}$ and $\alpha_{\text{OL}}$ measures the intensity of choice overload. When the assortment has three or fewer products, customers can still easily evaluate all options, form preferences, and make decisions without significant fatigue. Beyond this threshold, larger assortments increase cognitive load, making customers more likely to opt out of purchasing. Specifically, each additional product raises the no-purchase probability by $\alpha_\text{OL}$. Finally, the constant $0.5$ in \eqref{eq:baseline_no-purchase_prob} has no impact on the optimal assortment decision, as it uniformly scales all choice probabilities. We retain it so that when $\alpha_\text{OL}=0$, our setting coincides with that of \citeecom{berbeglia2022comparative}.
}

\subsection{Sample Transactions using the Ground-Truth Choice Model}

\YCRR{
We generate transaction data from the ground-truth model as follows. First, we sample a set of historical assortments $\Scal_{\text{sample}} = \{S_1,\ldots,S_{n_{\text{asst}}}\}$ uniformly without replacement from the set of all possible assortments. At each time period $\tau = 1,\ldots,n_{\text{tran}}$, an assortment $S_\tau$ is drawn uniformly at random from $\Scal_{\text{sample}}$, and a customer makes a purchase decision $o_\tau \in S_\tau \cup {0}$. Specifically, the arriving customer is assumed to belong to type $k$ with probability $1/5000$ and makes a choice according to $\tilde{\sigma}_k$. The purchase decision is then sampled as $o_\tau \sim \tilde{\sigma}_k[S_\tau]$. The resulting set of transactions is $\Tcal_1 \equiv \{ (S_\tau,o_\tau)_{\tau = 1,\ldots, n_\text{tran}} \}$.

We set $n_{\text{asst}} = 150$, which is roughly one-seventh of the total number of possible assortments ($2^n = 1024$). This design ensures that solving the data-to-decision problem requires the choice model to generalize effectively to unobserved assortments outside $\Scal_{\text{sample}}$. We also set $n_{\text{tran}} = 5000$, corresponding to five thousand transactions. To obtain multiple datasets, we repeat this data-generation procedure nine additional times, yielding ten transaction sets $\Tcal_1, \Tcal_2, \ldots, \Tcal_{10}$.
}

Before we present the numerical results, we comment on the difference between the setting in this section and the setting presented in Section~\ref{sec:IRI}. We argue that the juxtaposition of the two numerical studies can help us assess the optimal assortments returned by the decision forest model from two different angles. Recall that in Section~\ref{sec:IRI}, the IRI dataset offers real-world transaction data, which enable us to discuss how real consumer behavior and its inherit bias, such as the choice overload effect, influence the assortment decision; see Section~\ref{subsec:beer}. Meanwhile, the IRI dataset does not offer a ground truth model. Consequently, we were not able to \emph{absolutely} evaluate the performance of the assortment $S^{\text{DF}}$ returned by the decision forest model. The best we can do is to evaluate it \emph{relatively}, i.e., with respect to another estimated choice model; see Table~\ref{table:IRI}. In contrast, the numerical study in this section is built around a ground-truth choice model $P_{\texttt{GT}}(\cdot \mid \cdot)$. Having access to $P_{\texttt{GT}} (\cdot \mid \cdot)$ allows us to evaluate the expected revenue of any assortment \emph{absolutely}, rather than only in relative terms, and to quantify the optimality gap directly. Moreover, because $P_\GT(\cdot \mid \cdot)$ is explicitly constructed to incorporate both ranking-based preferences and choice overload, the synthetic datasets $\Tcal_1, \ldots, \Tcal_{10}$ capture consumer non-rationality solely through the choice overload effect, without additional confounding behaviors.

\subsection{Model Estimation and Performance Metrics}
\label{subsec:sushi_metrics}

Given a dataset $\Tcal_i$, we estimate the MNL, ranking-based, and decision forest models. The MNL model is estimated by the standard maximum likelihood estimation \citepecom{train2009discrete}. For the latter two, we follow the estimation procedures described in Section~\ref{sec:IRI}. In particular, for the decision forest model we again restrict the tree depth to three, and for the ranking-based model we limit the length of estimated rankings to at most three products above the no-purchase option. This restriction is imposed by adding an extra constraint to the column-generation subproblem when estimating the ranking-based model from data \citepecom{van2014market}\footnote{For example, adding the constraint $\sum_{j=1}^n x_{j0} \leq 3$ to Problem (12) of \citeecom{van2014market}.}. Under this setting, the class of ranking-based models we learn from data exactly contains the ground-truth model $P_\GT(\cdot \mid \cdot)$ when $\alpha_{\text{OL}} = 0$. Our empirical experience suggests that without imposing this length restriction, the ranking-based models lead to substantial performance degradation in assortment tasks examined in this section -- sometimes performing as poorly as the MNL model.

To assess the performance of each choice model in the assortment task, we propose a performance metric based on the optimality gap of the optimal assortment $S^{\Mcal}$ returned by the learned choice model $\Mcal$. Let $Z^*$ be the maximal revenue under the ground truth model $P_{\GT}$ and $Z^{\Mcal}$ be the expected revenue of assortment $S^{\Mcal}$ under the ground truth model. The optimality gap is thus defined as
\begin{align}
	\label{eq:sushi_optimality_gap}
	G_{\Mcal} = 100\% \times \left(Z^* - Z^{\Mcal}\right)/ Z^*.
\end{align} 
Note that the price information $\rho_i$ of each sushi type $i \in \Ncal$ needed to calculate the expected revenue is provided in the original dataset \citepecom{kamishima2003nantonac}

\YCRR{Before presenting the experimental results, we make additional remarks on the optimality gap metric~\eqref{eq:sushi_optimality_gap}. This metric can be interpreted as a measure of a model’s \emph{out-of-sample} performance in the data-to-decision assortment task. Recall that the observed assortments $\Scal_{\text{sample}}$ cover only about one-seventh of all possible assortments. Consequently, a choice model must be able to generalize its revenue predictions to unobserved assortments in order to make good decisions. More precisely, a choice model $\Mcal$ must extend the choice probability vector $P_{\Mcal}(\cdot \mid S)$ beyond the assortments in $\Scal_{\text{sample}}$, be able to compute the expected revenue of every possible assortment, and then identify the revenue-maximizing assortment $S^{\Mcal}$. In this sense, finding $S^{\Mcal}$ requires substantial generalizability beyond the training data.}

\subsection{Results}
\label{subsec:sushi_results}

\YCRR{Figure~\ref{figure:sushi_OPT_GAP} compares the out-of-sample revenue performance of the MNL, ranking-based, and decision forest models, as measured by the optimality gap~\eqref{eq:sushi_optimality_gap}. The x-axis shows the intensity of the choice overload effect $\alpha_{\text{OL}} \in \{ 0.005,0.01,0.015,0.02,0.025,0.03 \}$ tested in this  experiment. The y-axis shows the optimality gap $G_{\Mcal}$, with $\Mcal$ representing the MNL (MNL), ranking-based (RM), and decision forest (DF) models. Recall that for each $\alpha_{\text{OL}}$, we generate ten transaction datasets $\Tcal_i$, $i=1,\ldots,10$. Each dot in the figure represents the average $G_{\Mcal}$ across these datasets, with error bars indicating one standard deviation.}

\begin{figure}
	\centering
	\includegraphics[width=0.6\textwidth]{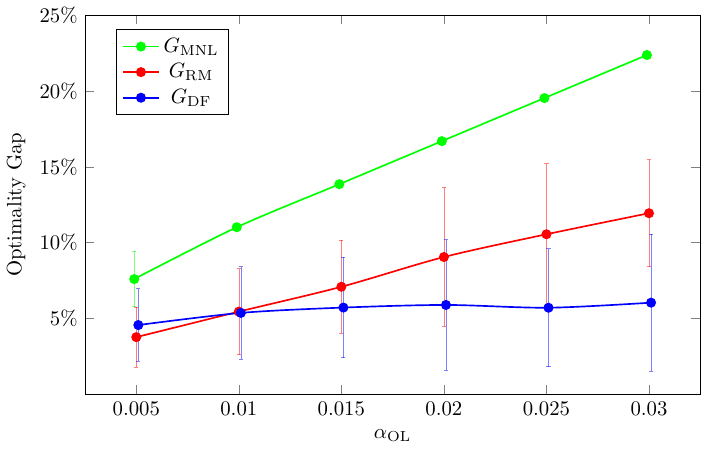}
	\caption{
		\YCRR{The performance of the MNL model, ranking-based model (RM) and the decision forest model (DF) measured by the optimality gap of the assortments they return. \label{figure:sushi_OPT_GAP}}
	}
\end{figure}

\YCRR{Among the three choice models, Figure~\ref{figure:sushi_OPT_GAP} shows that the MNL model performs the worst. Its optimality gap consistently increases with the intensity of choice overload, reaching roughly twice that of the ranking-based model. At the same time, the optimal assortment $S^{\text{MNL}}$ is remarkably stable compared to the other models, exhibiting zero standard deviation across the ten trials except when $\alpha_{\text{OL}} = 0.005$. This illustrates that, while the simplicity of the MNL model limits its ability to produce high-quality assortments, it offers an advantage in terms of stability for downstream operational tasks.}

\YCRR{When the choice overload is slight (i.e., $\alpha_{\text{OL}} = 0.005$), the ranking-based model performs best among the three. This is expected because, when $\alpha_{\text{OL}} = 0$, the ground-truth model~\eqref{eq:choice_overload_groundtruth} is a ranking-based model with length-three rankings -- exactly the class of model we estimate from the data (see Section~\ref{subsec:sushi_metrics}). The small increase to $\alpha_{\text{OL}} = 0.005$ does not materially affect its performance. The sub-optimality of the ranking-based model in this setting, around 4.5\%, can be viewed as a baseline for this experiment: it reflects the challenge of generalizing the assortment decision when the model observes only a small fraction (roughly one-seventh) of all assortments in the dataset.}

\YCRR{When $\alpha_{\text{OL}}$ is further increased to $0.02$, the ranking-based model’s performance noticeably degrades, and the decision forest model surpasses it. This illustrates that even a mild choice overload, where customers are only 2\% more likely to opt out of purchasing as the assortment size increases, can cause rational choice models, such as the ranking-based and MNL models, to suffer substantial sub-optimality, in this case at least 8\%. In contrast, the decision forest model remains robust: as $\alpha_{\text{OL}}$ increases from $0.01$ to $0.03$, its optimality gap stays around 5–6\%, while the gaps for the ranking-based and MNL models rise to 13\% and 22\%, respectively.}

\YCRR{These results highlight the flexibility of the decision forest model in general, non-rational choice regimes. When the ground-truth model is close to purely rational, more restricted models like the ranking-based model can outperform the decision forest model, as the latter’s flexibility may lead to overfitting and worse out-of-sample performance. However, this flexibility becomes an advantage under non-rational behavior: behavioral effects such as choice overload can significantly degrade the performance of rational choice models, whereas the decision forest model maintains consistent performance.}

\vspace{1.0em}
\begin{center}
	\fbox{
		\parbox{0.8\textwidth}{
		{\bfseries Sections~\ref{sec:proofs} and~\ref{sec:one_tree_analysis}} are not included in this electronic companion due to page limits. These materials are available at \url{https://arxiv.org/abs/2103.14067}.
		}
	}
\end{center}
\vspace{0.5em}

\makeatletter
\newcommand*\mysizeecom{%
  \@setfontsize\mysize{10.0}{9.0}%
}
\makeatother

\renewcommand{\bibfont}{\mysizeecom}
{\setlength{\bibsep}{3pt}
\bibliographystyleecom{plainnat}
\bibliographyecom{aodf_literature.bib}
}

\newpage

\section{Proofs}
\label{sec:proofs}

\subsection{Proof of Theorem~\ref{theorem:inapproximable_exponential_gap}}
\label{appendix:proof_AODF_inapproximable}

In what follows, we will transform the Max 1-in-E$k$SAT problem \citepapp{guruswami2005complexity} of $n$ variables into an instance of assortment problem $\AODF(k+1)$ over $n+1$ products. Notice that according to Theorem 9 of \citeapp{guruswami2005complexity}, unless P = NP, for any $\epsilon > 0$ and constant $k \geq 7$, there is no polynomial time  $( 2^{k - 2 \lceil  \sqrt{k}  \rceil } / (2ek) - \epsilon )$ approximation algorithm for the Max 1-in-E$k$SAT problem\footnote{\YCRR{The condition $k \geq 7$ originates from Theorem 7 of \citeapp{guruswami2005complexity}, which establishes the inapproximability of the Max $k$AND problem (first proved in \citeapp{samorodnitsky2000pcp}). Lemma 8 of \citeapp{guruswami2005complexity} then gives an approximation-preserving reduction from Max $k$AND to Max 1-in-E$k$SAT, thereby transferring the hardness result. Consequently, the $k \geq 7$ condition is implicitly assumed in Theorem 9 of \citeapp{guruswami2005complexity}.}}.

We first describe the Max 1-in-E$k$SAT problem, which is a variant of the SAT problem. In the Max 1-E$k$SAT problem, one has $n$ boolean variables $x_1,\ldots,x_n$ and a collection of clauses $C^1,C^2,\ldots,C^M$. Each clause consists of exactly $k$ literals and each literal is either one of the binary variables or its negation. A clause is satisfied if \emph{exactly} one literal is true. For example, consider the clause $C^1 = (  x_1, \neg x_2 , x_3  )$. The clause is satisfied if $(x_1,x_2,x_3) = (\texttt{true},\texttt{true},\texttt{false})$. However, $C^1$ is \emph{not} satisfied if $(x_1,x_2,x_3) = (\texttt{true},\texttt{false},\texttt{false})$, since in this case, the two literals $x_1$ and $\neg x_2$ are true. In the Max 1-in-E$k$SAT problem, one finds \YCRR{an} assignment of boolean variables $x_1,\ldots,x_n$ that maximizes the number of satisfied clauses.

Given an instance of the Max 1-in-E$k$SAT problem with $n$ variables, we show how the problem can be transformed into an instance of the assortment problem $\AODF(k+1)$ of $n+1$ products. First, we construct an instance of the latter problem such that each of its first $n$ products corresponds to a binary variable in the Max 1-in-E$k$SAT problem. Specifically, we let the $i$th product in $\AODF(k+1)$ correspond to variable $x_i$ in the Max 1-in-E$k$SAT problem, for $i = 1,2,\ldots,n$. The $(n+1)$-st product in $\AODF(k+1)$ is designed to ensure that the revenue of the assortment can correspond to the number of satisfied clauses. Furthermore, the action of assigning $\texttt{true}$ to variable $x_i$ in the Max 1-in-E$k$SAT problem corresponds to the action of including product $i$ in the assortment for $\AODF(k+1)$, implying that we will traverse to the left if we meet a split node of product $i$ in a decision tree of $\AODF(k+1)$. We assume that the marginal revenue of the products follows $\rho_1 = \ldots =\rho_n = 0$ and $\rho_{n+1} = M$, where $M$ is the number of clauses in the Max 1-in-E$k$SAT problem.

For each clause $C^m$ of the Max 1-in-E$k$SAT problem, we introduce a decision tree $t_m$ with depth $k+1$ and $O(k^2)$ leaf nodes that mimics the structure of $C^m$. We set the probability of tree $t_m$ to be $1/M$, so that $\lambda_{t_m} = 1/M$.

We then construct each tree as follows. Set the root node of $t_m$ to be a split node involving product $n+1$. The right child of the root node is set to be a leaf node with the no-purchase option $0$. We then construct a subtree by invoking the function $\texttt{GrowTree}$ on $C^m$ (formalized as Algorithm~\ref{alg:grow_subtree} below), and root this subtree, $\texttt{GrowTree}(C^m)$, at the left child of the root node of $t_m$; see Figure~\ref{fig:inapproximable_combined}. The function call $\texttt{GrowTree}(C^m)$ constructs a subtree according to the clause $C^m$ and the resulting tree $t_m$ will have the following property: an assortment will lead to a purchase decision of $n+1$ in tree $t_m$ if and only if we make exactly one of the literals true in $C^m$.

Note that Algorithm~\ref{alg:grow_subtree} will iteratively call a subroutine, $\texttt{GrowVine}$ (formalized as Algorithm~\ref{alg:grow_vine}), to build a subtree that corresponds to subclauses of clause $C^m$. To better explain the two algorithms, we introduce the following notation. We use $C^{m}_{j}$ to denote the $j$th literal of clause $C^m$ and $C^m_{j:k}$ to denote the subclause that consists of literals $C^m_j, C^m_{j+1},\ldots,C^m_{k}$. We further write $V(C^m_j)$ as the variable that appears in the literal $C^m_j$. For example, if $C^m = (  x_1, \neg x_2, x_3, \neg x_4 )$, then $C^m_2 = \neg x_2$, $C^m_{2:4} = ( \neg x_2, x_3, \neg x_4 ) $, and $V \left( C^m_2 \right) = 2$. Note that $C^m_{j:k}$ is simply an empty set if $j >k$. In the algorithms, we will introduce dummy variables $x_0$ and $x_{n+1}$ for convenience; however, they do not contribute to any real decision in the Max 1-in-E$k$SAT problem. With a slight abuse of notation, we use $v(s)$ to denote the product associated with a node $s$ in a decision tree. 

With this notation in hand, we now explain the mechanics of the function $\texttt{GrowVine}$. The function $\texttt{GrowVine}$ constructs a substructure, called a \emph{vine}, in a way that it only grows the split nodes in the \emph{opposite} directions of the literals. We design the vine to have the following property: one makes the purchase decision of $n+1$ in the vine if and only if one makes every literal in the subclause \emph{false}. Figure~\ref{fig:inapproximable_vine_algorithm} shows an example of a vine constructed according to subclause $( \neg x_2, x_3, \neg x_4  )$. The vine first places $2$ at its root, which \YCRR{is} the variable for the first literal $\neg x_2$ of the subclause. Since $x_2$ appears in negated form in the subclause, the vine branches to the left, which is the ``positive'' direction of the tree, to grow a split node for the next literal. The right child node of $2$ is set to a leaf node that is assigned to the no-purchase option. Now, we further proceed with the child split node of $2$, which will be associated with the next literal, $x_3$. We place $3$ at the node. Since $x_3$ appears in the non-negated form in the subclause, the vine branches to the right, which \YCRR{is} the ``negative'' direction of the tree, to grow a split node for the next literal. The left child node of $3$ is set to a leaf node that is assigned to the no-purchase option. The process repeats until we use all the literals in the subclause and place the purchase decision $n+1$ for the leaf which is reached if all the literals are false. Algorithm~\ref{alg:grow_vine} summarizes the detailed implementation in a pseudocode.

We now turn our attention back to \texttt{GrowTree}. Recall that Algorithm~\ref{alg:grow_subtree} grows a subtree, $\texttt{GrowTree}(C^m)$, according to the clause $C^m$, and that this subtree should be such that one reaches a leaf with the \YCRR{\emph{profitable}} purchase decision of $n+1$ if and only if one makes exactly one of the literals in the clause $C^m$ true. Figure~\ref{fig:inappximable_growtree_algo} shows an example of subtree $\texttt{GrowTree}(C^m)$ constructed according to the clause $C^m = (  x_1,\neg x_2, x_3, \neg x_4  )$. The algorithm first places $1$ at the root split node, which is the variable for the first literal $x_1$. The tree now grows in both directions. The left branch from the node with 1 corresponds to the condition that literal $x_1$ is already true, and so we construct a subtree rooted at the left child such that the profitable purchase node $n+1$ is reached if and only if the remaining literals $\neg x_2, x_3,$ and $\neg x_4$ are all made false. To this end, we use Algorithm~\ref{alg:grow_vine}, which takes the remaining literals as an input, to produce a vine $\texttt{GrowVine}( ( \neg x_2, x_3,\neg x_4 ) )$ as the left descendant of $1$. This vine is exactly the one in Figure~\ref{fig:inapproximable_vine_algorithm}. Now, we proceed to the right child of the split node with $1$. We set this node to a split node and place $2$ on this child node, which is the the variable for the second literal $\neg x_2$. We again grow the tree both to the left and right of the split node with $2$. For the right branch, which corresponds to the condition that in clause $C^m$ the first literal $x_1$ is already false and the second literal $\neg x_2$ is already true, we call Algorithm~\ref{alg:grow_vine} and place vine $\texttt{GrowVine}( (  x_3,\neg x_4 ) )$. We then move on to the left child node of 2 and repeat the process until we use all the literals. Figure~\ref{fig:inapproximable_algorithm_blocks} shows how we repeatedly call Algorithm~\ref{alg:grow_vine} in Algorithm~\ref{alg:grow_subtree} and Figure~\ref{fig:inapproximable_algorithm_details} shows the complete subtree. Figure~\ref{fig:inappximable_final_tree} shows the structure of the final tree $t_m$.

\begin{figure}[h!]
	\begin{subfigure}[t]{0.55\textwidth}
		\centering
		\includegraphics[height=1.5cm]{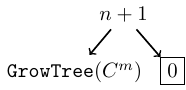}
		\caption{Tree constructed for clause $C^m$.}
		\label{fig:inapproximable_combined}
	\end{subfigure}
	\qquad
	\begin{subfigure}[t]{0.4\textwidth}
		\centering
		\includegraphics[height=3cm]{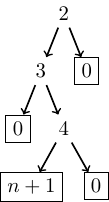}
		\caption{$\texttt{GrowVine}( ( \neg x_2 , x_3, \neg x_4   ) )$}
		\label{fig:inapproximable_vine_algorithm}
	\end{subfigure}%
	\caption{The construction of the tree $t_m$ and an example of the $\texttt{GrowVine}$ procedure (Algorithm~\ref{alg:grow_vine}) \YCRR{applying to a subclause $( \neg x_2, x_3, \neg x_4  )$}.}
	\label{fig:inappximable_tree}
\end{figure}

\begin{algorithm}
	\caption{GrowVine}
	\label{alg:grow_vine}
	\SingleSpacedXI
	\begin{algorithmic}[1]
		\REQUIRE subclause $C^{\text{sub}}$
		\STATE $C^{\text{sub}} \gets C^{\text{sub}} \cup \{ x_{n+1}  \}$ , \YCRR{to append $x_{n+1}$ to $C^{\text{sub}}$ at the end.}
		\STATE $s \gets$ root node
		\STATE $v(s) \gets  V\left(C^{\text{sub}}_{1}\right)$.
		\FOR{$q = 1, \ldots, |C^{\text{sub}}| - 1$}
		\STATE $(s_\text{L},s_\text{R}) \gets$ left and right child nodes of $s$.
		\IF{$C^{\text{sub}}_q$ is a negation}
		\STATE $\left(v(s_\text{L}), v(s_\text{R})\right) \gets \left( V\left(C^{\text{sub}}_{q+1}\right), 0   \right)$.
		\STATE $s \gets s_\text{L}$
		\ELSE
		\STATE $\left(v(s_\text{L}), v(s_\text{R})\right) \gets \left(  0, V\left(C^{\text{sub}}_{q+1}\right)   \right)$.
		\STATE $s \gets s_\text{R}$
		\ENDIF
		\ENDFOR
		\STATE Set all nodes with $0$ and $n+1$ as leaf nodes
	\end{algorithmic}
\end{algorithm}

\begin{figure}[h!]
	\begin{subfigure}[t]{0.35\textwidth}
		\centering
		\includegraphics[height=4cm]{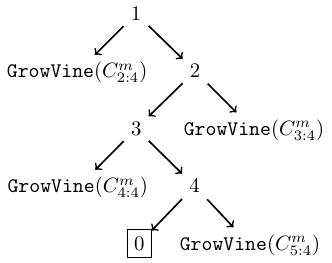}
		\caption{The \YCRR{growing} of $\texttt{GrowTree}(C^m)$}
		\label{fig:inapproximable_algorithm_blocks}
	\end{subfigure}%
	\qquad
	\begin{subfigure}[t]{0.6\textwidth}
		\centering
		\includegraphics[height=4cm]{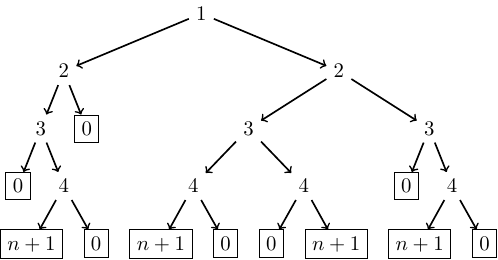}
		\caption{The resulting subtree $\texttt{GrowTree}(C^m)$}
		\label{fig:inapproximable_algorithm_details}
	\end{subfigure}	
	\caption{Example of the $\texttt{GrowTree}$ procedure (Algorithm~\ref{alg:grow_subtree}) for the clause $C^m = (  x_1,\neg x_2, x_3, \neg x_4  )$. 	\label{fig:inappximable_growtree_algo}}
\end{figure}

\begin{figure}[h!]
	\centering
	\includegraphics[width=0.5\textwidth]{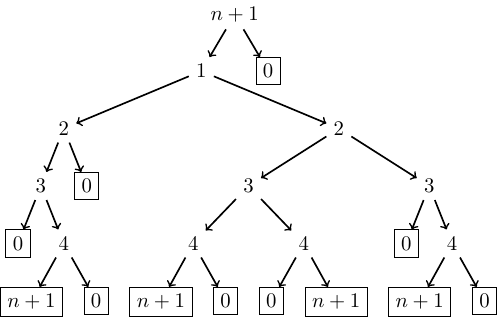}
	
	\caption{Final tree $t_m$ corresponding to clause $C^m =  (  x_1,\neg x_2, x_3, \neg x_4  )$. \label{fig:inappximable_final_tree}}
\end{figure}

\begin{algorithm}
	\caption{GrowTree}
	\label{alg:grow_subtree}
	\SingleSpacedXI
	\begin{algorithmic}
		\REQUIRE Clause $C$ and $k = |C|$
		\STATE $C \gets C \cup \{ x_0 \}$, added from behind.
		\STATE $s \gets$ root node
		\STATE $v(s) \gets V(C_1)$.
		\FOR{$q = 1, \ldots, k$}
		\STATE $(s_\text{L},s_\text{R}) \gets$ left and right child nodes of $s$.
		\IF{$C_q$ is a negation}
		\STATE Replace node $s_\text{R}$ by subtree $\texttt{GrowVine}(C_{q+1:k})$ via Algorithm~\ref{alg:grow_vine}.
		\STATE $v(s_\text{L}) \gets V(C_{q+1})$ 
		\STATE $s \gets s_\text{L}$
		\ELSE
		\STATE Replace node $s_\text{L}$ by subtree $\texttt{GrowVine}(C_{q+1:k})$ via Algorithm~\ref{alg:grow_vine}.
		\STATE $v(s_\text{R}) \gets V(C_{q+1})$ 
		\STATE $s \gets s_\text{R}$
		\ENDIF
		\ENDFOR
		\STATE Set all nodes with $0$ and $n+1$ \YCRR{other than the root} as leaf nodes
	\end{algorithmic}
\end{algorithm}

We note that the construction of the tree $t_m$ takes $O(k^2)$ \YCRR{runtime}, as we call out Algorithm~\ref{alg:grow_vine} at most $k$ times in Algorithm~\ref{alg:grow_subtree} and each \YCRR{call to} Algorithm~\ref{alg:grow_vine} uses $O(k)$ steps. One can easily show that the resulting tree $t_m$ has $2 + k(k+1)/2$ leaf nodes. Therefore, whenever $k \geq 2$, the tree $t_m$ has at most $k^2$ leaf nodes.

Finally, to complete the proof, it suffices to prove the following two claims. First, for any assignment $\xb$ to the Max 1-in-E$k$SAT problem with $z^{\text{SAT}}$ clauses satisfied, we can construct an assortment $S_{\xb}$ which has an expected revenue $z^{\text{SAT}}$ in the constructed instance of the assortment problem $\AODF(k+1)$. Second, for any assortment $S$ of expected revenue $z^{\AODF} > 0$ in the constructed instance of the assortment problem $\AODF(k+1)$, we can construct an assignment $\xb_S$ of variables to the the Max 1-in-E$k$SAT problem with $z^{\AODF}$ clauses satisfied. To prove the first claim, we simply construct an assortment $S_\xb$ that includes product $n+1$ and product $i $ if $x_i = \texttt{true}$ in the assignment $\xb$, for all $i = 1,\ldots,n$. Since tree $t_m$ contributes $1$ to the total expected revenue if and only if the claim $C^m$ is satisfied by $\xb$, we know that the expected revenue of the assortment is $z^{\text{SAT}}$.

To prove the second claim, we first note that assortment $S$ must have included product $n+1$, otherwise its expected revenue would be zero. We construct the assignment $\xb_S$ of variables to the Max 1-in-E$k$SAT problem by setting $x_i = \texttt{true}$ if product $i$ is offered in $S$. Furthermore, since each tree contributes either expected revenue $1$ or $0$ to the total expected revenue, there must be $z^{\AODF}$ trees that contributes to the revenue. For each such tree, the corresponding clause must be satisfied according to the tree construction. Therefore, the assignment $\xb$ satisfies $z^{\AODF}$ clauses. \hfill \Halmos

\subsection{Proof of Proposition~\ref{prop:randomized_approx_algo}}
\label{subsec:proof_prop:randomized_approx_algo}

Let $Z^*$ be the optimal objective value of the assortment problem $\AODF(d)$ and let $r^t_{\max}$ be the revenue of most expensive product in the leaf nodes of tree $t$. Call this leaf node $\ell^t_{\max}$; if there are multiple leaf nodes with the maximal revenue in tree $t$, we choose the leftmost one. Then we naturally have $Z^* \leq \YCRR{\sum_{t \in F}} \lambda_t \cdot r^t_{\max}$.

We construct a random assortment $S_X$ through the following procedure. First, we independently draw values for $X_1,X_2,\ldots,X_n$, where $X_i$ is an IID Bernoulli random variable with success rate $1/2$. We include product $i$ in the assortment $S_X$ if $X_i = 1$. Notice that given a decision tree of depth at most $d$, the probability of node $\ell^t_{\max}$ being chosen when $S_X$ is offered is
\begin{align*}
	\prod_{s \in \LS(\ell^t_{\max}) } P \left(  X_{v(s)} = 1 \right) \cdot \prod_{s \in \RS(\ell^t_{\max}) } P \left(  X_{v(s)} = 0 \right) = \prod_{s \in \LS(\ell^t_{\max}) }  \frac{1}{2} \cdot \prod_{s \in \RS(\ell^t_{\max}) } \frac{1}{2} \geq \frac{1}{2^d},
\end{align*}
where $\LS(\ell)$ and $\RS(\ell)$ are the sets of split nodes that consider leaf node $\ell$ as its left descendant and right descendant, respectively. The last inequality holds since the leaf node $\ell^t_{\max}$ would not have more than $d$ ancestors given that the tree $t$ is of depth at most $d$. Therefore, the expected revenue of assortment $S_X$ is
\begin{align*}
	E_{X} \left[   \YCRR{\sum_{t \in F}} \lambda_t \cdot R_t(S_X)  \right]  = \YCRR{\sum_{t \in F}} \lambda_t \cdot E_X \left[  R_t(S_X) \right]  \geq \YCRR{\sum_{t \in F}} \lambda_t \cdot P_{\ell^t_{\max}} \cdot r^t_{\max} \geq \YCRR{\sum_{t \in F}} {\lambda_t} \cdot \frac{r^t_{\max}}{2^d} \geq \frac{ Z^* }{2^d}
\end{align*}
where $R_t(S_X)$ is the revenue of tree $t$ under assortment $S_X$ and $P_{\ell^t_{\max}} $ is the probability that the tree $t$ chooses the leaf node $\ell^t_{\max}$ under assortment $S_X$. 

We can retrieve a deterministic assortment that approximates the optimal objective value $Z^*$ within \YCRR{a} factor $2^d$ by a derandomization procedure. Specifically, the derandomization procedure is based on iteratively evaluating the conditional expectations. Let $R(S) = \sum_{t \in F} \lambda_t \cdot R_t(S)$ be the revenue of assortment $S$. Starting from the first product, we compare the conditional expectations $E_X \left[  R(S_X)  \mid X_1 = x_1 \right]$ for $x_1 \in \{ 0,1 \}$. If $ E_X \left[  R(S_X)  \mid X_1 = 1 \right] > E_X \left[  R(S_X)  \mid X_1 = 0 \right] $, we set $x'_1 = 1$; otherwise, we set $x'_1 = 0$. Suppose we have determined $x'_1,\ldots,x'_{m}$. For the next product, we evaluate $E_X \left[  R(S_X)  \mid  X_1 = x'_1 , \ldots, X_m = x'_m, X_{m+1} = x_{m+1} \right]$ for $x_{m+1} \in \{ 0,1 \}$. If this conditional expectation is larger when $X_{m+1} = 1$, we set $x'_{m+1} = 1$; otherwise, we set $x'_{m+1} = 0$. Repeating this process for all $n$ variables yields a deterministic binary vector $(x'_1,\ldots,x'_n)$, which encodes an assortment $S'$ with revenue 
\begin{align*}
	R(S') & = E_X \left[  R(S_X) \mid X_1 = x'_1,\ldots,X_n = x'_n \right]  \\   
	& \geq  E_X \left[  R(S_X) \mid X_1 = x'_1,\ldots,X_n = x'_{n-1} \right]  \\ 
	& \geq \dots  \\ 
	& \geq E_X \left[  R(S_X) \mid X_1 = x'_1 \right] \\ 
	& \geq   E_X \left[  R(S_X) \right] \\ 
	& \geq Z^* / 2^d.
\end{align*}

To complete the proof, we show how to compute $E_X \left[  R(S_X)  \mid  X_1 = x'_1 , \ldots, X_m = x'_m \right]$ for any $m \in \{ 0,1,\ldots,n \}$ and partial assignment $(x'_1,\ldots,x'_m) \in \{ 0,1 \}^n$. By linearity of expectation,
\begin{equation*}
	E_X \left[  R(S_X)  \mid  X_1 = x'_1 , \ldots, X_m = x'_m \right] = \sum_{t \in F} \lambda_t \cdot E_X \left[  R_t(S_X)  \mid  X_1 = x'_1 , \ldots, X_m = x'_m \right]. 
\end{equation*}
Thus, it suffices to compute the conditional expectation $E_X \left[  R_t(S_X)  \mid  X_1 = x'_1 , \ldots, X_m = x'_m \right] $ for each tree $t \in F$. The procedure is straightforward: we traverse tree $t$ from the leaves to the root. For each split node $s \in \splits(t)$, let $\bar{r}_{t,s}$ denote its conditional revenue, indicating the revenue of an assortment that hits split node $s$ in the decision making process. If the corresponding product $i' = v(t,s)$ of the split node has already been fixed (i.e., $i' \leq m$), then the branch direction is determined. Specifically, if $x_{i'} = 1$, the conditional random assortment $S_{X \mid X_1 = x'_1,\ldots,X_m = x'_m}$ always follows the left child, so $\bar{r}_{t,s} =  \bar{r}_{t,\textbf{LC}(s)}$, where $\textbf{LC}(s)$ denotes the left child node of $s$. Otherwise, if $x_{i'} = 0$, the conditional random assortment always follows the right child, so $\bar{r}_{t,s} =  \bar{r}_{t,\textbf{RC}(s)}$, where $\textbf{RC}(s)$ denotes the right child node of $s$. If the product has not been fixed ($i' > m$), the random assortment branches left or right with equal probability, yielding $\bar{r}_{t,s} =  (1/2) \cdot \left(  \bar{r}_{t,\textbf{LC}(s)} + \bar{r}_{t,\textbf{RC}(s)}\right) $. This recursive evaluation gives $E_X \left[R_t(S_X) \mid X_1 = x_1,\ldots, X_m = x_m\right]$, as summarized in Algorithm~\eqref{algorithm:conditional_expectation}. In the algorithm, we use $\splits(t,d) = \{ s \in \splits(t) \mid d(s) = d\}$ to denote the set of all splits at a particular depth~$d$.

Note that this recursive evaluation to the conditional expectation is valid under Assumption~\ref{assumption:Requirement_3}. Particularly, this assumption states that no two split nodes $s$ and $s'$ in the same tree $t$, where $s'$ is a descendant of $s$, are assigned the same product (i.e., $v(t,s) \neq v(t,s')$). As a result, the random coin flips associated with the same product---although they may occur at multiple split nodes within the same tree---can be treated as independent copies of Bernoulli random variables with success probability 1/2. \hfill \Halmos

\begin{algorithm}
	\SingleSpacedXI
	\small
	\begin{algorithmic}[1]
		\STATE $\bar{r}_{t,\ell} \gets r_{t,\ell}$ for all $\ell \in \leaves(t)$
		\FOR{ $d = d_{\max}, d_{\max} - 1, \ldots, 1$ }
		\FOR{ $s \in \splits(t,d)$ }
		\IF{ $v(t,s) \leq m$}
		\STATE Set $\bar{r}_{t,s} =  \bar{r}_{t,\textbf{LC}(s)} \cdot x'_{v(t,s)} + \bar{r}_{t,\textbf{RC}(s)} \cdot \left(1 - x'_{v(t,s)}\right) $
		\ELSE
		\STATE Set $\bar{r}_{t,s} = (1/2)  \cdot \left(  \bar{r}_{t,\textbf{LC}(s)} +  \bar{r}_{t,\textbf{RC}(s)} \right)$
		\ENDIF
		\ENDFOR
		\ENDFOR
	\end{algorithmic}
	
	\caption{Conditional expectation $E_X \left[R_t(S_X) \mid X_1 = x'_1,\ldots, X_m = x'_m)\right]$ of tree $t$. \label{algorithm:conditional_expectation}}
\end{algorithm}

\subsection{Proof of Theorem~\ref{theorem:SplitMIO_benders_primal_greedy_BFS}}
\label{proof:SplitMIO_benders_primal_greedy_BFS}

\emph{Proof of part (a) (feasibility)}: By definition, the solution produced by Algorithm~\ref{algorithm:SplitMIO_primal_greedy} produces a solution $\yb_t$ that satisfies the left and right split constraints~\eqref{prob:SplitMIO_sub_primal_left} and \eqref{prob:SplitMIO_sub_primal_right}. With regard to the nonnegativity constraint~\eqref{prob:SplitMIO_sub_primal_nonnegative}, we can see that at each stage of Algorithm~\ref{algorithm:SplitMIO_primal_greedy}, the quantities $x_{v(t,s)} - \sum_{\ell \in \leftleaves(s)} y_{t,\ell}$, $1 - x_{v(t,s)} - \sum_{\ell \in \rightleaves(s)} y_{t,\ell}$ and $1 - \sum_{\ell \in \leaves(t)} y_{t,\ell}$ never become negative; thus, the solution $\yb_t$ produced upon termination satisfies the nonnegativity constraint~\eqref{prob:SplitMIO_sub_primal_nonnegative}. 

The only constraint that remains to be verified is constraint~\eqref{prob:SplitMIO_sub_primal_unitsum}, which requires that $\yb_t$ adds up to 1. Observe that it is sufficient for a $C$ event to occur during the execution of Algorithm~\ref{algorithm:SplitMIO_primal_greedy} to ensure that constraint~\eqref{prob:SplitMIO_sub_primal_unitsum} is satisfied. We will show that a $C$ event must occur during the execution of Algorithm~\ref{algorithm:SplitMIO_primal_greedy}. 

We proceed by contradiction. For the sake of a contradiction, let us suppose that a $C$ event does not occur during the execution of the algorithm. Note that under this assumption, for any split $s$, it is impossible that the solution $\yb_t$ produced by Algorithm~\ref{algorithm:SplitMIO_primal_greedy} satisfies $x_{v(t,s)}  = \sum_{\ell \in \leftleaves(s)} y_{t,\ell}$ and $1 - x_{v(t,s)}  = \sum_{\ell \in \rightleaves(s)} y_{t,\ell}$, 
as this would imply that $\sum_{\ell \in \leftleaves(s)} y_{t,\ell} + \sum_{\ell \in \rightleaves(s)} y_{t,\ell} = x_{v(t,s)} + 1 - x_{v(t,s)} = 1$; 
by the definition of Algorithm~\ref{algorithm:SplitMIO_primal_greedy}, this would have triggered a $C$ event at one of the leaves in $\leftleaves(s) \cup \rightleaves(s)$. 

Thus, this means that at every split, either $\sum_{\ell \in \leftleaves(s)} y_{t,\ell} < x_{v(t,s)}$ or $\sum_{\ell \in \rightleaves(s)} < 1 - x_{v(t,s)}$. Using this property, let us identify a leaf $\ell^*$ using the following procedure:
\vspace{1em}

\begin{center}
	\fbox{
		\parbox{0.8\textwidth}{
			\paragraph{Procedure 1:}
			\begin{enumerate}
				\item Set $j \gets \root(t)$. 
				\item If $j \in \leaves(t)$, terminate with $\ell^* = j$; otherwise, proceed to Step 3.
				\item If $\sum_{\ell \in \leftleaves(j)} y_{t,\ell} < x_{v(t,j)}$, set $j \gets \leftchild(j)$; otherwise, set $j \gets \rightchild(j)$.
				\item Repeat Step 2. 
			\end{enumerate}
		}
	}
\end{center}
\vspace{1em}

Note that by our observation that at most one of the left or right split constraints can be satisfied at equality for any split $s$, Procedure 1 above is guaranteed to terminate with a leaf $\ell^*$ such that:
\begin{align*}
	& y_{t,\ell^*} \leq \sum_{\ell \in \leftleaves(s)} y_{t,\ell} < x_{v(t,s)}, \quad \forall\ s \in \splits(t)\ \text{such that} \ \ell^* \in \leftleaves(s),\\
	& y_{t,\ell^*} \leq \sum_{\ell \in \rightleaves(s)} y_{t,\ell} < 1 - x_{v(t,s)}, \quad \forall\ s \in \splits(t)\ \text{such that} \ \ell^* \in \rightleaves(s).
\end{align*}
However, this is impossible, because Algorithm~\ref{algorithm:SplitMIO_primal_greedy} always sets each leaf $y_{t,\ell}$ to the highest value it can be without violating any of the left or right split constraints; the above conditions imply that $y_{t,\ell^*}$ could have been set higher, which is not possible. We thus have a contradiction, and it must be the case that a $C$ event occurs.

\noindent \emph{Proof of part (b) (extreme point)}: To show that $\yb_t$ is an extreme point, let us assume that $\yb_t$ is not an extreme point. Then, there exist feasible solutions $\yb^1_t$ and $\yb^2_t$ different from $\yb_t$ and a weight $\theta \in (0,1)$ such that $\yb_t = \theta \yb^1_t + (1 - \theta) \yb^2_t$. Let $\ell^*$ be the first leaf checked by Algorithm~\ref{algorithm:SplitMIO_primal_greedy} at which $y_{t,\ell^*} \neq y^1_{t,\ell^*}$ and $y_{t,\ell^*} \neq y^2_{t,\ell^*}$. Such a leaf must exist because $\yb_t \neq \yb^1_t$ and $\yb_t \neq \yb^2_t$, and because $\yb_t$ is the convex combination of $\yb^1_t$ and $\yb^2_t$. Without loss of generality, let us further assume that $y^1_{t,\ell^*} < y_{t,\ell^*} < y^2_{t,\ell^*}$.

By definition, Algorithm~\ref{algorithm:SplitMIO_primal_greedy} sets $y_{t,\ell}$ at each iteration to the largest it can be without violating the left split constraints~\eqref{prob:SplitMIO_sub_primal_left} and the right split constraints~\eqref{prob:SplitMIO_sub_primal_right}, and ensuring that $\sum_{\ell \in \leaves(t)} y_{t,\ell}$ does not exceed 1. Since $y^2_{t,\ell^*} > y_{t,\ell^*}$, and since $\yb^2_t$ and $\yb_t$ are equal for all leaves checked before $\ell^*$, this implies that $\yb^2_t$ either violates constraint~\eqref{prob:SplitMIO_sub_primal_left}, violates constraint~\eqref{prob:SplitMIO_sub_primal_right}, or is such that $\sum_{\ell \in \leaves(t)} y_{t,\ell} > 1$. This implies that $\yb^2_t$ cannot be a feasible solution, which contradicts the assumption that $\yb^2_t$ is a feasible solution. \Halmos

\subsection{Proof of Theorem~\ref{theorem:SplitMIO_benders_dual_greedy_BFS}}
\label{proof:SplitMIO_benders_dual_greedy_BFS}

\emph{Proof of part (a) (feasibility)}: Before we prove the result, we first establish a helpful property of the events that are triggered during the execution of Algorithm~\ref{algorithm:SplitMIO_primal_greedy}.

\begin{lemma}
	Let $s_1, s_2 \in \splits(t)$, $s_1 \neq s_2$, such that $s_2$ is a descendant of $s_1$. Suppose that $e_1 = A_{s_1}$ or $e_1 = B_{s_1}$, and that $e_2 = A_{s_2}$ or $e_2 = B_{s_2}$. If $e_1$ and $e_2$ occur during the execution of Algorithm~\ref{algorithm:SplitMIO_primal_greedy}, then $r_{t,f(e_1)} \leq r_{t,f(e_2)}$.
	
	\label{lemma:descendant_split_implies_higher_revenue}
\end{lemma}

\begin{proof}{Proof:}
	We will prove this by contradiction. Suppose that we have two splits $s_1$ and $s_2$ and events $e_1$ and $e_2$ as in the statement of the lemma, and that $r_{t,f(e_2)} < r_{t,f(e_1)}$. This implies that leaf $f(e_1)$ is checked before leaf $f(e_2)$. When leaf $f(e_1)$ is checked, the event $e_1$ occurs, which implies that either the left split constraint~\eqref{prob:SplitMIO_sub_primal_left} becomes tight (if $e_1 = A_{s_1}$) or the right split constraint~\eqref{prob:SplitMIO_sub_primal_right} becomes tight (if $e_1 = B_{s_1}$) at $s_1$. In either case, since $s_2$ is a descendant of $s_1$, the leaf $f(e_2)$ must be contained in the left leaves of split $s_1$ (if $e = A_{s_1}$) or the right leaves of split $s_1$ (if $e_1 = B_{s_1}$). Thus, when leaf $f(e_2)$ is checked, the event $e_2$ cannot occur, because $q_{s_1}$ in Algorithm~\ref{algorithm:SplitMIO_primal_greedy} will be zero (implying that $q_{A,B} = 0$), and so $s^*$ cannot be equal to $s_2$ because $s_1$ is a shallower split that attains the minimum of $q_{A,B} = 0$.  \Halmos
\end{proof}

To establish that $(\alphab_t, \betab_t, \gamma_t)$ is feasible for the \SplitMIO dual subproblem~\eqref{prob:SplitMIO_sub_dual}, we will first show that the $\alpha_{t,s}$ variables are nonnegative. Fix $s \in \splits(t)$. If $A_s \notin \Ecal$, then $\alpha_{t,s} = 0$, and constraint~\eqref{prob:SplitMIO_sub_dual_alphageqzero} is satisfied. If $A_s \in \Ecal$, then consider the split $\tilde{s} = \arg \min_{s'} \left[ \{ d(s') \ \mid \ s' \in \LS(f(A_s)),\ d(s') < d,\ A_{s'} \in \Ecal\} \cup \{ d(s') \ \mid \ s' \in \RS(f(A_s)),\ d(s') < d, \ B_{s'} \in \Ecal \} \right]$, 
where we recall that $d = d(s)$ is the depth of split $s$. In words, $\tilde{s}$ is the shallowest split (i.e., closest to the root) along the path of splits from the root node to split $s$ such that either an $A_{\tilde{s}}$ event occurs or a $B_{\tilde{s}}$ event occurs for split $\tilde{s}$. There are three possible cases that can occur here, which we now handle. 

\noindent \textbf{Case 1}: $\tilde{s} \in \LS(f(A_s))$. In this case, $A_{\tilde{s}} \in \Ecal$, and we have
\begin{align*}
	\alpha_{t,s} & = r_{t,f(A_s)} - \left[ \sum_{ \substack{ s' \in \LS(f(A_s)): \,\, d(s') < d, \,\, A_{s'} \in \Ecal}} \alpha_{t,s'} + \sum_{ \substack{s' \in \RS(f(A_s)):\,\, d(s') < d, \,\, B_{s'} \in \Ecal}} \beta_{t,s'} + \gamma_t \right] \\
	& = r_{t,f(A_s)} - \left[ \alpha_{t,\tilde{s}} + \sum_{ \substack{ s' \in \LS(f(A_s)): \,\, d(s') < d(\tilde{s}), \,\, A_{s'} \in \Ecal}} \alpha_{t,s'} + \sum_{ \substack{s' \in \RS(f(A_s)):\,\, d(s') < d(\tilde{s}), \,\, B_{s'} \in \Ecal}} \beta_{t,s'} + \gamma_t  \right] \\
	& = r_{t,f(A_s)} - \left[ \alpha_{t,\tilde{s}} + \sum_{ \substack{ s' \in \LS(f(A_{\tilde{s}})): \,\, d(s') < d(\tilde{s}), \,\, A_{s'} \in \Ecal}} \alpha_{t,s'} + \sum_{ \substack{s' \in \RS(f(A_{\tilde{s}})):\,\, d(s') < d(\tilde{s}), \,\, B_{s'} \in \Ecal}} \beta_{t,s'} + \gamma_t \right]  \\
	&  = r_{t, f(A_s)} - r_{t, f(A_{\tilde{s}})}  \\
	&  \geq 0,
\end{align*}
where the first step follows by the definition of $\alpha_{t,s}$ in Algorithm~\ref{algorithm:SplitMIO_dual_greedy}; the second step follows by the definition of $\alpha_{t, \tilde{s}}$ as the deepest split for which an $A$ or $B$ event occurs that is at a depth lower than $s$; the third step by the fact that the left splits and right splits of $f(A_{\tilde{s}})$ at a depth below $d(\tilde{s})$ are the same as the left and right splits of $f(A_s)$ at a depth below $d(\tilde{s})$; and the fourth step follows from the definition of $\alpha_{t,\tilde{s}}$ in Algorithm~\ref{algorithm:SplitMIO_dual_greedy}. The inequality follows by Lemma~\ref{lemma:descendant_split_implies_higher_revenue}.

\noindent \textbf{Case 2}: $\tilde{s} \in \RS(f(A_s))$. In this case, $B_{\tilde{s}} \in \Ecal$, and analogously to Case 1, we have:
\begin{align*}
	\alpha_{t,s} & = r_{t,f(A_s)} - \left[ \sum_{ \substack{ s' \in \LS(f(A_s)): \,\, d(s') < d, \,\, A_{s'} \in \Ecal}} \alpha_{t,s'} + \sum_{ \substack{s' \in \RS(f(A_s)):\,\, d(s') < d, B_{s'} \in \Ecal}} \beta_{t,s'} + \gamma_t \right] \\
	& = r_{t,f(A_s)} - \left[ \beta_{t,\tilde{s}} + \sum_{ \substack{ s' \in \LS(f(A_s)): \,\, d(s') < d(\tilde{s}), \,\, A_{s'} \in \Ecal}} \alpha_{t,s'} + \sum_{ \substack{s' \in \RS(f(A_s)):\,\, d(s') < d(\tilde{s}),\,\, B_{s'} \in \Ecal}} \beta_{t,s'} + \gamma_t  \right] \\
	& = r_{t,f(A_s)} - \left[ \beta_{t,\tilde{s}} + \sum_{ \substack{ s' \in \LS(f(B_{\tilde{s}})): \,\, d(s') < d(\tilde{s}), \,\, A_{s'} \in \Ecal}} \alpha_{t,s'} + \sum_{ \substack{s' \in \RS(f(B_{\tilde{s}})):\,\, d(s') < d(\tilde{s}),\,\, B_{s'} \in \Ecal}} \beta_{t,s'} + \gamma_t \right]  \\
	&  = r_{t, f(A_s)} - r_{t, f(B_{\tilde{s}})} \\
	& \geq 0.
\end{align*}

\noindent \textbf{Case 3}: $\tilde{s}$ is undefined because the underlying sets are empty. In this case, $\alpha_{t,s} = r_{t,f(A_s)} - \gamma_t$, and we have $\alpha_{t,s} = r_{t,f(A_s)} - \gamma_t = r_{t,f(A_s)} - r_{t,f(C)} \geq 0$, 
where the inequality follows because $f(C)$ is the last leaf to be checked before Algorithm~\ref{algorithm:SplitMIO_primal_greedy} terminates, and thus it must be that $r_{t,f(A_s)} \geq r_{t,f(C)}$. This establishes that $(\alphab_t, \betab_t, \gamma_t)$ satisfy constraint~\eqref{prob:SplitMIO_sub_dual_alphageqzero}. Constraint~\eqref{prob:SplitMIO_sub_dual_betageqzero} can be shown in an almost identical fashion; for brevity, we omit the steps. We thus only need to verify constraint~\eqref{prob:SplitMIO_sub_dual_r}. Let $\ell \in \leaves(t)$. Here, there are four mutually exclusive and collectively exhaustive cases to consider.

\noindent \textbf{Case 3.1}: $r_{t,\ell} \leq r_{t,f(C)}$. In this case we have $\sum_{s \in \LS(\ell)} \alpha_{t,s} + \sum_{s \in \RS(\ell)} \beta_{t,s} + \gamma_t  \geq \gamma_t = r_{t,f(C)} \geq r_{t,\ell}$.

\noindent \textbf{Case 3.2}: $r_{t,\ell} > r_{t,f(C)}$ and $\ell = f(A_s)$ for some $s \in \splits(t)$. In this case, we have
\begin{align*}
	\sum_{s' \in \LS(\ell)} \alpha_{t,s'} + \sum_{s' \in \RS(\ell)} \beta_{t,s'} + \gamma_t \geq \alpha_{t,s} + \sum_{ \substack{s' \in \LS(\ell): \\ d(s') < d(s), \,\, A_{s'} \in \Ecal} } \alpha_{t,s'} + \sum_{ \substack{ s' \in \RS(\ell): \\ d(s') < d(s), \,\, B_{s'} \in \Ecal }} \beta_{t,s'} + \gamma_t = r_{t, f(A_s)}  = r_{t,\ell},
\end{align*}
where the first step follows by the nonnegativity of $\alpha_{t,s'}$ and $\beta_{t,s'}$ for all $s'$, and the second step by the definition of $\alpha_{t,s}$ in Algorithm~\ref{algorithm:SplitMIO_dual_greedy}. 

\noindent \textbf{Case 3.3}: $r_{t,\ell} > r_{t,f(C)}$ and $\ell = f(B_s)$ for some $s \in \splits(t)$. By similar logic as case 2, we have
\begin{align*}
	\sum_{s' \in \LS(\ell)} \alpha_{t,s'} + \sum_{s' \in \RS(\ell)} \beta_{t,s'} + \gamma_t  \geq \beta_{t,s} + \sum_{ \substack{s' \in \LS(\ell): \\ d(s') < d(s), \,\, A_{s'} \in \Ecal} } \alpha_{t,s'} + \sum_{ \substack{ s' \in \RS(\ell): \\ d(s') < d(s), \,\, B_{s'} \in \Ecal }} \beta_{t,s'} + \gamma_t = r_{t, f(B_s)}  = r_{t,\ell}.
\end{align*}

\noindent \textbf{Case 3.4}: $r_{t,\ell} > r_{t,f(C)}$ and $\ell$ is not equal to $f(A_s)$ or $f(B_s)$ for any split $s$. In this case, when leaf $\ell$ is checked by Algorithm~\ref{algorithm:SplitMIO_primal_greedy}, the algorithm reaches line~17 where $s^*$ is determined and $e$ is set to either $A_{s^*}$ or $B_{s^*}$, and it turns out that $e$ is already in $\Ecal$. If $e = A_{s^*}$, then this means that leaf $f(A_{s^*})$ was checked before leaf $\ell$, and that $r_{t,\ell} \leq r_{t,f(A_{s^*})}$. We thus have
\begin{align*}
	\sum_{s \in \LS(\ell)} \alpha_{t,s} + \sum_{s \in \RS(\ell)} \beta_{t,s} + \gamma_t
	& \geq \alpha_{t,s^*} + \sum_{ \substack{s \in \LS(\ell): \,\, d(s) < d(s^*), \,\, A_{s} \in \Ecal} } \alpha_{t,s} + \sum_{ \substack{ s \in \RS(\ell): \,\, d(s) < d(s^*), \,\, B_{s} \in \Ecal }} \beta_{t,s} + \gamma_t \\
	& = \alpha_{t,s^*} + \sum_{ \substack{s \in \LS(f(A_{s^*})): \\ d(s) < d(s^*), \,\, A_{s} \in \Ecal} } \alpha_{t,s} + \sum_{ \substack{s \in \RS(f(A_{s^*})): \\ d(s) < d(s^*), \,\, B_{s} \in \Ecal }} \beta_{t,s} + \gamma_t  \\
	& = r_{t, f(A_{s^*})} \\
	& \geq r_{t,\ell},
\end{align*}
where the first equality follows because $\ell$ and $f(A_{s^*})$, by virtue of being to the left of $s^*$, share the same left and right splits at depths lower than $d(s^*)$. Similarly, if $e = B_{s^*}$, then the leaf $f(B_{s^*})$ was checked before $\ell$, which means that $r_{t,\ell} \leq r_{t,f(B_{s^*})}$; in this case, we have
\begin{align*}
	\sum_{s \in \LS(\ell)} \alpha_{t,s} + \sum_{s \in \RS(\ell)} \beta_{t,s} + \gamma_t 
	& \geq \beta_{t,s^*} + \sum_{ \substack{s \in \LS(\ell): \,\, d(s) < d(s^*), \,\, A_{s} \in \Ecal} } \alpha_{t,s} + \sum_{ \substack{ s \in \RS(\ell): \,\, d(s) < d(s^*), \,\, B_{s} \in \Ecal }} \beta_{t,s} + \gamma_t \\
	& \geq \beta_{t,s^*} + \sum_{ \substack{s \in \LS(f(B_{s^*})): \\ d(s) < d(s^*), \,\, A_{s} \in \Ecal} } \alpha_{t,s} + \sum_{ \substack{s \in \RS(f(B_{s^*})): \\ d(s) < d(s^*), \,\, B_{s} \in \Ecal }} \beta_{t,s} + \gamma_t \\
	& = r_{t, f(B_{s^*})} \\
	& \geq r_{t,\ell}.
\end{align*}

We have thus shown that $(\alphab_t, \betab_t, \gamma_t)$ is a feasible solution to the \SplitMIO dual subproblem~\eqref{prob:SplitMIO_sub_dual}. \\

\emph{Proof of part (b) (extreme point)}: To prove this, we will use the equivalence between extreme points and basic feasible solutions, and show that $(\alphab_t, \betab_t, \gamma_t)$ is a basic feasible solution of problem~\eqref{prob:SplitMIO_sub_dual}. 

Define the sets $L_A = \{ \ell \in \leaves(t) \mid \ell = f(A_s) \ \text{for some}\ s \in \splits(t) \} $ and $L_B = \{ \ell \in \leaves(t) \mid \ell = f(B_s) \ \text{for some}\ s \in \splits(t) \} $. 
Consider the following system of equations:
\begin{align}
	& \sum_{s \in \LS(\ell)} \alpha_{t,s} + \sum_{s \in \RS(\ell)} \beta_{t,s} + \gamma_t = r_{t,\ell}, \quad \forall \ell \in L_A, \label{eq:SplitMIO_dual_BFS_r_A} \\
	& \sum_{s \in \LS(\ell)} \alpha_{t,s} + \sum_{s \in \RS(\ell)} \beta_{t,s} + \gamma_t = r_{t,\ell}, \quad \forall \ell \in L_B, \label{eq:SplitMIO_dual_BFS_r_B} \\
	& \sum_{s \in \LS(f(C))} \alpha_{t,s} + \sum_{s \in \RS(f(C))} \beta_{t,s} + \gamma_t = r_{t,f(C)}, \label{eq:SplitMIO_dual_BFS_r_C} \\
	& \alpha_{t,s} = 0, \quad \forall \ s \ \text{such that}\ A_{s} \notin \Ecal, \label{eq:SplitMIO_dual_BFS_alpha} \\
	& \beta_{t,s} = 0, \quad \forall \ s \ \text{such that}\ B_{s} \notin \Ecal. \label{eq:SplitMIO_dual_BFS_beta}
\end{align}
Observe that each equation corresponds to a constraint from problem~\eqref{prob:SplitMIO_sub_dual} made to hold at equality. In addition, we note that there are $|L_A| + |L_B| + 1 + ( |\splits(t)| - |L_A|) + (|\splits(t)| - |L_B|) = 2 |\splits(t)| + 1$ equations, which is exactly the number of variables. We will show that the unique solution implied by this system of equations is exactly the solution $(\alphab_t, \betab_t, \gamma_t)$ that is produced by Algorithm~\ref{algorithm:SplitMIO_dual_greedy}. 

In order to establish this, we first establish a couple of useful results.

\begin{lemma}
	Suppose that $e \in \Ecal$, $\ell = f(e)$ and $e = A_s$ or $e = B_s$ for some $s \in \splits(t)$. Then: (a) $A_{s'} \notin \Ecal$ for all $s' \in \LS(\ell)$ such that $d(s') > d(s)$; and (b) $B_{s'} \notin \Ecal$ for all $s' \in \RS(\ell)$ such that $d(s') > d(s)$. 
	\label{lemma:alpha_beta_below_e_is_irrelevant}
\end{lemma}

\begin{proof}{Proof of Lemma~\ref{lemma:alpha_beta_below_e_is_irrelevant}:}
	We will prove this by contradiction. Without loss of generality, let us suppose that there exists an $A_{s'}$ event in $\Ecal$ where $s' \in \LS(\ell)$ and $d(s') > d(s)$. (The case where there exists an $B_{s'}$ event in $\Ecal$ where $s' \in \RS(\ell)$ and $d(s') > d(s)$ can be shown almost identically.)
	
	Since $A_{s'} \in \Ecal$, consider the leaf $\ell' = f(A_{s'})$. There are now two possibilities for when Algorithm~\ref{algorithm:SplitMIO_primal_greedy} checks leaf $\ell'$:
	\begin{enumerate}
		\item \textbf{Case 1}: Leaf $\ell'$ is checked after leaf $\ell$. In this case, in the iteration of Algorithm~\ref{algorithm:SplitMIO_primal_greedy} corresponding to leaf $\ell'$, it will be the case that $q_s = 0$ because the left constraint~\eqref{prob:SplitMIO_sub_primal_left} at split $s$ (if $e = A_s$) or the right constraint~\eqref{prob:SplitMIO_sub_primal_right} at split $s$ (if $e = B_s$) became tight when leaf $\ell$ was checked. As a result, $q_{A,B} = 0$ in the iteration for leaf $\ell'$. This implies that $s'$ cannot be the lowest depth split that attains the minimum $q_s$ value of $q_{A,B}$, because $q_s = 0$, and $s$ has a depth lower than $s'$, which contradicts the fact that the $A_{s'}$ event occurred.
		\item \textbf{Case 2}: Leaf $\ell'$ is checked before leaf $\ell$. In this case, consider the value of $q_s$ when leaf $\ell$ is checked in Algorithm~\ref{algorithm:SplitMIO_primal_greedy}. 
		
		If $q_s > 0$, then there is immediately a contradiction, because $q_{s'} = 0$ when leaf $\ell$ is checked (this is true because the left split constraint~\eqref{prob:SplitMIO_sub_primal_left} at $s'$ became tight after leaf $\ell'$ was checked), and thus it is impossible that $s^* = s$. If $q_s = 0$, then this implies that $x_{v(t,s)} = 0$. This would imply that $q_s = 0$ when leaf $\ell'$ was checked, which would imply that $s^*$ cannot be $s'$ when leaf $\ell'$ is checked because $s$ is at a lower depth than $s'$. 
	\end{enumerate}
	Thus, in either case, we arrive at a contradiction, which completes the proof. \Halmos
\end{proof}

\begin{lemma}
	Suppose $\ell = f(C)$. Then: (a) $A_{s'} \notin \Ecal$ for all $s' \in \LS(\ell)$; and (b) $B_{s'} \notin \Ecal$ for all $s' \in \RS(\ell)$.\label{lemma:alpha_beta_at_C_is_irrelevant}
\end{lemma}
{\it Proof of Lemma~\ref{lemma:alpha_beta_at_C_is_irrelevant}:} We proceed by contradiction. Suppose that $A_s$ occurs for some $s \in \LS(\ell)$ or that $B_s$ occurs for some $s \in \LS(\ell)$; in the former case, let $e = A_s$, and in the latter case, let $e = B_s$. Let $\ell' = f(e)$. Then $\ell'$ must be checked before $\ell$ by Algorithm~\ref{algorithm:SplitMIO_primal_greedy}, since the algorithm always terminates after a $C$ event occurs. Consider what happens when Algorithm~\ref{algorithm:SplitMIO_primal_greedy} checks leaf $\ell$:
\begin{enumerate}
	\item \textbf{Case 1}: $q_C > 0$. This is impossible, because if $e$ occurs, then $q_s$ when leaf $\ell$ is checked would have to be 0, which would imply that $q_{A,B} < q_C$ and that a $C$ event could not have occurred when $\ell$ was checked.
	\item \textbf{Case 2}: $q_C = 0$. This is also impossible, because it implies that the unit sum constraint~\eqref{prob:SplitMIO_sub_primal_unitsum} was satisfied at an earlier iteration, which would have triggered the $C$ event at a leaf that was checked before $\ell$.
\end{enumerate}
We thus have that $A_{s'}$ does not occur for any $s' \in \LS(\ell)$ and $B_{s'}$ does not occur for any $s' \in \RS(\ell)$. \hfill \Halmos 

With these two lemmas in hand, we now return to the proof of Theorem~\ref{proof:SplitMIO_benders_dual_greedy_BFS} (b). 
Observe now that by using Lemmas~\ref{lemma:alpha_beta_below_e_is_irrelevant} and \ref{lemma:alpha_beta_at_C_is_irrelevant} and using equations~\eqref{eq:SplitMIO_dual_BFS_alpha_simple} and \eqref{eq:SplitMIO_dual_BFS_beta_simple}, the system of equations~\eqref{eq:SplitMIO_dual_BFS_r_A}-\eqref{eq:SplitMIO_dual_BFS_beta} is equivalent to
\begin{align}
	& \alpha_{t,s} + \sum_{ \substack{s' \in \LS(f(A_s): \,\, d(s') < d(s), \,\, A_{s'} \in \Ecal} } \alpha_{t,s'} + \sum_{ \substack{s' \in \RS(\ell): \,\, d(s') < d(s), \,\, B_{s'} \in \Ecal} } \beta_{t,s'} + \gamma_t = r_{t,f(A_s)}, \quad \forall s \ \text{such that} \ A_s \in \Ecal, \nonumber\\
	& \beta_{t,s} + \sum_{ \substack{s' \in \LS(\ell):\,\, d(s') < d(s), \,\, A_{s'} \in \Ecal} } \alpha_{t,s'} + \sum_{ \substack{s' \in \RS(\ell): \,\, d(s') < d(s), \,\, B_{s'} \in \Ecal} } \beta_{t,s'} + \gamma_t = r_{t,f(B_s)}, \quad \forall s \ \text{such that} \ B_s \in \Ecal, \nonumber \\
	& \gamma_t = r_{t,f(C)}, \nonumber \\
	& \alpha_{t,s} = 0, \quad \forall \ s \ \text{such that}\ A_{s} \notin \Ecal, \label{eq:SplitMIO_dual_BFS_alpha_simple} \\
	& \beta_{t,s} = 0, \quad \forall \ s \ \text{such that}\ B_{s} \notin \Ecal. \label{eq:SplitMIO_dual_BFS_beta_simple}
\end{align}
Thus the solution implied by this system of equations is exactly the solution produced by Algorithm~\ref{algorithm:SplitMIO_dual_greedy}, establishing that $(\alphab_t, \betab_t, \gamma_t)$ is a basic feasible solution of problem~\eqref{prob:SplitMIO_sub_dual}, and thus an extreme point. \Halmos

\subsection{Proof of Theorem~\ref{theorem:SplitMIO_benders_primal_dual_greedy_optimal}}
\label{proof:SplitMIO_benders_primal_dual_greedy_optimal}

To prove that the $\yb_t$ and $(\alphab_t, \betab_t, \gamma_t)$ produced by Algorithms~\ref{algorithm:SplitMIO_primal_greedy} and \ref{algorithm:SplitMIO_dual_greedy} are optimal for their respective problems, we show that they satisfy complementary slackness for problems~\eqref{prob:SplitMIO_sub_primal} and \eqref{prob:SplitMIO_sub_dual}:
\begin{align}
	\alpha_{t,s} \cdot \left( x_{v(t,s)} - \sum_{\ell \in \leftleaves(s)} y_{t,\ell} \right) & = 0, \quad \forall \ s \in \splits(t), \label{eq:SplitMIO_CS_alpha_left} \\
	\beta_{t,s} \cdot \left( 1 - x_{v(t,s)} - \sum_{\ell \in \rightleaves(s)} y_{t,\ell} \right) & = 0, \quad \forall \ s \in \splits(t), \label{eq:SplitMIO_CS_beta_right}  \\
	y_{t,\ell} \cdot \left( \sum_{s \in \LS(\ell)} \alpha_{t,s} + \sum_{s \in \RS(\ell)} \beta_{t,s} + \gamma_t - r_{t,\ell} \right) & = 0, \quad \forall \ \ell \in \leaves(t). \label{eq:SplitMIO_CS_y_r}
\end{align}

\textbf{Condition~\eqref{eq:SplitMIO_CS_alpha_left}}: If $\alpha_{t,s} = 0$, then the condition is trivially satisfied. If $\alpha_{t,s} > 0$, then this implies that $A_s \in \Ecal$. This means that the left split constraint~\eqref{prob:SplitMIO_sub_primal_left} at $s$ became tight after leaves $f(A_s)$ was checked, which implies that $\sum_{\ell \in \leftleaves(s)} y_{t,\ell} = x_{v(t,s)}$ or equivalently, that $ x_{v(t,s)} - \sum_{\ell \in \leftleaves(s)} y_{t,\ell} = 0$, which again implies that the condition is satisfied.

\textbf{Condition~\eqref{eq:SplitMIO_CS_beta_right}}: This follows along similar logic to condition~\eqref{eq:SplitMIO_CS_alpha_left}, only that we use the fact that $\beta_{t,s} > 0$ implies that a $B_s$ event occurred and that the right split constraint~\eqref{prob:SplitMIO_sub_primal_right} at $s$ became tight. 

\textbf{Condition~\eqref{eq:SplitMIO_CS_y_r}}: If $y_{t,\ell} = 0$, then the condition is trivially satisfied. If $y_{t,\ell} > 0$, then either $\ell = f(C)$, or $\ell = f(A_s)$ for some split $s \in \LS(\ell)$, or $\ell = f(B_s)$ for some split $s \in \RS(\ell)$. In any of these three cases, as shown in the proof of part (b) of Theorem~\ref{theorem:SplitMIO_benders_dual_greedy_BFS}, the dual constraint~\eqref{prob:SplitMIO_sub_dual_r} holds with equality for any such leaf $\ell$. Thus, we have that $ \sum_{s \in \LS(\ell)} \alpha_{t,s} + \sum_{s \in \RS(\ell)} \beta_{t,s} + \gamma_t - r_{t,\ell} = 0$, and the condition is again satisfied. 

As complementary slackness holds, $\yb_t$ is feasible for the primal problem~\eqref{prob:SplitMIO_sub_primal} (Theorem~\ref{proof:SplitMIO_benders_primal_greedy_BFS}), and $(\alphab_t, \betab_t, \gamma_t)$  is feasible for the dual problem~\eqref{prob:SplitMIO_sub_dual} (Theorem~\ref{theorem:SplitMIO_benders_dual_greedy_BFS}). %
Thus they are optimal for their respective problems. \Halmos

\subsection{Proof of Theorem~\ref{theorem:SplitMIO_benders_primal_dual_binary_closedform_optimal}}
\label{proof:SplitMIO_benders_primal_dual_binary_closedform_optimal}

\emph{Proof of part (a):} Observe that by construction, $\yb_t$ automatically satisfies the unit sum constraint~\eqref{prob:SplitMIO_sub_primal_unitsum} and the nonnegativity constraint~\eqref{prob:SplitMIO_sub_primal_nonnegative}. We thus need to verify constraints~\eqref{prob:SplitMIO_sub_primal_left} and \eqref{prob:SplitMIO_sub_primal_right}. For constraint~\eqref{prob:SplitMIO_sub_primal_left}, observe that for any split $s \notin \LS(\ell^*)$, it must be that $\ell^* \notin \leftleaves(s)$. Thus, we will have $\sum_{\ell \in \leftleaves(s)} y_{t,\ell} = 0$, 
which means that constraint~\eqref{prob:SplitMIO_sub_primal_left} is automatically satisfied, because the right hand side $x_{v(t,s)}$ is always at least 0. On the other hand, for any split $s \in \LS(\ell^*)$, we will have that $x_{v(t,s)} = 1$, and that $\sum_{\ell \in \leftleaves(s)} y_{t,\ell} = \sum_{\ell \in \leftleaves(s) : \ell \neq \ell^*} y_{t,\ell} + y_{t,\ell^*} = 1$, 
which implies that constraint~\eqref{prob:SplitMIO_sub_primal_left} is satisfied. Similar reasoning can be used to establish that constraint~\eqref{prob:SplitMIO_sub_primal_right} holds. This establishes that $\yb_t$ is indeed a feasible solution of problem~\eqref{prob:SplitMIO_sub_primal}.

\emph{Proof of part (b):} By construction, $\alpha_{t,s} \geq 0$ and $\beta_{t,s} \geq 0$ for all $s \in \splits(t)$, so constraints~\eqref{prob:SplitMIO_sub_dual_alphageqzero} and \eqref{prob:SplitMIO_sub_dual_betageqzero} are satisfied. To verify constraint~\eqref{prob:SplitMIO_sub_dual_r}, fix a leaf $\ell \in \leaves(t)$. If $\ell \neq \ell^*$, then either $\ell \in \leftleaves(s')$ for some $s' \in \RS(\ell^*)$ or $\ell \in \rightleaves(s')$ for some $s' \in \LS(\ell^*)$. If $\ell \in \leftleaves(s')$ for some $s' \in \RS(\ell^*)$, then 
\begin{align*}
	\sum_{s: \ell \in \leftleaves(s)} \alpha_{t,s} + \sum_{s : \ell \in \rightleaves(s)} \beta_{t,s} + \gamma_t  \geq \alpha_{t,s'} + \gamma_t  \geq \max_{\ell' \in \leftleaves(s')} r_{t,\ell'} - r_{t,\ell^*} + r_{t,\ell^*}  \geq r_{t,\ell}
\end{align*}
where the first inequality follows because $\ell \in \leftleaves(s')$ and the fact that all $\alpha_{t,s}$ and $\beta_{t,s}$ variables are nonnegative; the second follows by how the dual solution is defined in the statement of the theorem; and the third by the definition of the maximum. 
Similarly, if $\ell \in \rightleaves(s')$ for some $s' \in \LS(\ell^*)$, then we have
\begin{align*}
	\sum_{s: \ell \in \leftleaves(s)} \alpha_{t,s} + \sum_{s : \ell \in \rightleaves(s)} \beta_{t,s} + \gamma_t  \geq \beta_{t,s'} + \gamma_t  \geq \max_{\ell' \in \rightleaves(s')} r_{t,\ell'} - r_{t,\ell^*} + r_{t,\ell^*}  \geq r_{\ell} - r_{t,\ell^*} + r_{t,\ell^*} = r_{t,\ell}.
\end{align*}
Lastly, if $\ell = \ell^*$, then we automatically have $\sum_{s : \ell^* \in \leftleaves(s)} \alpha_{t,s} + \sum_{s: \ell^* \in \rightleaves(s)} \beta_{t,s} + \gamma_t \geq \gamma_t  = r_{t,\ell^*}$. Thus, we have established that constraint~\eqref{prob:SplitMIO_sub_dual_r} is satisfied for all leaves $\ell$, and thus $(\alphab_t, \betab_t, \gamma_t)$ as defined in the statement of the theorem is a feasible solution of the dual~\eqref{prob:SplitMIO_sub_dual}. 

\emph{Proof of part (c):} To establish that the two solutions are optimal, by weak duality it is sufficient to show that the two solutions attain the same objective values in their respective problems. For the primal solution $\yb_t$, it is immediately clear that its objective is $r_{t,\ell^*}$. For the dual solution $(\alphab_t, \betab_t, \gamma_t)$, we have
\begin{align*}
	\sum_{s \in \splits(t)} \alpha_{t,s} x_{v(t,s)} + \sum_{s \in \splits(t)} \beta_{t,s} (1 - x_{v(t,s)}) + \gamma_t  = \sum_{s \in \RS(\ell^*)} \alpha_{t,s} x_{v(t,s)} + \sum_{s \in \LS(\ell^*)} \beta_{t,s} (1 - x_{v(t,s)}) + \gamma_t =  \gamma_t  = r_{t,\ell^*},
\end{align*}
where the first step follows because $\alpha_{t,s} = 0$ for $s \notin \RS(\ell^*)$ and $\beta_{t,s} = 0$ for $s \notin \LS(\ell^*)$; the second step follows by the fact that $x_{v(t,s)} = 0$ for $s \in \RS(\ell^*)$ and $x_{v(t,s)} = 1$ for $s \in \LS(\ell^*)$; and the final step follows just by the definition of $\gamma_t$. This establishes that $\yb_t$ and $(\alphab_t, \betab_t, \gamma_t)$ are optimal for their respective problems. \Halmos

\subsection{Proof of Theorem~\ref{theorem:ProductMIO_benders_primal_dual_binary_closedform_optimal}}
\label{proof:ProductMIO_benders_primal_dual_binary_closedform_optimal}

The proof follows a similar argument to the proof of Theorem~\ref{theorem:SplitMIO_benders_primal_dual_binary_closedform_optimal}. For brevity, we omit the proof.

\section{Additional Analysis of the \SplitMIO and \ProductMIO Formulations}
\label{sec:one_tree_analysis}

\subsection{\ProductMIO and the Independent Branching Schemes \cite{huchette2019combinatorial}}
\label{subsec:appendix-CDC}

As noted in Section~\ref{subsec:model_ProductMIO}, the \ProductMIO formulation is closely connected to the literature on independent branching schemes \citeapp{huchette2019combinatorial}. In this section, we elaborate on this connection. Assume one has a finite ground set $J$, and is interested in optimizing over a particular subset of the $(|J|-1)-$dimensional unit simplex over $J$, $\Delta^J = \{ \lambdab \in \Rbb^J \mid \sum_{j \in J} \lambda_j = 1; \lambdab \geq \zerob\}$. The specific subset that we are interested in is called a \emph{combinatorial disjunctive constraint} (CDC), and is given by $\CDC(\Scal) = \bigcup_{S \in \Scal} Q(S)$, 
where $\Scal$ is a finite collection of subsets of $J$ and $Q(S) = \{ \lambdab \in \Delta \mid \lambda_j \leq 0 \ \text{for} \ j \in J \setminus S\}$ for any $S \subseteq J$. This approach is very general: for example, by associating each $j$ with a point $\xb^j$ in $\Rbb^n$, one can use $\CDC(\Scal)$ to model an optimization problem over a union of polyhedra, where each polyhedron is the convex hull of a collection of vertices in $S \in \Scal$. 

A \emph{$k$-way independent branching scheme of depth $t$} is a representation of $\CDC(\Scal)$ as a sequence of $t$ choices between $k$ alternatives. Specifically, we have $\CDC(\Scal) = \bigcap_{m=1}^t \bigcup_{i = 1}^k Q({L^i_m})$, 
where ${L^i_m} \subseteq J$. 
In the special case that $k = 2$, we can write $\CDC(\Scal) = \cap_{m=1}^t ( L_{{m}} \cup R_{{m}})$ where $L_{{m}}, R_{{m}} \subseteq J$. This representation is known as a \emph{pairwise independent branching scheme} and the constraints of the corresponding MIO can be written simply as 
\begin{subequations}
	\label{prob:CDC_pairwise_IBS}
	\begin{alignat}{2}
		& & & \sum_{j \in L_{{m}}} \lambda_j \leq z_m, \quad \forall\ m \in \{1,\dots, t\}, \label{prob:CDC_pairwise_IBS_left}\\
		& & & \sum_{j \in R_{{m}}} \lambda_j \leq 1 - z_m, \quad \forall\ m \in \{1,\dots, t\}, \label{prob:CDC_pairwise_IBS_right}\\
		& & & z_m \in \{0,1\}, \quad \forall \ m \in \{1,\dots, t\}, \label{prob:CDC_pairwise_IBS_binary} \\
		& & & \sum_{j \in J} \lambda_j = 1, \\
		& & & \lambda_j \geq 0, \quad \forall \ j \in J.
	\end{alignat}%
\end{subequations}
This particular special case is important because it is always integral (see Theorem~1 of \citeapp{vielma2010mixed}). Moreover, we can see that \ProductMIO bears a strong resemblance to formulation~\eqref{prob:CDC_pairwise_IBS}. Constraints~\eqref{prob:CDC_pairwise_IBS_left} and \eqref{prob:CDC_pairwise_IBS_right} correspond to constraints~\eqref{prob:ProductMIO_leftsplitconstraint} and \eqref{prob:ProductMIO_rightsplitconstraint}, respectively. In terms of variables, the $\lambda_j$ and $z_m$ variables in formulation~\eqref{prob:CDC_pairwise_IBS} correspond to the $y_{t,\ell}$ and $x_i$ variables in \ProductMIO, respectively.

One notable difference is that in practice, one would use formulation~\eqref{prob:CDC_pairwise_IBS} in a modular way; specifically, one would be faced with a problem where the feasible region can be written as $\CDC(\Scal_1) \cap \CDC(\Scal_2) \cap \dots \cap \CDC(\Scal_G)$, where each $\Scal_g$ is a collection of subsets of $J$. To model this feasible region, one would introduce a set of $z_{g,m}$ variables for the $g$th CDC, enforce constraints~\eqref{prob:CDC_pairwise_IBS_left} - \eqref{prob:CDC_pairwise_IBS_binary} for the $g$th CDC, and use only one set of $\lambda_j$ variables for the whole formulation. Thus, the $\lambda_j$ variables are the ``global'' variables, while the $z_{g,m}$ variables would be ``local'' and specific to each CDC. In contrast, in \ProductMIO, the $x_i$ variables (the analogues of $z_m$) are the ``global'' variables, while the $y_{t,\ell}$ variables (the analogues of $\lambda_j$) are the ``local'' variables.

\subsection{Comparison of \SplitMIO and \ProductMIO when $|F| = 1$}
\label{subsec:appendix-single-tree}

In this section, we analyze the \SplitMIO and \ProductMIO formulations in the special case where $|F| = 1$, i.e., the decision forest model consists of a single tree. To simplify notation, we drop the index $t$. Let$\Fcal_{\SplitMIO}$ and $\Fcal_{\ProductMIO}$ denote the feasible regions of \SplitMIO and \ProductMIO, respectively. We begin by characterizing the condition under which $\Fcal_{\SplitMIO}$ is integral.
\begin{proposition}
	Let $(F, \lambdab)$ be a decision forest model consisting of a single tree, i.e., $|F| = 1$. Then $\Fcal_{\SplitMIO}$ is integral if and only if $v(s) = i$ for at most one $s \in \splits$ for every product $i \in \Ncal$.
	\label{proposition:SplitMIO_Keq1_integral}
\end{proposition}
The proof (see Appendix~\ref{appendix:proof_SplitMIO_Keq1_integral}) follows by showing that the constraint matrix defining $\Fcal_{\SplitMIO}$ is totally unimodular.

Proposition~\ref{proposition:SplitMIO_Keq1_integral} establishes that when $|F| = 1$, the feasible region $\Fcal_{\SplitMIO}$ is integral \emph{if and only if} each product appears in at most one split of the decision tree. By contrast, the following result shows that \ProductMIO enjoys integrality unconditionally in this $|F| = 1$ setting.
\begin{proposition}
	\label{proposition:ProductMIO_Keq1_integral}
	For any decision forest model $(F, \lambdab)$ with $|F| = 1$, $\Fcal_{\ProductMIO}$ is integral.
\end{proposition}
Unlike Proposition~\ref{proposition:SplitMIO_Keq1_integral}, the proof of Proposition~\ref{proposition:ProductMIO_Keq1_integral} (see Appendix~\ref{appendix:proof_ProductMIO_Keq1_integral}) does not rely on total unimodularity, since the constraint matrix of \ProductMIO is generally not totally unimodular (see the example in Section~\ref{subsec:Non_TU_ProductMIO}). Instead, we exploit a connection between \ProductMIO and the independent branching scheme introduced in Section~\ref{subsec:model_ProductMIO}. The integrality property then follows from Theorem 1 of \citeapp{vielma2010mixed}.

Taken together, Propositions~\ref{proposition:SplitMIO_Keq1_integral} and \ref{proposition:ProductMIO_Keq1_integral} highlight a sharp distinction: for \SplitMIO, integrality depends critically on whether products are reused across splits, while for \ProductMIO, integrality holds universally in the single-tree case. This difference underscores the formulation strength of \ProductMIO and provides a general example in which the inclusion in Proposition~\ref{proposition:ProductMIO_stronger_than_SplitMIO} is strict.

Comparing the two propositions also offers insights into the computational behavior of the two formulations at full scale (see Sections~\ref{subsec:synthetic_T_integrality_gap} and \ref{subsec:synthetic_T_optimality_gap_T1_T2}). While the integrality properties in Propositions~\ref{proposition:SplitMIO_Keq1_integral} and \ref{proposition:ProductMIO_Keq1_integral} do not necessarily extend to multi-tree settings, formulations that exhibit strong structural properties in simplified cases often scale more effectively in practice \citepapp{vielma2015mixed,anderson2020strong}. On the other hand, as discussed in Section~\ref{subsec:SplitMIO_vs_ProductMIO}, \SplitMIO has the important advantage of enabling an efficient Benders decomposition approach for large-scale problems. In general, a compact formulation such as \ProductMIO may not yield the structure required for decomposition-based methods, so each formulation has distinct contexts where it is preferable.

\subsection{Proof of Proposition~\ref{proposition:SplitMIO_Keq1_integral}}
\label{appendix:proof_SplitMIO_Keq1_integral}

\emph{Proof of the sufficient condition:} \YCRR{When $|F| = 1$, we can write the feasible region $\Fcal_{\SplitMIO}$ of the relaxation of \SplitMIO as follows:
	\begin{align}
		\sum_{\ell \in \leftleaves(s)} y_{\ell} - x_{v(s)}  & \leq 0, \quad \forall\ s \in \splits, \label{eq:SplitMIO_TU_firstconstraint} \\
		\sum_{\ell \in \rightleaves(s)} y_{\ell} + x_{v(s)}  & \leq 1, \quad \forall\ s \in \splits, \\
		\sum_{\ell \in \leaves} y_{\ell} & = 1, \quad \forall\ \ell \in \leaves, \label{eq:SplitMIO_TU_unitsum}\\
		x_i & \leq 1, \quad \forall\ i \in [n]. \label{eq:SplitMIO_TU_x_leq_1} \\
		y_{\ell} & \geq 0, \quad \forall\  \ell \in \leaves,  \\
		x_i & \geq 0, \quad \forall\ i \in [n], \label{eq:SplitMIO_TU_lastconstraint}
	\end{align}
	When we concatenate all of the decision variables in a vector $(\yb, \xb)$, the constraint coefficient matrix can be in block form:
	\begin{equation}
		\Ab = \left[ \begin{array}{cc} 
			\Bb & \Cb \\
			\zerob & \Ib\\
			\Ib & \zerob \\
			\zerob & \Ib
		\end{array} \right],
	\end{equation}
	where $\Ib$ is used to denote an appropriately-sized identity matrix and $\zerob$ is an appropriately sized matrix with all entries equal to zero. The first row in the block form corresponds to constraints \eqref{eq:SplitMIO_TU_firstconstraint} - \eqref{eq:SplitMIO_TU_unitsum}, and the second row in the block form corresponds to \eqref{eq:SplitMIO_TU_x_leq_1}. The last two rows in the block form correspond to the element-wise non-negativity of $\yb$ and $\xb$, respectively.} We assume that the columns in $\Bb$ are ordered so as to follow the order of the leaves in the tree topology, i.e., the left-most leaf in the tree topology appears as the first column, followed by the second left-most leaf, and so on, until the right-most leaf which would be the final column in $\Bb$. Note that appending columns or rows from the identity matrix to a matrix preserves the total unimodularity of that matrix, the above matrix $\Ab$ is totally unimodular if and only if the matrix $[\Bb \ \Cb]$ is totally unimodular. Thus, we focus on establishing that $[\Bb \ \Cb]$ is totally unimodular.

Consider now a square submatrix $\Ab'$ of $[\Bb \ \Cb]$. To show that $[\Bb \ \Cb]$ is totally unimodular, we need to verify that $\det \Ab' \in \{+1,-1,0\}$. There are two cases to consider: \\

\noindent \emph{Case 1}. If $\Ab'$ is a submatrix of only $\Bb$, i.e., $\Ab' = [\Bb']$ for some square submatrix $\Bb'$ of $\Bb$, then the determinant of $\Ab'$ must be +1, -1, or 0. This is because the matrix $\Bb$ satisfies the consecutive ones property -- each row is of the form $[0,\dots,0,1,\dots,1,0,\dots,0]$ -- which is guaranteed because the columns of $\Bb$ are assumed to follow the order of the leaves in the tree topology. Since $\Bb$ satisfies the consecutive ones property and is an interval matrix, it is totally unimodular, and any square submatrix has determinant +1, -1, or 0 (\citeapp{fulkerson1965incidence}; see also Theorem 3.3(c) in \citeapp{bertsimas2005optimization}).\\

\noindent \emph{Case 2}. If $\Ab'$ a submatrix of both $\Bb$ and $\Cb$, i.e., $\Ab' = [\Bb' \ \Cb']$ for appropriate submatrices $\Bb'$ of $\Bb$ and $\Cb'$ of $\Cb$, write $\Cb'$ as $\Cb' = [\Cb'_{1} \ \dots \Cb'_{k} ]$, where each $\Cb'_{j}$ is a column of the submatrix $\Cb'$. For each column $\Cb'_j$, there are four possibilities: either the column contains all zeros; the column contains one +1, and all other entries are zero; the column contains one -1, and all other entries are zero; or the column contains one +1, one -1, and all other entries are zero. Note that it is not possible for a +1 to appear more than once, or a -1 to appear more than once, because of the hypothesis that each product appears at most once in the splits of the tree. 

For each column that falls into the first three categories, do nothing. For each column that falls into the fourth category, perform an elementary row operation as follows: take the row in which the +1 occurs, and add it to the row in which the -1 occurs. These two rows have the form
\begin{align}
	& [\fb_{\leftleaves(s)} \ \ \; 0 \ \dots \ 0\ -1\ 0 \ \dots 0 ] \qquad \text{(row with -1)} \\
	& [\fb_{\rightleaves(s)} \ 0 \ \dots \ 0\ +1\ 0 \ \dots 0 ] \qquad \text{(row with +1)}
\end{align}
for some split $s$, where $\fb_{\Ucal}$ is the incidence vector of the leaf set $\Ucal$, i.e., a 0-1 vector where a 1 occurs for each column in the submatrix $\Bb$ if the corresponding leaf is in the set $\Ucal$, and a 0 occurs otherwise. Observe that when we add them, the +1 and -1 cancel out, and the row in which the -1 occurs becomes
\begin{align}
	[\fb_{\leftleaves(s)} \! +\fb_{\rightleaves(s)} \ 0 \ \dots\ 0 ]
\end{align}
which is equivalent to
\begin{align}
	[\  \fb_{\leftleaves(s) \cup \rightleaves(s)} \ \  0 \ \dots \ 0 ],
\end{align}
because $\leftleaves(s)$ and $\rightleaves(s)$ do not intersect. Note also that for any split $s$, the vector $\fb_{\leftleaves(s) \cup \rightleaves(s)}$ satisfies the consecutive ones property, again by the assumption that the columns of $\Bb$ are ordered in the same way as the leaves in the tree topology.

Thus, by performing this operation -- adding the row with a -1 to the row with the +1 -- for each column $\Cb'_j$ with one +1 and one -1, we obtain a new matrix, $[\tilde{\Bb} \ \tilde{\Cb}]$, where each column of $\tilde{\Cb}$ is either a vector of zeros, a vector of zeros with one +1 (i.e., a column of an appropriately sized identity matrix), or a vector of zeros with one -1 (i.e., the negative of a column of an appropriately sized identity matrix). Additionally, each row of $\tilde{\Bb}$ still satisfies the consecutive ones property. This follows because the modified row satisfies the consecutive ones property, as explained in the previous paragraph, and all other unmodified rows also satisfy it, as discussed for Case 1. Note also that because each product appears at most once in each split, any row that is modified by the prescribed elementary row operation is only modified once, and cannot be modified multiple times. Thus, the square matrix $[\tilde{\Bb} \ \tilde{\Cb}]$ is totally unimodular, and its determinant therefore must be +1, -1 or 0. Since this matrix was obtained by elementary row operations from the submatrix $[\Bb' \ \Cb']$, it follows that this original submatrix has the same determinant as $[\tilde{\Bb} \ \tilde{\Cb}]$, and \YCRR{therefore its determinant is +1, -1, or 0.} \\

\noindent Thus, since $\det \Ab' \in \{+1,-1,0\}$, it follows that $[\Bb \ \Cb]$ is totally unimodular, and that the overall constraint matrix $\Ab$ is totally unimodular. By standard integer programming results (\citeapp{hoffman1956integral}; see also Theorem 3.1(a) in \citeapp{bertsimas2005optimization}), it follows that all extreme points of the polyhedron described by \eqref{eq:SplitMIO_TU_firstconstraint} - \eqref{eq:SplitMIO_TU_lastconstraint} are integral. Lastly, since the polyhedron \eqref{eq:SplitMIO_TU_firstconstraint} - \eqref{eq:SplitMIO_TU_lastconstraint} is isomorphic to the polyhedron described by constraints~\eqref{prob:SplitMIO_y_unitsum} - \eqref{prob:SplitMIO_y_nonnegative} when $|F| = 1$, each extreme point of \eqref{eq:SplitMIO_TU_firstconstraint} - \eqref{eq:SplitMIO_TU_lastconstraint} is an extreme point of $\Fcal_{\SplitMIO}$ (see Exercise 2.5(b) in \citeapp{bertsimas1997introduction}), and thus each extreme point of $\Fcal_{\SplitMIO}$ is integral. \\

\noindent\emph{Proof of the necessary condition:} To show this direction, we will prove the contrapositive. Thus, suppose that there exists a product $i^*$ such that for two different splits, $s_1, s_2$, it is the case that $v(s_1) = i^*$ and $v(s_2) = i^*$ (note that for convenience, we drop $t$ whenever it appears). We now need to show the existence of a non-integral extreme point of $\Fcal_{\SplitMIO}$. 

Let $\ell_1 \in \leftleaves(s_1)$ be any leaf left of $s_1$, and $\ell_2 \in \leftleaves(s_2)$ be any leaf left of $s_2$. Let $\nu^1_1, \dots, \nu^1_{k_1}$ be the set of split nodes traversed on the path from the root node to leaf $\ell_1$, where $\nu^1_1$ denotes the root node. Similarly, let $\nu^2_1, \dots, \nu^2_{k_2}$ be the set of split nodes traversed on the path from the root node to leaf $\ell_2$, where again $\nu^2_1$ denotes the root node. Let $k^*$ be defined as
\begin{equation}
	k^* = \max\{ 1 \leq k \leq \min\{ k_1, k_2\} \mid \nu^1_k = \nu^2_k\}.
\end{equation}
Observe that $k^*$ denotes the position of the node where the path from the root node to $\ell_1$ and the path from the root node to $\ell_2$ diverge. Let $s^* = \nu^1_{k^*} = \nu^2_{k^*}$ denote this split node. Without loss of generality, let us assume that $\ell_1 \in \leftleaves(s^*)$ and $\ell_2 \in \rightleaves(s^*)$. See an illustration in Figure~\ref{fig:TU_proof}.

\begin{figure}[h!]
	\centering
	\includegraphics[width=0.15\textwidth]{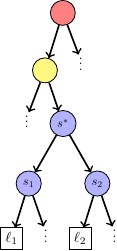}
	\caption{An illustration supporting the proof of the necessary condition in Proposition~\ref{proposition:SplitMIO_Keq1_integral}. Nodes are labeled according to the definitions in the proof. Unspecified parts of the tree are represented by vertical dots. Split nodes are colored blue, red, and yellow, corresponding to $x_i$ values of 0.5, 1, and 0, respectively, as defined in equations~\eqref{eq:TU_proof_xi_1} - \eqref{eq:TU_proof_xi_05}. Additionally, $y_{\ell_1} = y_{\ell_2} = 0.5$, while the $y_\ell$ values for all other leaf nodes are 0. \label{fig:TU_proof}}
\end{figure}

Next, we will define subsets of split nodes. We define the sets $\Scal^{0,\leftleaves}, \Scal^{0,\rightleaves}, \Scal^{1,\leftleaves}, \Scal^{1,\rightleaves}, \Scal^{2,\leftleaves}, \Scal^{2,\rightleaves}$ as
\begin{align}
	\Scal^{0,\leftleaves} & = \LS(\ell_1) \cap \{\nu^1_1, \dots, \nu^1_{k^*-1} \},  \\
	\Scal^{0,\rightleaves} & = \RS(\ell_1) \cap \{\nu^1_1, \dots, \nu^1_{k^*-1} \}, \\
	\Scal^{1,\leftleaves} & = \LS(\ell_1) \cap \{\nu^1_{k^*}, \nu^1_{k^*+1}, \dots, \nu^1_{k_1} \}, \\
	\Scal^{1,\rightleaves} & = \RS(\ell_1) \cap \{\nu^1_{k^*}, \nu^1_{k^*+1}, \dots, \nu^1_{k_1} \}, \\
	\Scal^{2,\leftleaves} & = \LS(\ell_2) \cap \{\nu^2_{k^*}, \nu^2_{k^*+1}, \dots, \nu^2_{k_2} \}, \\
	\Scal^{2,\rightleaves} & = \RS(\ell_2) \cap \{\nu^2_{k^*}, \nu^2_{k^*+1}, \dots, \nu^2_{k_2} \}.
\end{align}
In words, $\Scal^{0,\leftleaves}$ and $\Scal^{0,\rightleaves}$ are the sets of splits that we follow to the left and to the right, respectively, to reach $s^*$. From $s^*$, $\Scal^{1,\leftleaves}$ and $\Scal^{1,\rightleaves}$ are the sets of splits that we follow to the left and to the right, respectively, to reach $\ell_1$. Finally, from $s^*$, $\Scal^{2,\leftleaves}$ and $\Scal^{2,\rightleaves}$ are the sets of splits that we follow to the left and to the right, respectively, to reach $\ell_2$. 

Next, we define sets of product indices. We let $\Ical = \{ i \in \Ncal \mid v(s) = i \ \text{for some}\ s \in \splits\}$ be the overall set of products that appear in the splits of the tree. We define additional sets as follows:
\begin{align*}
	& \Ical^{0, \leftleaves} = \{ i \in \Ncal \mid v(s) = i \ \text{for some} \ s \in \Scal^{0, \leftleaves} \} \\
	& \Ical^{0, \rightleaves} = \{ i \in \Ncal \mid v(s) = i \ \text{for some} \ s \in \Scal^{0, \rightleaves} \} \\
	& \Ical^{1, \leftleaves}  = \{ i \in \Ncal \mid v(s) = i \ \text{for some} \ s \in \Scal^{1, \leftleaves} \} \\
	& \Ical^{1, \rightleaves} = \{ i \in \Ncal \mid v(s) = i \ \text{for some} \ s \in \Scal^{1, \rightleaves} \} \\
	& \Ical^{2, \leftleaves} = \{ i \in \Ncal \mid v(s) = i \ \text{for some} \ s \in \Scal^{2, \leftleaves} \} \\
	& \Ical^{2, \rightleaves}  = \{ i \in \Ncal \mid v(s) = i \ \text{for some} \ s \in \Scal^{2, \rightleaves} \}
\end{align*}

We now construct a fractional solution, as follows:
\begin{align}
	y_{\ell} & = 0, \quad \forall \ell \in \leaves \setminus \{\ell_1,\ell_2\}, \\
	y_{\ell_1} & = 0.5, \\
	y_{\ell_2} & = 0.5, \\
	x_{i} & = 1, \quad \forall i \in \Ical^{0,\leftleaves}, \label{eq:TU_proof_xi_1} \\
	x_{i} & = 0, \quad \forall i \in \Ical^{0,\rightleaves}, \label{eq:TU_proof_xi_0}\\
	x_{i} & = 0.5, \quad \forall i \in \Ical^{1,\leftleaves} \cup \Ical^{2,\leftleaves} \cup \Ical^{1,\rightleaves} \cup \Ical^{2,\rightleaves}, \label{eq:TU_proof_xi_05} \\
	x_i & = 0, \quad \forall i \in \Ncal \setminus (\Ical^{0,\leftleaves} \cup \Ical^{0,\rightleaves} \Ical^{1,\leftleaves} \cup \Ical^{2,\leftleaves} \cup \Ical^{1,\rightleaves} \cup \Ical^{2,\rightleaves})
\end{align}
This construction of the solution $(\xb,\yb)$ is well-defined because $ \left( \Ical^{1,\leftleaves} \cup \Ical^{2,\leftleaves} \cup \Ical^{1,\rightleaves} \cup \Ical^{2,\rightleaves} \right)$, $\Ical^{0,\leftleaves}$, and $\Ical^{1,\rightleaves}$ are mutually disjoint by their definitions and in accordance with Assumption~\ref{assumption:Requirement_3}. We illustrate this constructed solution in Figure~\ref{fig:TU_proof}.

It is also not difficult to verify that this solution is a feasible solution of $\Fcal_{\SplitMIO}$. Additionally, one can verify that this is the unique solution of the following system of active constraints: 
\begin{subequations}
	\begin{alignat}{2}
		& & & y_{\ell} = 0, \quad \forall \ \ell \in \leaves \setminus \{\ell_1, \ell_2\} \\
		& & & \sum_{\ell \in \leftleaves(s_1)} y_{\ell} = x_{v(s_1)}, \\
		& & & \sum_{\ell \in \leftleaves(s_2)} y_{\ell} = x_{v(s_2)}, \\
		& & & \sum_{\ell \in \leaves} y_{\ell} = 1, \\
		& & & \sum_{\ell \in \leftleaves(s)} y_{\ell} = x_{v(s)}, \quad \forall s \in \Scal^{0,\leftleaves}, \\
		& & & \sum_{\ell \in \rightleaves(s)} y_{\ell} = 1 - x_{v(s)}, \quad \forall s \in \Scal^{0,\rightleaves}  \\
		& & & \sum_{\ell \in \leftleaves(s)} y_{\ell} = x_{v(s)}, \quad \forall s \in \Scal^{1,\leftleaves} \cup \Scal^{2,\leftleaves}, \\
		& & & \sum_{\ell \in \rightleaves(s)} y_{\ell} = 1 - x_{v(s)}, \quad \forall s \in \Scal^{1,\rightleaves} \cup \Scal^{2,\rightleaves}, \\
		& & & x_i = 0, \quad \forall i \in \Ncal \setminus (\Ical^{0,\leftleaves} \cup \Ical^{0,\rightleaves} \Ical^{1,\leftleaves} \cup \Ical^{2,\leftleaves} \cup \Ical^{1,\rightleaves} \cup \Ical^{2,\rightleaves})
	\end{alignat}
\end{subequations}
Since the proposed $(\xb,\yb)$ is a feasible solution of $\Fcal_{\SplitMIO}$ and is uniquely determined by the above set of active constraints, it follows that there are $n + |\leaves|$ linearly independent active constraints. This means that $(\xb,\yb)$ is a basic feasible solution (and hence an extreme point) of $\Fcal_{\SplitMIO}$, which implies that $\Fcal_{\SplitMIO}$ is not integral. \Halmos

\subsection{Example of a Non-TU \ProductMIO Constraint Matrix}
\label{subsec:Non_TU_ProductMIO}

We present an instance of the \ProductMIO formulation in which the constraint matrix is not totally unimodular. Consider a decision forest model consisting of a single tree, denoted by $t$. This tree has a depth of three and contains seven split nodes. The root split corresponds to product 1, the two splits at depth two correspond to product 2, and the four splits at depth three correspond to product 3. These split nodes follow the structure depicted in the left panel of Figure~\ref{figure:example_T1_T2_T3}. Note that the products associated with the leaf nodes do not need to be specified, as $r_{t,\ell}$ appears only in the objective function of \ProductMIO and does not influence the constraint matrix. 

We now explicitly construct the constraint matrix. Since there is only one decision tree, we omit the subscript $t$. Ordering the variables as \YCRR{$(y_1,y_2,y_3,y_4,y_5,y_6,y_7,y_8,x_1,x_2,x_3)$}, the constraint matrix is given by
\setcounter{MaxMatrixCols}{20}
\begin{align}
	\label{eq:ProductMIO_matrix_nonTU}
	\begin{bmatrix}
		1 & 1 & 1 & 1 & 1 & 1 & 1 & 1 &      &       & \\
		1 & 1 & 1 & 1 &    &    &    &    & -1 &       &  \\
		1 & 1 &    &    & 1 & 1 &    &    &      & - 1 &\\
		1 &    & 1 &    & 1 &    & 1 &    &      &       &  -1  \\
		&    &    &    & 1 & 1 & 1 & 1 & 1  &  &  & \\
		&    & 1 & 1 &    &    & 1 & 1 &     & 1&  &  \\
		& 1 &  & 1 &  & 1 &  & 1  & &        & 1 
	\end{bmatrix},
\end{align}
where we omit the zeros. To demonstrate that this matrix is not totally unimodular, consider the submatrix formed by the second, third, and fourth rows and the columns corresponding to variables $y_2$, $y_3$, and $y_5$. This submatrix has the determinant
\begin{align*}
	\YCRR{\text{det} \left(  \begin{matrix}
			1 & 1 & 0\\
			1 & 0 & 1\\
			0 & 1 & 1
		\end{matrix}  \right)}
	= -2.
\end{align*}
Recall that a matrix is totally unimodular matrix if every square submatrix has a determinant of 0, 1 or -1. Therefore, the constraint matrix~\eqref{eq:ProductMIO_matrix_nonTU} is not totally unimodular.

\subsection{Proof of Proposition~\ref{proposition:ProductMIO_Keq1_integral}}

\label{appendix:proof_ProductMIO_Keq1_integral}

Define $\Delta^{\leaves} = \{ \yb \in \Rbb^{|\leaves|} \mid \sum_{\ell \in \leaves} y_{\ell} = 1; y_{\ell} \geq 0, \forall \ \ell \in \leaves \}$ to be the $(|\leaves|-1)$-dimensional unit simplex. In addition, for any $S \subseteq \leaves$, define $Q(S) = \{ \yb \in \Delta^{\leaves} \mid y_{\ell} \leq 0 \ \text{for} \ \ell \in \leaves \setminus S \}$. We write the combinatorial disjunctive constraint over the ground set $\leaves$ as $
\CDC(\leaves) = \bigcup_{\ell \in \leaves} Q(\{\ell\})$. Consider now the optimization problem
\begin{align}
	\label{prob:CDC_abstract}
	\underset{\yb}{\text{maximize}} \left\lbrace   \sum_{\ell \in \leaves} r_{\ell} y_{\ell} \,\, \bigg| \,\, \yb \in \CDC(\leaves)  \right\rbrace.
\end{align}
We will re-formulate this problem into a mixed-integer optimization problem. To do this, we claim that $\CDC(\leaves)$ can be written as the following pairwise independent branching scheme: 
\begin{equation}
	\bigcup_{\ell \in \leaves} Q(\{\ell\}) = \bigcup_{i=1}^n \left( Q(L_i) \cup Q(R_i) \right), \label{eq:CDC_IBS}
\end{equation}
where $L_i = \{ \ell \in \leaves \mid \ell \in \leftleaves(s)\ \text{for some}\ s \ \text{with} \ v(s) = i \}$ and $R_i = \{ \ell \in \leaves \mid \ell \in \rightleaves(s)\ \text{for some}\ s \ \text{with} \ v(s) = i \}$. Note that \eqref{eq:CDC_IBS} is equivalent to the statement
\begin{equation}
	\bigcup_{\ell \in \leaves} \{ \ell \} = \bigcup_{i=1}^n \left( (\leaves \setminus L_i) \cup (\leaves \setminus R_i) \right). \label{eq:CDC_IBS_leaves}
\end{equation}

To establish \eqref{eq:CDC_IBS_leaves}, it is sufficient to prove the following equivalence:
\begin{equation}
	\{ \ell \} = \bigcap_{i \in I(\ell)} (\leaves \setminus R_i) \cap \bigcap_{i \in E(\ell)} (\leaves \setminus L_i) \cap \bigcap_{ \substack{i \in \Ncal: \\ i \notin I(\ell) \cup E(\ell) }} (\leaves \setminus L_i), \label{eq:CDC_IBS_single_leaf}
\end{equation}
where $I(\ell) = \{ i \in \Ncal \mid \ell \in \bigcup_{s: v(s) = i} \leftleaves(s) \}$ and $E(\ell) = \{ i \in \Ncal \mid \ell \in \bigcup_{s: v(s) = i} \rightleaves(s) \} $.

We now prove \eqref{eq:CDC_IBS_single_leaf}.\\

\noindent \emph{Equation~\eqref{eq:CDC_IBS_single_leaf}, $\subseteq$ direction:} For $i \in I(\ell)$, we have that $\ell \in \bigcup_{s : v(s) = i} \leftleaves(s)$. This means that there exists $\bar{s}$ such that $\ell \in \leftleaves(\bar{s})$ and $v(\bar{s}) = i$. Since $\ell \in \leftleaves(\bar{s})$, this means that $\ell \notin \rightleaves(\bar{s})$ (a leaf cannot be to the left and to the right of any split). Moreover, $\ell$ cannot be in $\rightleaves(s)$ for any other $s$ with $v(s) = i$, because this would mean that product $i$ appears more than once along the path to leaf $\ell$, violating Assumption~\ref{assumption:Requirement_3}. Therefore, $\ell \in \leaves \setminus R_i$ for any $i \in I(\ell)$. 

For $i \in E(\ell)$, we have that $\ell \in \bigcup_{s: v(s) = i} \rightleaves(s)$. This means that there exists a split $\bar{s}$ such that $\ell \in \rightleaves(\bar{s})$ and $v(\bar{s}) = i$. Since $\ell \in \rightleaves(\bar{s})$, we have that $\ell \notin \leftleaves(\bar{s})$. In addition, $\ell$ cannot be in $\leftleaves(s)$ for any other $s$ with $v(s) = i$. Therefore, $\ell \in \leaves \setminus L_i$ for any $i \in E(\ell)$. 

Lastly, for any $i \notin I(\ell) \cup E(\ell)$, note that product $i$ does not appear in any split along the path from the root of the tree to leaf $\ell$. Therefore, for any $s$ with $v(s) = i$, it will follow that either $\leftleaves(s) \subseteq \leftleaves(s')$ for some $s'$ for which $\ell \in \rightleaves(s')$, or $\leftleaves(s) \subseteq \rightleaves(s')$ for some $s'$ for which $\ell \in \leftleaves(s')$ -- in other words, there is a split $s'$ such that every leaf in $\leftleaves(s)$ is to one side of $s'$ and $\ell$ is on the other side of $s'$. This means that $\ell'$ cannot be in $\leftleaves(s)$ for any $s$ with $v(s) = i$, or equivalently, $\ell \in \leaves \setminus L_i$ for any $i \notin I(\ell) \cup E(\ell)$. \\

\noindent \emph{Equation~\eqref{eq:CDC_IBS_single_leaf}, $\supseteq$ direction:} We will prove the contrapositive $\{ \ell' \in \leaves \mid \ell' \neq \ell \} \subseteq \bigcup_{i \in I(\ell)} R_i \cup \bigcup_{i \in E(\ell)} L_i \cup \bigcup_{ \substack{i \in \Ncal: \\ i \notin I(\ell) \cup E(\ell) }} L_i$. A straightforward result (see Lemma~EC.1 from \citeapp{misic2019optimization}) is that $\{ \ell' \in \leaves \mid \ell' \neq \ell \} = \bigcup_{s: \ell \in \leftleaves(s)} \rightleaves(s) \cup \bigcup_{s: \ell \in \rightleaves(s)} \leftleaves(s)$. 
Thus, if $\ell' \neq \ell$, then we have that $\ell' \in \rightleaves(s)$ for some $s$ such that $\ell \in \leftleaves(s)$, or $\ell' \in \leftleaves(s)$ for some $s$ such that $\ell \in \rightleaves(s)$. Let $i^* = v(s)$. In the first case, since $\ell \in \leftleaves(s)$, we have that $i^* \in I(\ell)$, and we thus have
\begin{align*}
	\rightleaves(s)  \subseteq \bigcup_{s': v(s') = i^*} \rightleaves(s') = R_{i^*} \subseteq \bigcup_{i \in I(\ell)} R_i \cup \bigcup_{i \in E(\ell)} L_i \cup \bigcup_{\substack{i' \in \Ncal: \\ i' \notin I(\ell) \cup E(\ell)}} L_i.
\end{align*}

In the second case, since $\ell \in \rightleaves(s)$, we have that $i^* \in E(\ell)$, and thus we have
\begin{align*}
	\leftleaves(s)  \subseteq \bigcup_{s': v(s') = i^*} \leftleaves(s')  = L_{i^*} \subseteq \bigcup_{i \in I(\ell)} R_i \cup \bigcup_{i \in E(\ell)} L_i \cup \bigcup_{\substack{i' \in \Ncal: \\ i' \notin I(\ell) \cup E(\ell)}} L_i.
\end{align*}

This establishes the validity of the pairwise independent branching scheme~\eqref{eq:CDC_IBS_leaves}. Thus, a valid formulation for problem~\eqref{prob:CDC_abstract} (see formulation (9) in \citeapp{vielma2010mixed}, formulation (14) in \citeapp{vielma2011modeling} and  formulation (13) in \citeapp{huchette2019combinatorial}) is
\begin{subequations}
	\begin{alignat}{2}
		& \underset{\xb, \yb}{\text{maximize}} & \quad & \sum_{\ell \in \leaves} r_{\ell} y_{\ell} \\
		& \text{subject to} & & \sum_{\ell \in \leaves} y_{\ell} = 1, \\
		& & & \sum_{\ell \in L_i} y_{\ell} \leq x_i, \quad \forall \ i \in \Ncal, \\
		& & & \sum_{\ell \in R_i} y_{\ell} \leq 1 - x_i, \quad \forall \ i \in \Ncal, \\
		& & & y_{\ell} \geq 0, \quad \forall \ \ell \in \leaves, \\
		& & & x_i \in \{0,1\}, \quad \forall\ i \in \Ncal.
	\end{alignat}%
	\label{prob:CDC_MIO}
\end{subequations}
Observe that, by the definition of $L_i$ and $R_i$, formulation~\eqref{prob:CDC_MIO} is identical to \ProductMIO when $|F| = 1$. By invoking Theorem~1 from \citeapp{vielma2010mixed} with appropriate modifications, we can assert that formulation~\eqref{prob:CDC_MIO} is integral. Therefore, when $|F| = 1$, formulation \ProductMIO is always integral. \Halmos

\vspace{1em}

\renewcommand{\bibfont}{\mysizeecom}
{\setlength{\bibsep}{3pt}
\bibliographystyleapp{plainnat}
\bibliographyapp{aodf_literature.bib}
}

\end{APPENDICES}
\end{document}